\documentclass[11pt]{article}
\usepackage[title]{appendix}
\usepackage{epsfig, amsmath, amssymb, amsthm, times, stackrel, bm, float}
\usepackage{color}
\usepackage{esint}

\usepackage{tikz}
\usetikzlibrary{trees}
\usepackage{tkz-euclide}

\newcommand{\field}[1]{\mathbb{#1}}

\newcommand{\N}{\field{N}}    
\newcommand{\Z}{\field{Z}}    
\newcommand{\R}{\field{R}}    
\newcommand{\fP}{\field{P}}     
\newcommand{\T}{\field{T}}      

\newcommand{\cA}{{\mathcal A}}  
\newcommand{\cB}{{\mathcal B}}  
\newcommand{\cC}{{\mathcal C}}  
\newcommand{\cF}{{\mathcal F}}  
\newcommand{\cG}{{\mathcal G}}  
\newcommand{\cH}{{\mathcal H}}  
\newcommand{\cN}{{\mathcal N}}  
\newcommand{\cO}{{\mathcal O}}  

\newcommand{\bM}{\mathbf{M}}
\newcommand{\bX}{\mathbf{X}}

\newcommand{\bu}{\mathbf{u}}
\newcommand{\bv}{\mathbf{v}}

\newcommand{\1}{\mathbf{1}}  

\def\be{\begin{equation}}
  \def\ee{\end{equation}}
\def\lbl{\label}
\def\p{\partial}

\newcommand{\Lip}{\operatorname{Lip}}

\DeclareMathOperator{\esssup}{ess\,sup}

\newtheorem{theorem}{Theorem}[section]

\newtheorem{lemma}[theorem]{Lemma}

\newtheorem{remark}[theorem]{Remark}
\newtheorem{definition}[theorem]{Definition}

\newtheorem{example}[theorem]{Example}

\title{Interacting dynamical systems on networks and fractals: discrete and continuous models,
   mean-field limit, and  convergence rates
}
\date{\today}
\author{
Georgi S. Medvedev\thanks{Department of Mathematics, 
Drexel University, 3141 Chestnut Street, Philadelphia, PA 19104,
{\tt medvedev@drexel.edu}}
}

\begin{document}
\maketitle
\begin{abstract}
  We develop a continuum limit and mean-field theory for interacting particle systems (IPS) on
  self-similar networks, a new class of discrete models whose large-scale behavior gives rise to nonlocal
  evolution equations on fractal domains. This work extends the graphon-based framework for IPS,
  used to derive continuum and mean-field limits in the non-exchangeable setting, to situations where 
  the spatial domain is fractal rather than Euclidean. The motivation arises from both physical models
  naturally formulated on fractals (e.g., heterogeneous media and fractal scattering)
  and real-world networks exhibiting hierarchical or quasi-self-similar structure.

  Our analysis relies on tools from fractal geometry, including Iterated Function Systems (IFS)
  and self-similar measures. A central result is an explicit isomorphism between
  self-similar IPS and graphon IPS, which allows us to justify the continuum and mean-field
  (Vlasov-type) limits in the self-similar setting. This connection reveals that macroscopic dynamics on
  fractal domains emerge naturally as limits of dynamics on appropriate discretizations of fractal sets.

  Another contribution of the paper is the derivation of optimal convergence rates for the discrete self-similar
  models. We introduce a scale of generalized Lipschitz spaces on fractals, extending the
  Nikolskii-Besov spaces used in the Euclidean setting, and obtain convergence estimates for discontinuous
  Galerkin approximations of nonlocal equations posed on fractal domains.
  These results apply to kernels with minimal regularity, covering models relevant in applications,
including discontinuous singular kernels arising in acoustic scattering on fractal screens.

  Finally, we interpret self-similar networks as Galerkin discretizations of nonlocal problems on fractal sets
  connecting our results to numerical analysis of PDEs on fractal domains. On the other hand, this work fits into the
  broader effort to extend the classical PDEs to fractal domains, which holds significant potential for both
  theoretical advances and numerical exploration.
 \end{abstract}

  \leftline{\small 2020 {\it Mathematical Subject Classification.\/}
          34C15, 
          28A80, 
          37M15, 
92B20.           
}
\noindent{\small{\it Keywords and phrases. \/}
  Interacting particle system, self-similar network, continuum limit, mean-field limit,
  rate of convergence, martingale, isomorphism, graphon, fractal, Sierpinski
  Gasket, Galerkin method, Kuramoto model
}  
  


\section{Introduction}
\setcounter{equation}{0}

By an interacting particle system (IPS), we understand a differential
equation model describing the evolution of a large ensemble of particles, whose 
dynamics is governed by the combination of intrinsic forces acting on individual particles
and pairwise interactions with other particles according to specified rules. Examples include gases
composed of many molecules moving under the laws of classical
mechanics and interacting through gravitational forces, plasmas where ions and electrons interact via
electrostatic potentials, as well as neuronal networks where neurons interact by means of synaptic currents,
to name a few \cite{Neu84,Gol16,Jab14}.

When the size of the system is large enough, the long time asymptotic behavior in such systems often
can be understood by considering a continuum (thermodynamic) limit as the size of the system goes to infinity.
If the system contains random initial data or random parameters, its macroscopic behavior can be captured by the 
Vlasov equation, a PDE describing the evolution of the distribution of particles in the phase space.
The derivation of the Vlasov equation in the \textit{exchangeable} setting when all particles are statistically
identical dates back to the works by Dobrushin and Neunzert
\cite{Dob79, Neu84, Gol16}. Modern applications of dynamical networks 
in biology and social sciences, such as neuronal networks, flocking and swarming models, opinion dynamics,
and consensus protocols,
as well as the models of technological systems like the Internet and power grids, feature spatially structured
interactions.
The pattern of interactions is determined by a graph or rather a sequence of graphs. Such systems are 
\textit{non-exchangeable}.
The derivation of the Vlasov equation in this case takes into account the
information about how these graphs behave as the system size grows.

The approach to deriving the continuum limit and the Vlasov equation
for non-exchangeable
IPS on convergent
graph sequences was developed in \cite{Med14a, Med14b, Med19, KVMed18}.
The key new ingredient in the derivation of the limiting equations for the IPS on networks is the
use of graphons, measurable functions that represent graphs and limit points of graph sequences
\cite{LovaszAMS}.
This graphon-based framework has since been successfully applied in a variety of contexts, including
coupled oscillator models \cite{ChiMed19a, ChiMed19b, CMM18, ChiMed22, MedMiz22, CMM23, Cop-oscill22},
models
of interacting diffusions \cite{Luc18, Cop22, BCN24}, optimal control \cite{GaoCai20}, mean-field games
\cite{CaiHua21, AurCar22}, neuronal networks \cite{JabDat23},
epidemic spread \cite{DDS24},
and image processing \cite{HFE20}
among others. In addition, related techniques have been developed for other types of graph limit
representations \cite{GJK22, GkoKue22, KueXu22}. In particular, \cite[Example~2.18(iii)]{KueXu22},
presents Vlasov equation on Cantor--like sets.

  In this paper, we develop a continuum-limit theory for IPS evolving on self-similar networks.
  Our motivation for studying dynamics on self-similar structures is as follows.
  First, many real-world networks exhibit self-similar or approximately self-similar organization,
  often arising from an underlying modular hierarchy. For example, the synaptic organization of
  the mammalian brain exhibits connectivity across multiple structural scales, ranging from
  individual neurons and local circuits to large-scale networks
  \cite{Harris, Tso, ChoOme2018}.
  Similarly, the vascular system of the yolk sac embryo displays a striking multiscale organization,
  in which branching patterns observed at larger scales are replicated at finer ones \cite{GarciaLarina2014}.
  Related hierarchical structures occur in the dendritic trees of pyramidal neurons \cite{Shepard}.
  Hierarchical organization is also prominent in important technological networks,
  including the Internet and artificial neural networks \cite{RavBar2003, YooJeoBar2002}.

  Reproducing the exact connectivity of such complex networks is generally neither
  feasible nor desirable from the standpoint of mathematical analysis. Instead,
  one seeks idealized models that retain salient structural features while remaining
  analytically tractable. For networks with hierarchical organization, self-similar
  graphs generated by fractals provide one such class of models.
  They offer a mathematically precise setting in which the effects of organization
  across scales on collective dynamics can be studied systematically.

  In this paper, we use tools from fractal geometry to construct self-similar networks
  and to develop an analytical framework for their dynamics. In particular,
  the continuum and mean-field limits established here open the way to
  studying collective phenomena, such as the transition to synchrony in
  the Kuramoto model, on self-similar graphs. Previous analyses of the
  Kuramoto model on dense graphs have revealed subtle connections
  between network topology and emergent dynamics \cite{CMM2018,
    CMM22a, CMM23}.
  The framework developed here provides a setting in which analogous questions
  can be investigated for networks with self-similar organization.

  In recent work with Matthew Mizuhara, we studiedd the relationship
  between self-similar
  network organization and
  the structure of attractors in the Kuramoto model with local coupling \cite{MM2026a,MM2026b}.
  In particular, we established a one-to-one correspondence between topologically distinct
  stable steady states and homotopy classes of continuous functions on fractals.
  The present work develops a complementary approach for models with nonlocal interactions:
  rather than focusing on the stability of stationary states, the emphasis here is on
  collective dynamics through
  continuum and mean-field descriptions of IPS on self-similar networks.

  A further motivation comes from a recurring feature of continuum limits of coupled networks:
the limiting problem often possesses greater regularity, and sometimes additional symmetries,
than its finite-dimensional discrete counterparts. This additional structure can be exploited
in the analysis of the continuum problem and subsequently transferred back to large finite
networks through perturbative arguments. For an example of this approach, we refer
the interested reader to the derivation of the normal form for a pitchfork bifurcation for the
discrete Swift-Hohenberg equation on W-random graphs in \cite[Section~5]{MedPel24}.
For self-similar networks, the continuum limit is
naturally defined on a fractal domain and therefore inherits symmetries associated
with the underlying Iterated Function System. This suggests that methods for dynamical
systems with symmetries (cf.~\cite{Bick2025}) may provide additional tools for the analysis of
the limiting equations and, ultimately, of large finite networks.
On a related note, techniques developed in the present paper have been used in recent work with Emmanuel
Tr{\'e}lat to obtain new insights into evolution equations on fractals \cite{MedTre2026}.

\begin{figure}
	\centering
 \textbf{a}\includegraphics[width =.22\textwidth]{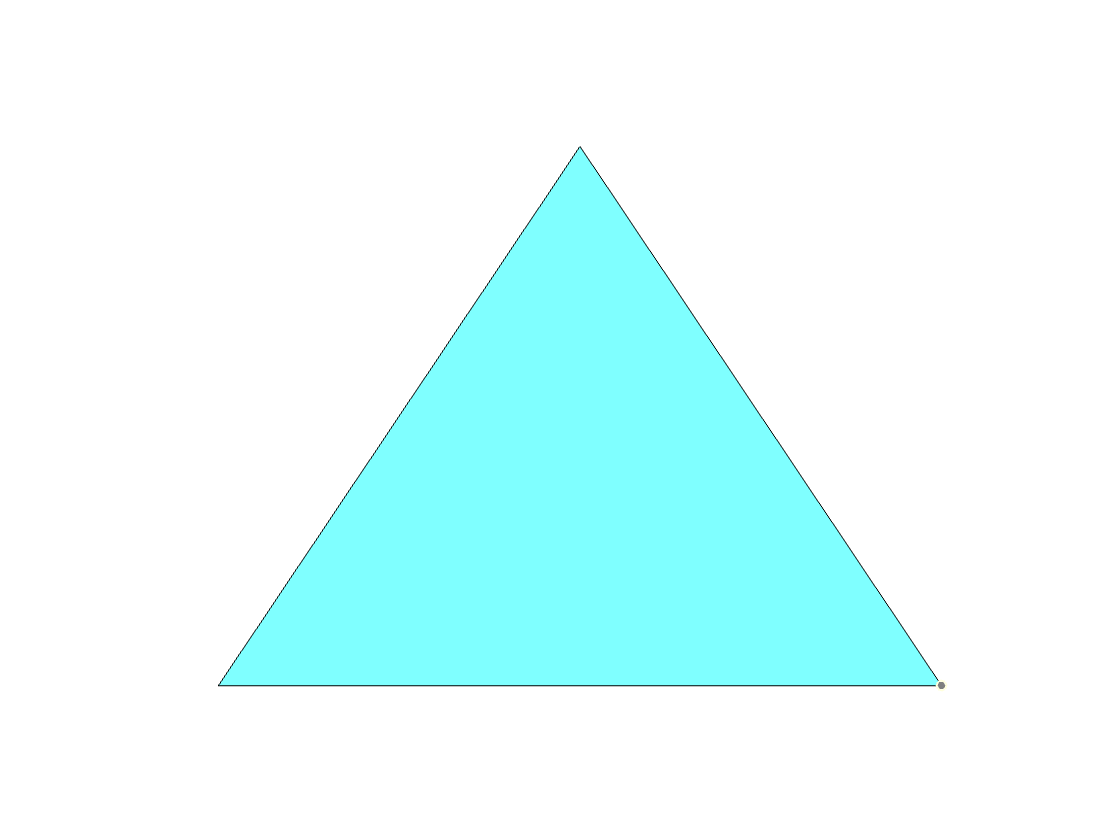}
 \textbf{b} \includegraphics[width = .22\textwidth]{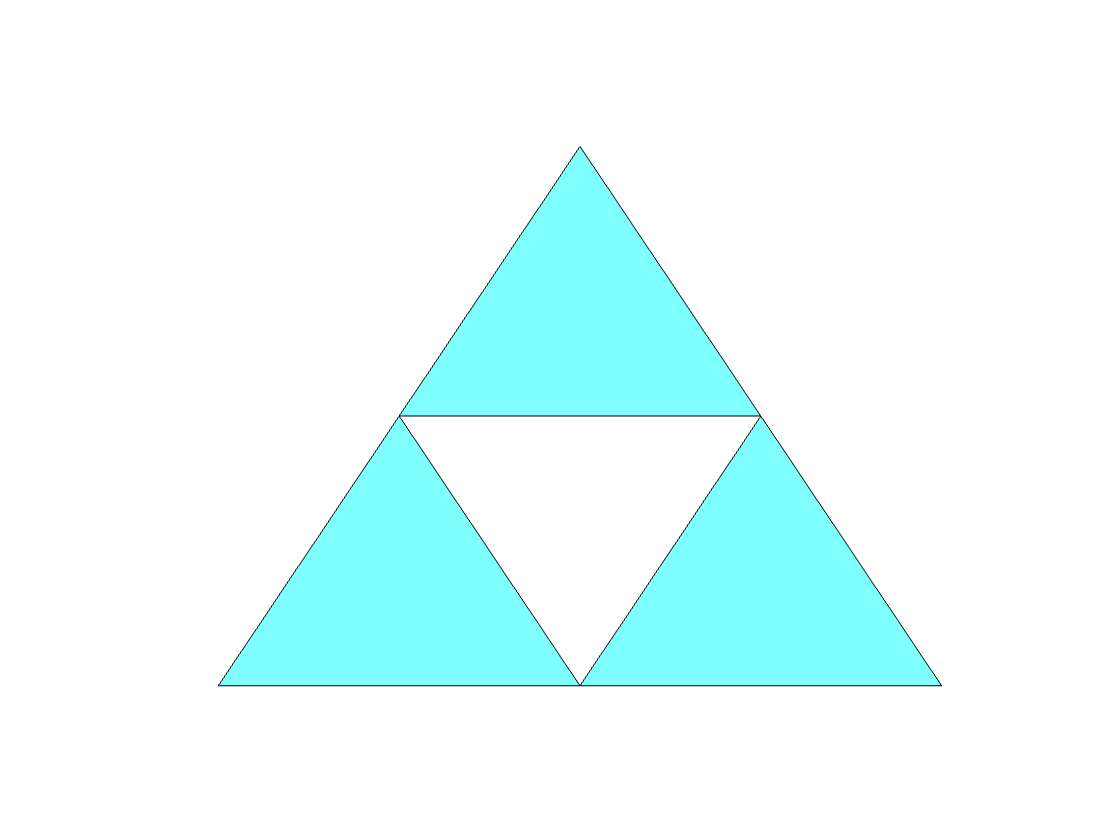}
 \textbf{c}\includegraphics[width =.22\textwidth]{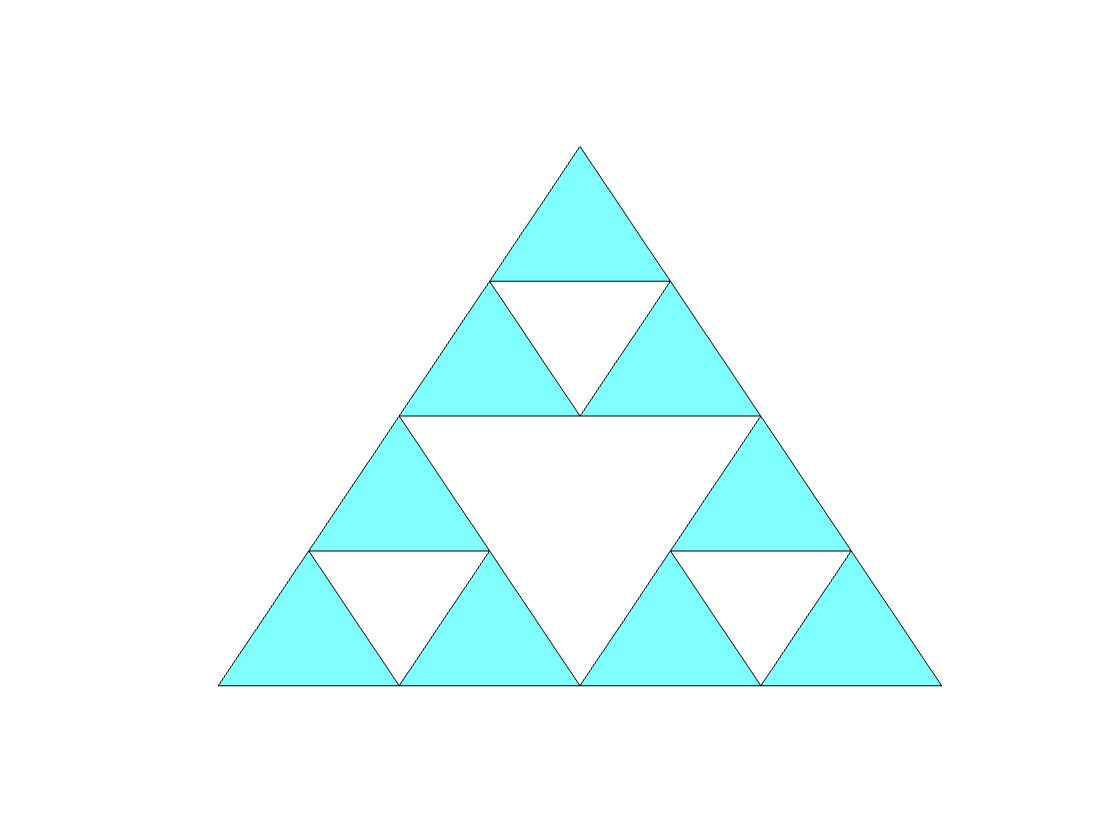}
  \textbf{d}\includegraphics[width = .22\textwidth]{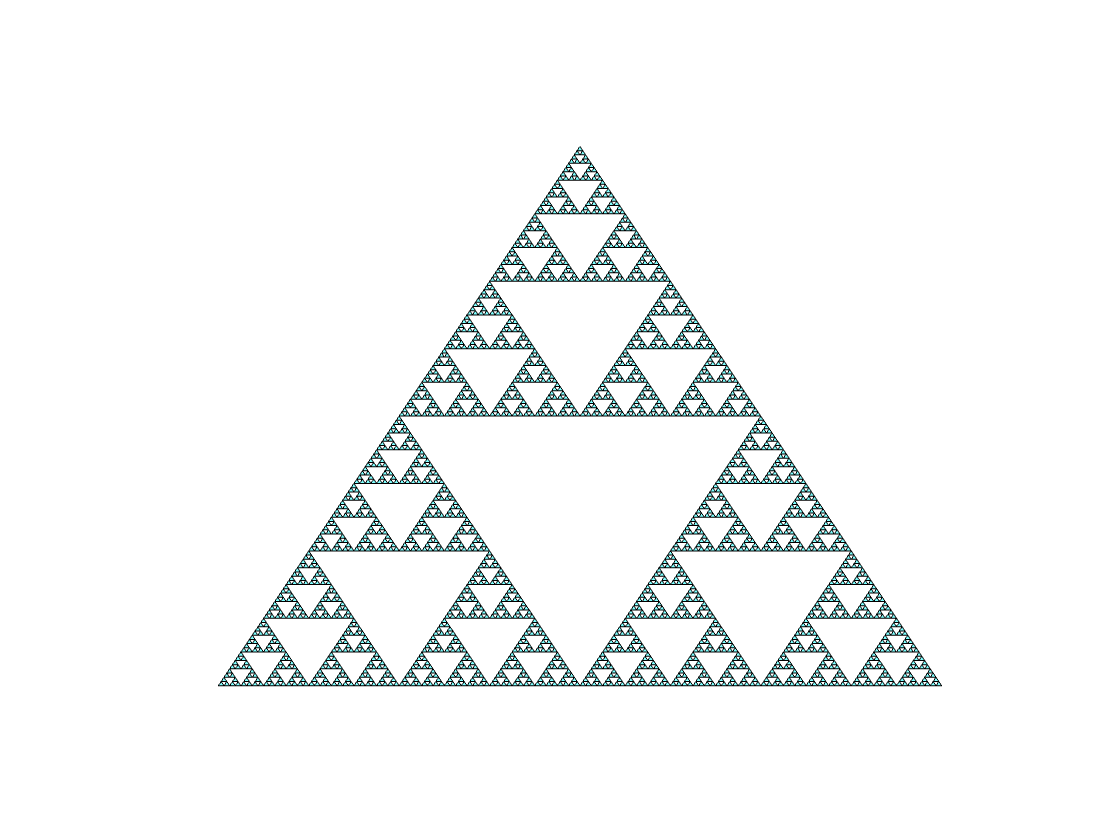}
  \caption{  Panels (\textbf{a}, \textbf{b}, \textbf{c}) illustrate the first three steps of the geometric
    construction of the Sierpinski gasket shown in (\textbf{d}).
  }\label{f.SG}
\end{figure}

Conversely, there are physical models that are readily expressed as
nonlocal equations on a fractal domain, and one might be interested in methods for numerical
integration of such models.
The discrete networks analyzed in this paper can be interpreted as
Galerkin approximations of the continuous models on fractal domains.
In this context, the results in Section~\ref{sec.rate} establish convergence of the
discontinuous Galerkin method.

To study the speed of convergence, we introduce a scale of normed spaces of integrable functions on fractals,
which imitate generalized Lipschitz spaces \cite[Chapter~2, \S~9]{DeVore-book}, also
known as Nikolskii-Besov spaces \cite{BKP2019}.
A related scale of Besov-type spaces was used in \cite{JOP2017}.
Our rate of convergence estimates extend the results for the discontinuous Galerkin
method for problems on Euclidean domains analyzed in \cite{KVMed22} to the fractal setting.
The results in Section~\ref{sec.rate} are optimal in the sense that they reflect the generalized Lipschitz
regularity of the graph limit given by the integrable kernel\footnote{Throughout this paper, the term \emph{kernel} refers to a function
defining an integral operator. In graph limit theory, the term is sometimes
used more specifically for a bounded measurable symmetric function
$W:[0,1]\times [0,1]\to\mathbb{R}$, while the term \emph{graphon} is reserved for
such functions taking values in $[0,1]$ \cite{LovGraphLim12}. We do not make this distinction
here. Instead, we use \emph{graphon} in a broader sense for a measurable function on
the unit square representing a graph or a limit of a graph sequence.}
$W$.

A prototypical example of a nonlocal model posed on a fractal domain is the acoustic
scattering problem for a fractal screen, which reduces to the integral equation
(cf.~\cite{Hew2025})
\begin{equation}\label{scatter}
\int_\Gamma \Phi(x,y)\psi(y)\,d\cH^d(y)=g(x),\qquad x\in\Gamma.
\end{equation}
Here, $\cH^d$ denotes the $d$-dimensional Hausdorff measure on the $d$-set $\Gamma$,
$\Phi$ is the fundamental solution of the Helmholtz equation, and the right-hand side $g$
is selected from an admissible class ensuring the solvability of \eqref{scatter}
(see \cite{Hew2025} for details). 

A central component in determining the convergence rate of the discontinuous Galerkin method for
\eqref{scatter} is the estimation of the error of approximating a solution by its projection
onto a given finite-dimensional space of piecewise constant functions \cite{Hew2025, CaeChW24}.
Section~\ref{sec.rate} develops such estimates for a broad class of self-similar domains,
providing the foundation for Galerkin schemes on fractal domains \cite{KVZ1972}.
For the evolution models analyzed in this paper, we provide optimal
rates of convergence in terms of the regularity of the kernel and
initial data. The approximation estimates underlying our analysis are not specific to these
models and may therefore be useful in the convergence analysis of Galerkin schemes for other
equations on self-similar domains. Among existing applications of Galerkin schemes
to PDEs on fractal domains, a related approach was used in \cite{Hew2025}.
The convergence analysis in that work is carried out under the assumption that the solution
possesses sufficient regularity. Our results do not apply directly to the discretization of the
model considered in \cite{Hew2025}, since its solution is a priori known to belong only to
$H^{-1}$. They would, however, become relevant if $L^2$-regularity could be established.

In addition to the applications outlined above, there is another compelling reason to study
self-similar networks. We show that dynamical models on self-similar graphs lead to equations for
macroscopic dynamics on fractal domains. This, in turn,
emphasizes a different suite of tools needed to characterize the convergence of the discrete models
to the continuum limit.
While the analysis of graphon dynamical systems relies on the spaces of graphons and their
properties
\cite{Med14a, Med14b, Med19, KVMed18, DupMed22}, the analysis of self-similar networks
requires tools from fractal geometry, including
iterated function systems (IFS),
symbolic (shift) spaces,
self-similar measures, and fractal dimension, to name a few
\cite{Hut81, Falc-Tech, BSS-self-similar, BP-Fractals}.
In this respect, the continuum limit
theory for IPS on self-similar networks
extends and complements the theory for graphon dynamical systems. It gives rise to a new class of
models that fits naturally into the broader effort to extend the classical theory of PDEs to fractal domains
(see, e.g., \cite{Kig01, Str06, HinMei20,FalcHu2001,Falc1999}),
which holds significant potential for both theoretical advances and numerical exploration \cite{Mosco13}.

At the same time, the present work aligns with a growing trend toward
exploring interacting dynamical systems on discrete structures beyond graphs and their limits.
For instance, models on hypergraphs, metric graphs, simplicial complexes, and discretizations
of Riemannian
manifolds have gained attention due to their potential applications in data science and image
processing \cite{NijDeV2022, ABD2023, BotPor2025}. In a similar spirit, the present work
introduces a new class of networks leading to
nonlocal equations on fractal domains.
\begin{figure}
\centering
\includegraphics[width=.3\textwidth]{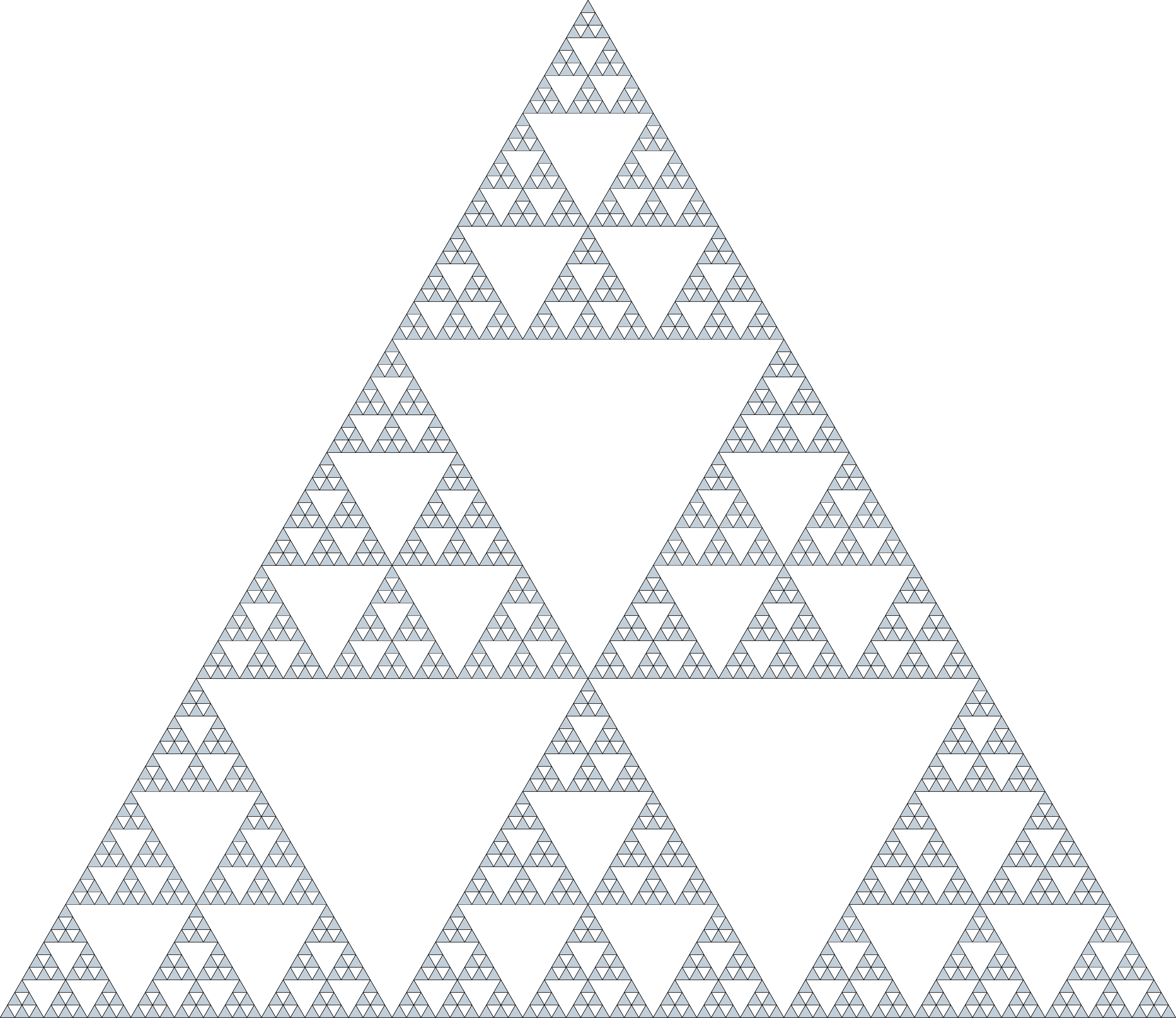}
\includegraphics[width=.3\textwidth]{pentagasket.pdf}
\hspace{4mm}
\includegraphics[width=.25\textwidth]{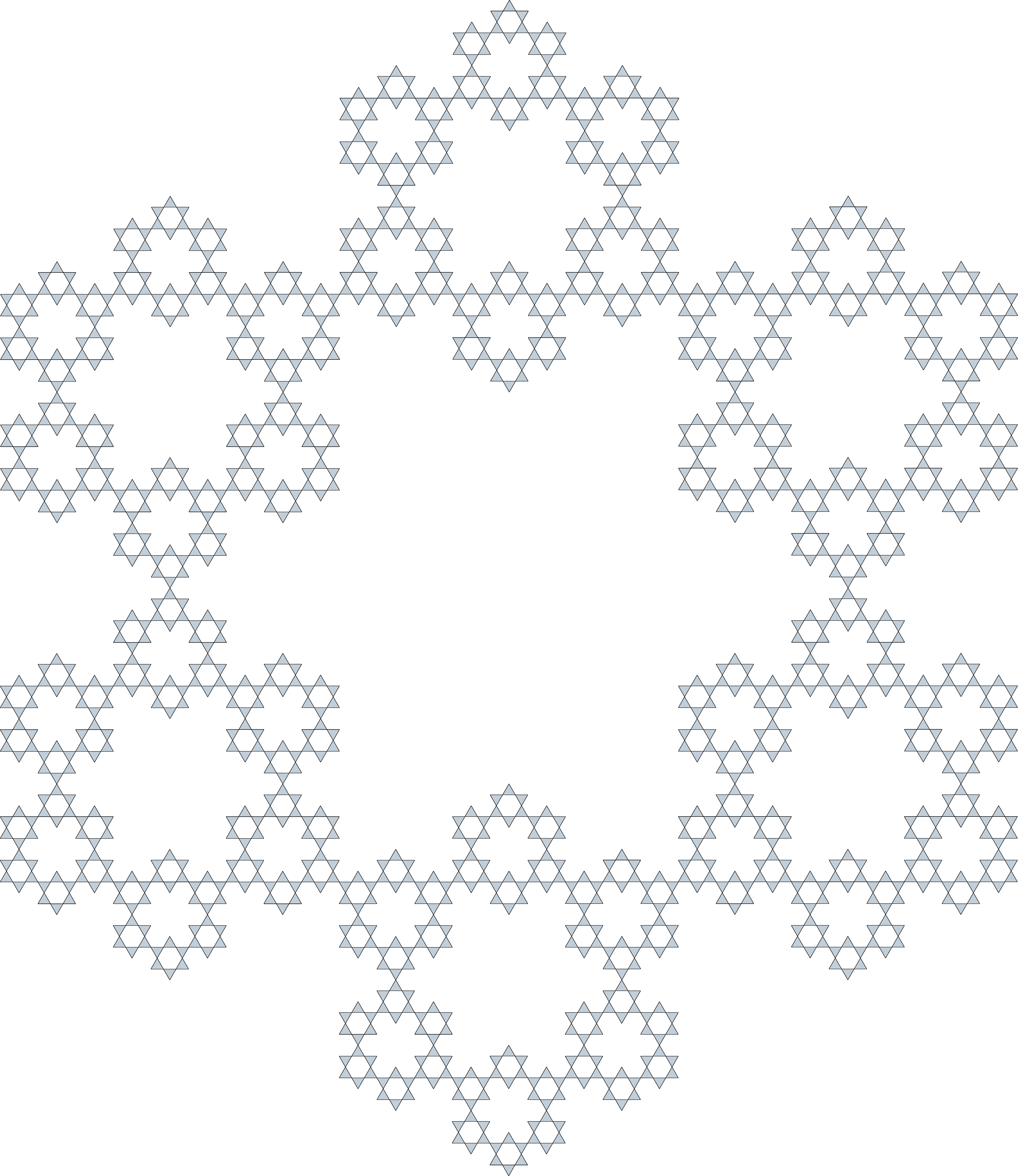}
\caption{Examples of attractors of IFS: the three-level Sierpinski gasket, the pentagasket, and
  the hexagasket (see \cite{Str06} for more details). }
\label{f.2}
\end{figure}

In the next section, we recall the principal results on graphon IPS. These include the continuum limit for IPS on graphs,
formulated as a nonlocal nonlinear heat equation and studied in \cite{Med14a, Med14b, Med19},
as well as the mean-field (Vlasov) limit, which was rigorously justified in \cite{KVMed18}.
We then summarize the convergence rate analysis for discrete-to-continuum approximations
developed in \cite{KVMed22}.

The main objective of this paper is to extend these results to IPS defined on self-similar graphs.
To prepare for this,
we first review the analytical tools from fractal geometry that will be essential in the
self-similar
setting. As a model of a self-similar network, we choose graphs approximating self-similar sets,
which arise as attractors of IFS. The topology of the attractors of an
IFS is described by establishing a correspondence with the space of symbolic sequences. The
latter is also used to define a measure on the attractor of the IFS. Both ingredients are essential
for formulating discrete models on self-similar networks. They are also used in the analysis
of the continuum limit and the mean-field limit of the discrete network. In
Section~\ref{sec.self-sim}
we review the IFS and self-similar measures. We also discuss the separation conditions that are
used to classify self-similar sets.

After that  in Section~\ref{sec.self-IPS}, we introduce a class of self-similar IPSs and state
our main results concerning their continuum and mean-field limits. Specifically, we
establish the wellposedness of the continuum limit in the form of the nonlocal evolution
equation on a self-similar domain, write down the Vlasov equation describing the mean-field
limit for the IPS at hand and formulate two theorems: one - on the convergence of solutions
of discrete models to the continuum limit and the second one - on the convergence of empirical
measures to the solution of the initial value problem for the Vlasov equation.

To prove convergence to the continuum limit and to justify the mean-field limit, we use the 
isomorphism between self-similar sets and the unit interval as well as between integrable functions
on self-similar sets and those on the unit intervals established in Section~\ref{sec.isomorph}.
We use these relationships to show the correspondence between solutions for self-similar IPS
and their graphon counterparts and prove the theorems from the previous section.

We then analyze the convergence rate of self-similar IPS toward their continuum limit in
Section~\ref{sec.rate}. The main problem here is to find the right way to describe the regularity
of the data in particular the regularity of the kernel of the nonlocal operator, which controls the
rate of convergence of the $L^2$-projections onto a subspace of piecewise constant functions.
In the Euclidean setting, it was shown in \cite{KVMed22} that the generalized Lipschitz
spaces provide the desired formalism. The definition of these spaces is based on the
$L^p$-modulus of continuity. In Section~\ref{sec.rate}, we find a way to extend the modulus
of continuity to the fractal setting by taking into account the geometry of the fractal domain
(see Definition~\ref{df.modulus}). Building on this definition, we introduce generalized
Lipschitz spaces for functions on self-similar domains and estimate the rate of convergence of the
$L^2$-projections in terms of these spaces. This yields the rate of convergence of the discrete
models to the continuum limit, or equivalently for the piecewise constant Galerkin method
for nonlocal equations on self-similar domains.
Section~\ref{sec.example} presents a numerical example illustrating the rate of convergence estimates
from the previous section.
The implementation of the Galerkin method
requires numerical evaluation of integrals for functions on self-similar domains. We review
two methods that can be used to that effect and formulate the corresponding algorithms
in the Appendix.
We conclude with a brief discussion in Section~\ref{sec.discuss}.

\section{Graphon interacting particle systems}\label{sec.GIPS}
\setcounter{equation}{0}

\subsection{Models and examples}
Consider the following IPS
\begin{equation}\label{KM}
  \dot u_{n,i} = f(t,u_{n,i}) + n^{-1} \sum_{j=1}^n w_{n,ij} D(u_{n,i}, u_{n,j}),\quad i\in [n]\doteq\{1,2,\dots, n\}
\end{equation}
where $u_{n,i}:\R^+\to \R^k$ stands for the state of particle $i$. Function $f$ defines intrinsic dynamics
of particle $i$. The sum on the right--hand side of
\eqref{KM} models the interactions between particles. $D$ is the interaction function.

Throughout this paper, \( D \) is a bounded Lipschitz continuous function
and \( f \) is jointly continuous and Lipschitz continuous in \( u \).
These conditions will be implicitly assumed in all statements below.

Weights
\begin{equation}\label{weights}
  w_{n,ij}=n^2 \int_{Q_{n,i}}\int_{Q_{n,j}} W(x,y) dxdy,\qquad Q_{n,i}=\left[ \frac{i-1}{n}, \frac{i}{n} \right),
  \end{equation}
are derived from a given graphon $0\le W\le 1$, which determines the limiting connectivity
of the network.

Along with \eqref{KM}, we consider the following  
IPS on random graphs:
\begin{equation}\label{W-KM}
  \dot u_{n,i} = f(t,u_{n,i}) + n^{-1} \sum_{j=1}^n \xi_{n,ij} D(u_{n,i}, u_{n,j}),\quad i\in [n].
\end{equation}
Here, we replace the deterministic weights $w_{n,ij}$ with Bernoulli random variables:
\begin{equation}\label{Bernoulli}
\fP(\xi_{n,ij}=1)=w_{n,ij}\quad \mbox{and}\quad \fP(\xi_{n,ij}=0)=1-w_{n,ij}.
\end{equation}

For the random network model, we assume that the graphon $W$ is nonnegative and
bounded by $1$. We also assume that $\xi_{n,ij}$ are independent for $i\neq j$.
For undirected graphs, one can first restrict to $1\le i< j\le n$ and then extend
to the remaining $i$ and $j$ by symmetry: $\xi_{n,ji}=\xi_{n,ij}$. The deterministic and
continuum systems \eqref{KM} and \eqref{W-KM} share the same continuum
limit \cite{Med14a, Med19}.

\begin{remark}\label{rem.negative}
The sign of $\xi_{n,ij}$ in \eqref{W-KM} determines the type of
interactions. Depending on the modeling formalism, negative sign  may correspond to
repulsive vs attractive
coupling in coupled oscillator models \cite{CMM2018},
or antiferromagnetic vs ferromagnetic coupling in
spin models \cite{AleMed2026}.
One can extend \eqref{W-KM} to allow for interactions of both types
(cf.~\cite[Equation (2.17)]{Med19}). To this end, let $W=W^+-W^-$, where nonnegative
$W^+$ and $W^-$ are the positive and negative parts of $W$
respectively. Then one can define weights separately for $W^+$ and $W^-$:
\begin{equation}\label{weights+}
  w^\pm_{n,ij}=n^2 \int_{Q_{n,i}}\int_{Q_{n,j}} W^\pm(x,y) dxdy.
\end{equation}
and 
\begin{equation}\label{Bernoulli+}
\fP(\xi^\pm_{n,ij}=1)=w^\pm_{n,ij}\quad \mbox{and}\quad \fP(\xi^\pm_{n,ij}=0)=1-w^\pm_{n,ij}.
\end{equation}
Finally, replace the  coupling term in \eqref{W-KM} by two sums as follows
\begin{equation}\label{W-KM+}
  \dot u_{n,i} = f(t,u_{n,i}) + n^{-1} \left\{
    \sum_{j=1}^n \xi^+_{n,ij} D(u_{n,i}, u_{n,j})
- \sum_{j=1}^n \xi^-_{n,ij} D(u_{n,i}, u_{n,j})\right\},\quad i\in [n].
\end{equation}  
The continuum limit of \eqref{W-KM+} is derived exactly in the same way as
that for \eqref{W-KM}, and,
therefore, we focus on the latter model.
\end{remark}  
\begin{figure} 
	\centering
\textbf{a}\;\includegraphics[width=.29\textwidth]{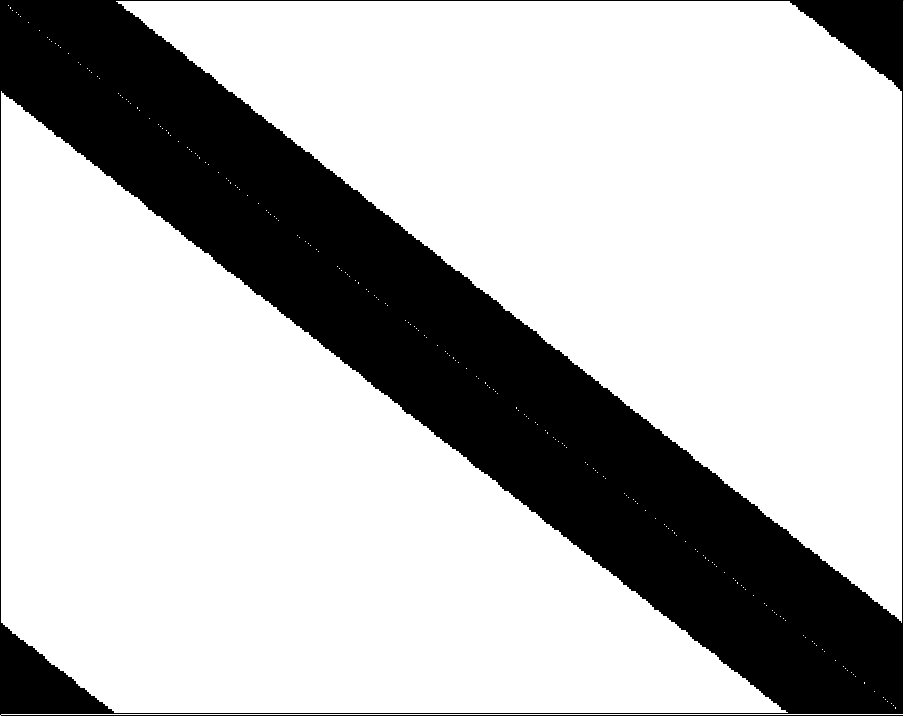}
\textbf{b}\;\includegraphics[width=.29\textwidth]{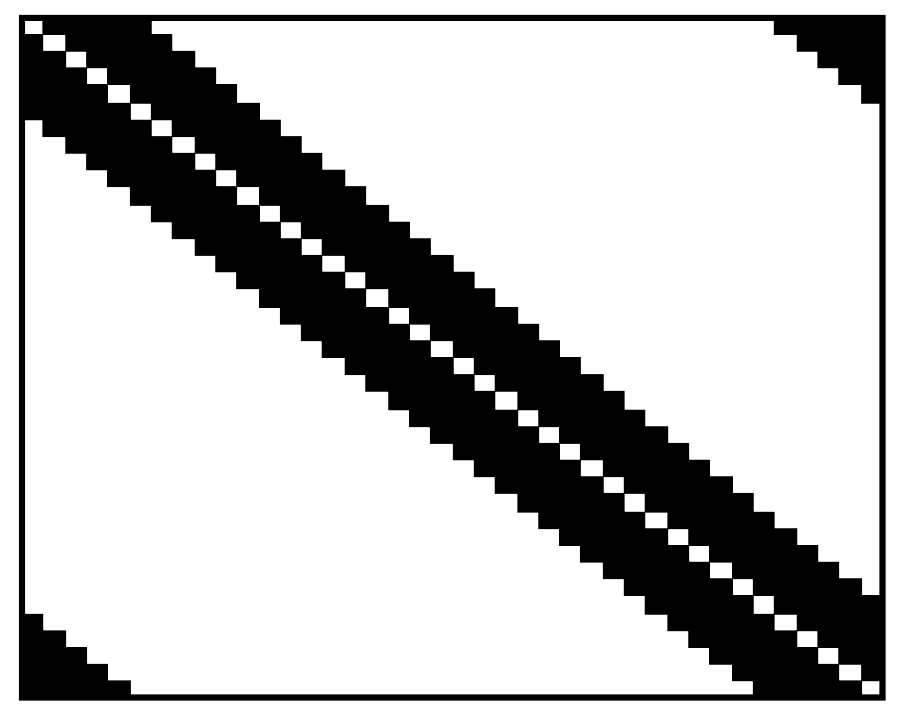}
 \textbf{c}\;\includegraphics[width=.29\textwidth]{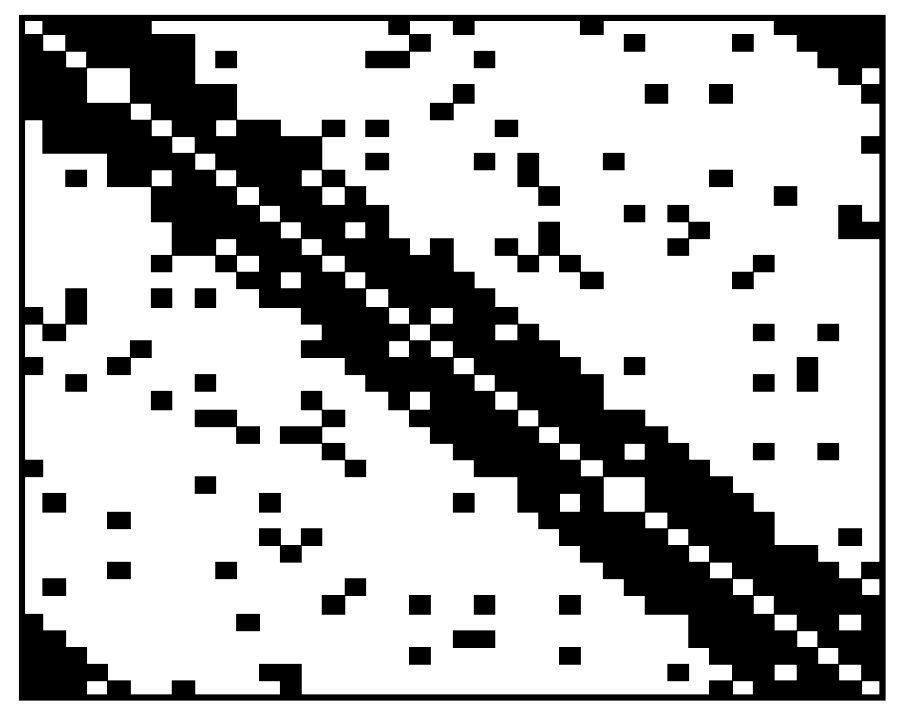}
 \caption{\textbf{a}~$W$ takes values $1-p$ and $p$ over the black and
   white regions respectively. \textbf{b, c}~Pixel plots  of the deterministic weighted network
   \eqref{weights} and random graph
                \eqref{Bernoulli}. }
	\label{f.band}
\end{figure}

\begin{example}~(cf.~\cite{Med14c})
  Let $0\le p,r\le 1/2$ and define
 \begin{equation}\label{band}
   W(x,y) =\left\{
                  \begin{array}{ll}
                     1-p, & \min\{ |x-y|, 1-|x-y|\} \le r,\\
                    p, & \mbox{otherwise}.
                  \end{array}   \right.               
                \end{equation}
                The band structure of $W$ is shown in Figure~\ref{f.band}\textbf{a}. Plots \textbf{b} and \textbf{c}
                illustrate the structure of the deterministic weighted network \eqref{weights} and random graph
                \eqref{Bernoulli}. The latter is called \textit{a small-world graph}.                
  \end{example}

  Additional parameters can be easily incorporated into \eqref{KM} or \eqref{W-KM}.
  For the former model, this can be done as follows:
\begin{align}\label{KM-aug}
  \dot u_{n,i} &= f(t, u_{n,i},\lambda_i) + n^{-1} \sum_{j=1}^n w_{n,ij} D(u_{n,i}, u_{n,j}),\\
  \label{KM-aug-para}
    \dot\lambda_i & =0, \quad \lambda_i\in \R^p,\; i\in [n].
\end{align}

Equation~\eqref{KM} covers many interesting models in nonlinear science. We discuss a few representative
  examples below. 
\begin{example}\label{ex.KM}
  We start with the Kuramoto model of coupled phase oscillators (cf.~\cite{Kur84, ChiMed19a, MedMiz22}):
  \begin{align}
    \label{example-KM}
    \dot u_i & =\omega_i +\frac{K}{n} \sum_{j=1}^n a_{ij}\sin\left(2\pi(u_j-u_i)\right),\\
    \nonumber
    \dot \omega_i &=0, \quad i\in [n],
  \end{align}
  Here, $u_i:\R^+\to\T\doteq \R/\Z$ is the phase of oscillator $i\in[n]$, $K\in\R$ is the coupling
  strength.  The adjacency matrix $(a_{ij})$ defines the
  connectivity of the network. The intrinsic frequencies $\omega_i$, which may be random,  are
  assigned through the initial condition.

  The Kuramoto model plays an important role  in the theory of synchronization. 
  With $\omega_i=0$ Equation \eqref{example-KM}  has been used to study relaxation dynamics of the $XY$-model
  in statistical physics \cite{CosSha21}.
\end{example}
\begin{example}\label{ex.KM-inertia}
  If inertia and damping are included into the Kuramoto model, we obtain the following model
  \begin{align*}
    \dot u_i & = v_i\\
    \dot v_i &= -\gamma v_i+\omega_i +\frac{K}{n} \sum_{j=1}^n a_{ij}\sin\left(2\pi(u_j-u_i)\right),\\
    \dot \omega_i &=0, \quad i\in [n].
  \end{align*}
  A variant of this model appears in statistical physics under the name of Hamiltonian mean-field model
  \cite{VRR99}.
  
  Although the phase space of each oscillator is two-dimensional (three-dimensional for the augmented
  model), the continuum and the mean-field limits for \eqref{ex.KM-inertia} are derived exactly
  in the same way as for the scalar Kuramoto model in \eqref{ex.KM} (cf.~\cite{ChiMed22, CMM23}).

    For the derivation of the continuum limit for a closely related model, we refer the interested reader to \cite{KueThr2019}.
    The derivation of the mean--field limit for the \textit{scalar} Kuramoto model relies on the wellposedness
    and continuous dependence on parameters of the equation of characteristics (see \cite[\S~2.2]{KVMed18}).
    For the second--order model, after rewriting the equation of characteristic for the augmented model as
    the first--order ODE in the appropriate multidimensional phase space, wellposedness and continuous
    dependence of solutions follow from the arguments in \cite{KVMed18} adapted to the phase space at hand.
\end{example}

\begin{example}\label{ex.consensus}
  The following model appears in the literature in various contexts, including
  opinion dynamics and consensus protocols
  \begin{equation*}
    \dot u_i  = \frac{1}{n} \sum_{j=1}^n a_{ij} D\left(u_j-u_i\right),\quad i\in [n],
  \end{equation*}
  where $u_i\in\R$ quantifies the tendency of agent $i$ to support a certain opinion and
  $D$ is an interaction function, which could be the identity map or have a more
elaborate form, as, for example, in the Krauss--Hegselmann bounded
confidence model \cite{HegKra2005}.
\end{example}

Other examples include the Cucker–Smale model \cite{CucSma2007}, power networks \cite{DorBul12},
and the classical $N$-body problem \cite{Neu78}, to name a few. For all such models,
when it comes to describing the continuum limit of coupled networks with structured
interactions, \eqref{KM} serves as an analytically convenient prototypical system.

\subsection{Continuum limit}
Pathwise analysis of \eqref{KM} is often not feasible
for large $n$, leading one to seek a macroscopic description of the collective dynamics of
\eqref{KM}. Below we explain three types of results concerning the description
of discrete systems \eqref{KM} and \eqref{W-KM} in the limit as $n\to\infty$.

We will start our discussion of the continuum limit of 
\eqref{KM} and \eqref{W-KM} with the following nonlinear heat equation
\begin{equation}\label{cKM}
\p_t u(t,x) = f(t,u) +\int_{Q} W(x,y) D\left(u(t,x),u(t,y)\right) dy,\quad x\in Q.
\end{equation}
Here and below, $Q=[0,1]$ denotes the unit interval used as the label space. The graphon
$W(x,y)$ describes the limiting connectivity of the network.

\begin{remark}\label{rem.label}
We have replaced a large finite-size IPS with a continuum of particles
labeled by $x\in [0,1]$, i.e., the label space is the unit interval equipped with Lebesgue measure.
In fact, any Borel probability space can serve as a label space. In practice, it 
suffices to use the unit interval equipped with Lebesgue measure, since any standard nonatomic
probability space is isomorphic to it \cite[Theorem~9.4.7]{Bogachev2}. In the analysis of self-similar IPS below,
we will be led to use a fractal set with the appropriate self-similar measure as the label space.
\end{remark}

The connection between the continuum limit \eqref{cKM} and the discrete models \eqref{KM}
and \eqref{W-KM}
is explained in the following theorem.

\begin{theorem}\label{thm.heat}\cite{Med14a, Med14b}
  Let $u(t,x)$ be the solution of the IVP for \eqref{cKM} subject to $u(0,\cdot)=g\in L^2(Q)$.
  Likewise, let
  \begin{align*}
    u^n(t,x) &= \sum_{i=1}^n u_{n,i}(t) \1_{Q_{n,i}}(x),\\
    \bar u^n(t,x) &= \sum_{i=1}^n \bar u_{n,i}(t) \1_{Q_{n,i}}(x)
\end{align*}
solve the IVPs for \eqref{KM} and \eqref{W-KM}, respectively, and satisfy 
$$
u^n(0,\cdot)=\bar u^n(0,\cdot)= g^n.
$$
Then for $v\in\{ u^n, \bar u^n\}$
\begin{equation}\label{approx-heat}
  \|u-v\|_{C(0,T; L^2(Q))} \le C \left( \|g-g^n\|_{L^2(Q)} +
    \|W-W^n\|_{L^2(Q\times Q)}\right),
\end{equation}
where 
$$
\|w\|_{C(0,T; L^2(Q))}\doteq\sup_{t\in [0,T]} \| w(t,\cdot)\|_{L^2(Q)}
$$
is  the norm in $C(0,T; L^2(Q))$.
Here, in the case of $v=\bar u^n$ estimate \eqref{approx-heat} holds almost surely with respect to the
probability measure used in the construction of random graphs \eqref{Bernoulli}.

\end{theorem}
\begin{remark}\label{rem.LLN}
  In the case of the model on a random graph \eqref{W-KM}, Theorem~\ref{thm.heat} yields the Strong Law of Large
  Numbers for the solutions of the IVP across different realizations of the random graph. The Large Deviation Principle
  is also available in this setting \cite{DupMed22}.
  \end{remark}

Equation \eqref{cKM} was derived and rigorously justified as the continuum limit of \eqref{KM}
on dense graphs in \cite{Med14a}. It was later extended to particle systems on
random graphs, including sparse random graphs, in \cite{Med14b, KVMed17, Med19, DupMed22}.
In the case of sparse networks, $W$ must be allowed to be unbounded, for example, 
$W\in L^1([0,1]\times [0,1])$.

The estimates developed for the justification of the continuum limit \eqref{cKM}
are also important for the justification of the mean-field limit in the form of the Vlasov equation
\eqref{Vlasov}, which we will discuss next. Specifically, they are used
in the derivation of the Liouville equation, which serves as an important step in the derivation of
the mean-field limit \cite{KVMed18}.

Representative applications of the continuum limit \eqref{cKM} include stability and bifurcations 
of steady states in particle models on (random) graphs
(cf.~\cite{Med14c, MedTan15a, MedTan18, MMP22, MedPel24}).
The applications of the mean-field limit for \eqref{KM} include synchronization
and pattern formation in coupled dynamical systems on graphs
(cf.~\cite{ChiMed19a, ChiMed19b, ChiMed22, CMM18, CMM22a, CMM23}). 

\subsection{Mean-field limit}
Suppose that the discrete models \eqref{KM} and \eqref{W-KM} are subject to random initial
conditions or they contain random parameters. Then the macroscopic behavior of the
system in the limit as $n\to\infty$ is described by the Vlasov equation (cf.~\cite{ChiMed19a})
\begin{align}\label{Vlasov}
 & \partial_t\rho(t,u,x)+\partial_u\left\{ V(t,u,x)\rho(t,u,x) \right\}=0,\\
  \label{Vlasov-VF}
  & V(t,u,x)= f(t,u)+  \int_{[0,1]}\int_{\R^k} W(x,y) D(u,v) \rho(t,v,y) dvdy.
\end{align}
Here, $\rho(t,\cdot, x)$ represents the probability density of the state of particle $x\in [0,1]$
at time $t$. As before, we have replaced a large finite-size IPS with a continuum of interacting
particles labeled by $x\in [0,1]$.


If $W\equiv\operatorname{const}$, meaning that every particle interacts with all other particles
in the same way, then all particles share the same distribution, and the label $x\in [0,1]$ is
not needed for describing their distribution in the continuum limit. In this case, \eqref{KM}
is referred to as an \textit{exchangeable} system. The mathematical justification of the Vlasov
equation in this setting follows from the classical works of Dobrushin \cite{Dob79} and
Neunzert \cite{Neu78, Neu84}. For structured networks, 
the justification of the Vlasov equation
\eqref{Vlasov} was provided in \cite{KVMed18} by extending Neunzert’s method \cite{Neu78}
to the \textit{non-exchangeable} setting (see also \cite{ChiMed19a} for a different approach).
Although the analysis in \cite{KVMed18} was originally motivated
by the Kuramoto model of phase oscillators on graphs, the key ideas developed to
address non-exchangeability extend to related models on graphs and other structures
(see \cite{GkoKue22,KueXu22,KueXu25}). Moreover, the local empirical measures and graphons
that were introduced for the analysis of the Vlasov equation in \cite{KVMed18} played a key
role in extending the McKean-Vlasov framework to the non-exchangeable setting and in its
subsequent applications to interacting diffusions
(see, e.g., \cite{Luc18, CDG20,Cop22,Cop-oscill22,BCN24}).

To describe the precise relation between the Vlasov equation and the discrete models \eqref{KM} and
\eqref{W-KM},
let $n=m\ell, \; m,\ell\in\Z,$ and define the local empirical measure as follows
\begin{equation}\label{loc-emp}
  \mathfrak{m}^x_{m,\ell,t}(A)=\frac{1}{\ell} \sum_{j=1}^\ell \1_A \left(u_{n, (i-1)\ell+j}(t)\right), \quad
         A\in\mathcal{B}(\R^k), \; x\in Q_{m,i},\; i\in [m],
       \end{equation}
       where $\mathcal{B}(\R^k)$ denotes the Borel $\sigma$-algebra on $\R^k$.
       Here, we first divide $[0,1]$ into $m$ cells $Q_{m,i}, \; i \in [m]$. Our goal is to construct
       a piecewise constant  (over each $Q_{m,i}$, $i\in [m]$) approximation
of the probability distribution of the location particle $x\in Q$ in the phase space at time $t$,
       $\mathfrak{m}^x_{m,\ell,t}$. To this end, the number of cells $m$ is chosen sufficiently large so that  
       $W^m$ is sufficiently close to $W$. Further, each $Q_{m,i}$ is subdivided into $\ell$ parts.
       Each part corresponds to  
       a particle in the discrete model \eqref{KM} with $n=m\ell$. The number of parts $\ell$ must be large enough
       to achieve a good approximation of the empirical measure $\mathfrak{m}^x_{m,\ell,t}$ for $x \in Q_{m,i}$,
       for every (fixed) $i \in [m]$. Roughly speaking,  
       particles within each cell $Q_{m,i}$ are treated as identically distributed, because $W$ is
       approximately
       constant over
       each  $Q_{m,i} \times Q_{m,j}$. We refer the interested reader to Section~3 in \cite{KVMed18} for
       precise statements and further details.

       On the other hand, from the solution of the IVP for the Vlasov equation \eqref{Vlasov} with a given initial
       condition, we compute
  \begin{equation}\label{Vlasov-mes}
    \mathfrak{m}^x_t(A)=\int_A \rho(t,x,u) du, \quad A\in\mathcal{B}(\R^k).
  \end{equation}
  Note that the local empirical measures $\mathfrak{m}^x_{m,\ell,t}$ are computed from the solutions of
  the discrete system~\eqref{KM}, whereas the two-parameter family of measures $\mathfrak{m}^x_t$ is
  derived from the mean-field limit \eqref{Vlasov}.
  The following theorem details the relationship between the two families of measures and thus
  clarifies the connection between the Vlasov PDE and the non-exchangeable IPS \eqref{KM}.

\begin{theorem}\label{thm.Vlasov}\cite{KVMed18}
For given $\epsilon,T>0$ and sufficiently large $m,\ell\in\N$, we have
\[
  \sup_{t\in[0,T]}
  \int_Q d_{BL}\bigl(\mathfrak{m}^x_{m,\ell,t},
  \mathfrak{m}^x_t\bigr)\,dx < \epsilon,
\]
provided $\mathfrak{m}^x_{m,\ell,0}$ converges weakly to
$\mathfrak{m}^x_0$ for almost every $x\in Q$ as $m,\ell\to\infty$.
Here $d_{BL}$ denotes the bounded Lipschitz distance
(cf.~\cite{Dud02}),
\[
  d_{BL}(\mu,\nu)
  :=\sup_{\|f\|_{BL}\leq 1}
  \left|\int_{\R^k} f\,d\mu-\int_{\R^k} f\,d\nu\right|,
\]
where
\[
  \|f\|_{BL}
  :=\|f\|_{L^\infty(\R^k)}+\operatorname{Lip}(f),\qquad
  \operatorname{Lip}(f)
  :=\sup_{\substack{x,y\in \R^k\\x\neq y}}
  \frac{|f(x)-f(y)|}{|x-y|}.
\]
\end{theorem}

\subsection{Convergence rate}
Theorem~\ref{thm.Vlasov} is the second result concerning the continuum description of the
IPS \eqref{KM} and \eqref{W-KM}.
We now turn to the third result, which addresses the speed of convergence of the discrete models to the
continuum limit. It follows from Theorem~\ref{thm.heat} that the accuracy of approximation of the
discrete model by the continuum
limit \eqref{cKM} depends on the error of approximation of the graphon $W$ by its $L^2$-projection $W^n$.
The same
is true for the Vlasov equation, albeit  it is not explicitly  stated  in Theorem~\ref{thm.Vlasov}.
Thus, the rate with which $\|W-W^n\|_{L^2(Q\times Q)}$ tends to zero with $n\to\infty$ determines
  the accuracy of the approximation of the  discrete models \eqref{KM} and \eqref{W-KM} by the continuum limit
\eqref{cKM} and the Vlasov equation \eqref{Vlasov}.

It is straightforward to estimate the rate of convergence of 
\(\|W - W^n\|_{L^2(Q \times Q)}\) when \(W\) is H\"{o}lder continuous \cite{KVMed22}. 
On the other hand, for a graphon \(W\) representing the limit of a graph sequence, the  natural 
regularity assumption is measurability or integrability in the case of \(L^p\)-graphons 
with \(p \geq 1\) \cite{BCCZ19}. For \(W \in L^p([0,1]\times [0,1])\), it is known that 
\(W^n \to W\) almost everywhere and in \(L^p\) (see, e.g., \cite[Proposition~2.6]{Cha17}). 
However, in the absence of additional assumptions, the convergence may be arbitrarily 
slow, as the following example from \cite{Med14a} suggests.

\begin{example}\label{ex.slow}
  Consider $W: [0,1]\times [0,1]\to \{0,1\}$ and denote by $W^+$ and $\p W^+$ the support of $W$ and
  its boundary respectively. Then as was observed in  \cite{Med14a}
\be\lbl{error-01}
\|W-W^n\|^2_{L^2([0,1]\times [0,1])}\le  N(\p W^+, n)n^{-2},
\ee
where $N(\p W^+,n)$ is the number of discrete cells 
$\left[\frac{i-1}{n}, \frac{i}{n}\right)\times\left[\frac{j-1}{n}, \frac{j}{n}\right), i,j\in [n],$
that intersect $\p W^+$.

Let $\gamma$ stand for the upper box counting dimension of $\p W^+$,
$N(\p W^+,n)$ (cf.~\cite{Falc-FracGeom}),
$$
\gamma=\varlimsup_{n\to\infty} \frac{\log N(\p W^+,n)}{\log n}.
$$
Then for any $\epsilon>0$, we have
\be\lbl{bound-N}
N(\p W^+,n)\le n^{\gamma +\epsilon}.
\ee
By plugging \eqref{bound-N} into \eqref{error-01}, we have
\be\lbl{final-01}
\|W-W^n\|_{L^2([0,1]\times [0,1])}\le n^{-1+\frac{\gamma+\epsilon}{2} }.
\ee
The bound on the error of approximation suggests that the convergence can be in principle arbitrarily slow 
if the upper box counting dimension of the boundary of support of $W$ is sufficiently close to $2$.
\end{example}

Thus we are led to the following question: \emph{Under what natural assumptions on \( W \),
beyond mere integrability, can one deduce a convergence rate for \( W^n \to W \)?}
In \cite{KVMed22}, the generalized Lipschitz
spaces (cf.~\cite{Nikol-approximation, DeVore-book}) were proposed to provide an appropriate functional
framework for this purpose:
\begin{equation}\label{Lip-space}
\Lip\left( L^p(Q),\alpha\right) = \left\{
    f\in L^p(Q):\quad \omega_p (f,\delta)\le C \delta^\alpha \right\},\quad
  0<\alpha\le 1.
  \end{equation}
  Here, the $L^p$--modulus of continuity $\omega_p(f,\delta)$ is  defined as follows
  \begin{equation}\label{modulus}
  \omega_p(f,\delta) =\sup_{|h|\le \delta} \| f(\cdot)-f(\cdot+h) \|_{ L^p(Q_h^\prime)},\qquad
    Q_h^\prime =\{x \in Q: \quad x+h\in Q \},
  \end{equation}
and $|x|$ stands for the $\ell^\infty$--norm of $x\in\R^d$, and $Q$ is a given domain.

$\Lip\left( L^p(Q),\alpha\right)$ is equipped with the norm
     \begin{equation}\label{Lip-norm}
       \|f\|_{\Lip\left( L^p(Q),\alpha\right)}=\sup_{\delta>0} \delta^{-\alpha}\omega_p(f,\delta).
     \end{equation}
     These spaces play an important role in approximation theory
     \cite{Nikol-approximation, DeVore-book} and are closely related to
     Nikolskii–Besov spaces, which are used in the analysis of partial differential
     equations \cite{BKP2019}.

  \begin{lemma}\label{lem.rate} \cite{KVMed22, MedSim22}
    For $f\in\Lip(\alpha, L^p(Q))$, 
    \be\lbl{rate-Lip}
    \|f-f^n\|_{L^p(Q)} \le C n^{-\alpha}.
    \ee
  \end{lemma}
  \begin{remark}\label{rem.daydic}
  The proof of Theorem~\ref{lem.rate} in \cite{KVMed22} uses a dyadic
  discretization of \( Q\),  
  exploiting the self-similarity of the unit cube. This approach suggests that the approximation
  result in Lemma~\ref{lem.rate} can be extended to functions in
  self-similar domains, such as the 
  Sierpinski gasket. We will address this problem in Section~\ref{sec.rate}.
\end{remark}

With Lemma~\ref{lem.rate} in hand, we can refine the convergence estimate in Theorem~\ref{thm.heat}
by identifying the exact rate of convergence.
\begin{theorem}\label{thm.rate}
  Let $u(t,x)$ be the solution of the IVP for \eqref{cKM} subject to
  $u(0,\cdot)=g\in \operatorname{Lip}\left( \alpha_1, L^2(Q)\right)$
 and let $u^n$ and $\mathsf{u}^n$ be the solutions of the discrete
  problems, as described in Theorem~\ref{thm.heat}.
  If $W\in \operatorname{Lip}\left( \alpha_2, L^2(Q\times Q)\right)$
 then
\begin{equation}\label{rate-Galerkin}
  \|u-v\|_{C(0,T; L^2(Q))}\le C n^{-\alpha}, \quad v\in\{ u^n, \mathsf{u}^n\},\;
  \alpha=\min\{\alpha_1, \alpha_2\}.
\end{equation}
Here, in the case of $v=\mathsf{u}^n$ estimate \eqref{rate-Galerkin} holds almost surely.
\end{theorem}
 
\section{Background on fractals}\label{sec.self-sim}
\setcounter{equation}{0}

In this section, we collect the material from fractal geometry that is necessary for
the formulation and analysis of dynamical models on self-similar networks. Specifically,
we review the attractors of Iterated Function Systems (IFS) and self-similar measures.
The IFS framework was introduced by Hutchinson \cite{Hut81}. Our presentation
is based on \cite{BSS-self-similar, BP-Fractals}.

\subsection{Iterated Function Systems}\label{sec.ifs}

We begin with the definition of the IFS.
\begin{definition}\label{df.IFS}
Let $\mathcal{F}=\{f_i:\mathbb{R}^d \to \mathbb{R}^d,\; i\in[k]\}$
be a finite collection
of maps such that each $f_i$ is a strict contraction with Lipschitz
constant
$0<r_i<1$, i.e.,
\begin{equation}\label{contract}
  |f_i(x)-f_i(y)| \le r_i |x-y| \quad \text{for all } x,y\in\mathbb{R}^d,\; i\in[k].
\end{equation}
Assume that the fixed points of the maps $f_i$ are not all identical, and that each $f_i$ is injective.  
Then $\mathcal{F}$ is called a contracting IFS.
\end{definition}

All results reviewed below hold for IFSs defined on a complete metric space. 
However, since our motivating examples of fractals are embedded in the Euclidean space, 
we restrict our attention to $\mathbb{R}^d$ for notational convenience.

In fact, in all examples considered below the contractions are \emph{similitudes}, i.e.,
\begin{equation}\label{similitude}
  |f_i(x)-f_i(y)| = r_i |x-y| 
  \quad \text{for all } x,y\in\mathbb{R}^d,\; i\in[k],
\end{equation}
with $0<r_i<1$. Nevertheless, since several of the results that follow remain
valid under the weaker assumption \eqref{contract}, we do not
impose \eqref{similitude} at this stage.

The contraction mapping principle, applied to the map
\[
A \longmapsto \bigcup_{i\in[k]} f_i(A)
\]
acting on the space of nonempty compact subsets of $\mathbb{R}^d$ equipped with the Hausdorff metric,
implies the existence of a unique compact set $K$ satisfying
\begin{equation}\label{selfsim}
  K = \bigcup_{i\in[k]} f_i(K)
\end{equation}
(cf.~\cite[Theorem~2.1.1]{BP-Fractals}, \cite{Hut81}).
The set $K$ is called the \emph{attractor} of the IFS $\mathcal{F}=\{f_i\}_{i\in [k]}$.

If \(f_i\), \(i\in[k]\), are similitudes, then the attractor \(K\) of the IFS
\(\mathcal{F}\) is called a \emph{self-similar set}. When no confusion is
likely to arise, we sometimes use this term more loosely to refer to
attractors of general IFS.

Many canonical examples of fractals, including sets of noninteger Hausdorff
dimension, can be realized as self-similar sets. On the other hand,
certain regular nonfractal sets, such as the unit cube, also fall within
this framework. The fact that the unit cube is a self-similar set is used
in Section~\ref{sec.isomorph}, where we establish an isomorphism between
IPS on self-similar sets and those on graphons.

\begin{example}\label{ex.fractals}
  We now present several representative examples of self-similar sets.
 \begin{enumerate}
    \item The Cantor set is the attractor of the following IFS on $\R$:
    \[
      \{f_1(x) = \frac{x}{3},\; f_2(x) = \frac{x}{3} + \frac{2}{3}\}.
    \]

  \item The Sierpinski gasket (SG) is a prototypical example of a planar fractal and is often used as a
model domain in the analysis of fractals \cite{Kig01, Str06}. We shall
frequently refer to it throughout this work. It is the attractor of the
following IFS on \(\mathbb{R}^2\):
    \begin{equation}\label{IFS-SG}
      f_i(x) = \frac{1}{2}(x + v_i), \quad i \in [3],
    \end{equation}
    where
    \[
      v_1 = (0,0),\quad v_2 = \Big(\frac{1}{2},\frac{\sqrt{3}}{2}\Big),\quad v_3 = (1,0).
    \]
    \item The unit interval is the attractor of the following IFS on $\mathbb{R}$:
    \[
      f_1(x) = \frac{x}{2}, \quad f_2(x) = \frac{x}{2} + \frac{1}{2}.
    \]
  \end{enumerate}
\end{example}

\subsection{Symbolic space}\label{sec.symbol}

Self-similar sets admit a canonical symbolic representation, which we review in this subsection.
It provides a convenient way for analyzing topological and ergodic properties
of self-similar sets.

Let 
\begin{equation}\label{def-Sigma}
\Sigma\doteq [k]^\N
\end{equation}
denote the set of infinite sequences of $k$ symbols. For each $n\ge1$, let $\Sigma_n$ be the set of words
of length $n$,
and let
\begin{equation}\label{Sigma-star}
  \Sigma^\ast=\bigcup_{n=0}^\infty\Sigma_n
\end{equation}
denote the collection of all finite words over $[k]$.

We equip $\Sigma$ with the metric
$$
\rho(i,j)=k^{-|i\wedge j|}, \quad
i=(i_1,i_2,\dots),\; j=(j_1,j_2,\dots),
$$
where $i\wedge j$ stands for the common prefix of the two sequences $i$ and $j$, and
$|i\wedge j|$ denotes its
length. Thus, $|i\wedge j|=\sup\{l\in\N:\, i_l=j_l\}$ if $i\wedge j\neq\emptyset$ and $0$ otherwise.
This metric defines a topology on $\Sigma$ that is consistent with the product
topology. By Tikhonov's theorem, $(\Sigma,\rho)$ is a compact metric space
(cf.~\cite{BSS-self-similar}).

Define the shift map $\sigma:\Sigma\to\Sigma$ by
$$
\sigma(w_1w_2w_3\dots)=w_2w_3\dots .
$$
For each $i\in[k]$, the $i$th branch of the inverse of $\sigma$ is given by
\begin{equation}
\sigma_i(w_1w_2w_3\dots)=iw_1w_2w_3\dots .
\end{equation}
Each $\sigma_i$ is a contracting similitude on $\Sigma$ with Lipschitz constant $k^{-1}$.
Consequently, $\Sigma$ is a self-similar set with respect to $\{\sigma_i\}_{i\in[k]}$.

We now describe the symbolic representation of $K$, the self-similar set with respect to
$\{f_i\}_{i\in[k]}$. It is provided by \textit{the natural projection} $\pi:\Sigma\to K$ to be defined below.

For $w=(w_1\ldots w_n)\in\Sigma_n$, denote
$$
  f_w\doteq f_{w_1}\circ\cdots\circ f_{w_n}, \quad K_w\doteq f_w(K).
$$

The natural projection is defined by
\begin{equation}\label{def-pi}
  \pi(w)=\bigcap_{n\ge1} K_{w_1\ldots w_n}, \qquad w=(w_1w_2\dots)\in\Sigma.
\end{equation}
Since $\{K_{w_1\ldots w_n}\}_{n\ge1}$ is a nested sequence of compact sets with diameters tending to zero,
their intersection consists of a single point. The map $\pi$ is continuous and surjective
(cf.~\cite[Theorem~1.2.3]{Kig01}).
Moreover, for every $i\in[k]$
\begin{equation}\label{semiconj}
  \pi\circ\sigma_i=f_i\circ\pi.
\end{equation}

\subsection{Self-similar measures}\label{sec.selfsim-measures}

Measures on a self-similar set \(K\) will be defined as pushforwards of
Bernoulli measures on the symbolic space. Therefore, we first review the
construction of Bernoulli measures.

For $i_1,\dots,i_n\in[k]$, define the cylinders
\begin{equation}\label{cylinder}
[i_1 i_2 \dots i_n]
\doteq
\{j=(j_1j_2\dots) \in\Sigma:\; j_l=i_l,\; 1\le l\le n\}.
\end{equation}
The collection of all cylinders generates the Borel $\sigma$-algebra in $\Sigma$, $\cB(\Sigma)$.

We say that $p=(p_1,p_2,\dots,p_k)$ is \textit{a probability vector}, if
$$
p_i>0 \quad \text{for all } i\in[k], \qquad \sum_{i=1}^k p_i=1.
$$
For a given probability vector $p$, the Bernoulli measure of the cylinders is defined by 
\begin{equation}\label{Bernoulli-m}
\mu_p([i_1 i_2 \dots i_n])=p_{i_1}p_{i_2}\cdots p_{i_n}.
\end{equation}

It extends uniquely to a Borel probability measure on $(\Sigma,\cB(\Sigma))$.
Bernoulli measure $\mu_p$ is invariant under the shift map $\sigma$
(cf.~\cite{BSS-self-similar}), i.e.,
$$
\mu_p(\sigma^{-1}([i_1 i_2 \dots i_n]))=\mu_p([i_1 i_2 \dots i_n]).
$$

Furthermore, $\mu_p$ is \textit{ergodic}, meaning that for every invariant set $A=\sigma^{-1}(A)$,
$\mu_p(A)\in\{0,1\}$.

Next, let $K$ be an attractor of the IFS $\cF$ and consider the pushforward 
$\nu_p=\pi_\ast\mu_p$:
\begin{equation}\label{push}
\nu_p(A)=\mu_p(\pi^{-1}(A)), \; A\in\cB(K).
\end{equation}
The pushforward measure $\nu_p=\pi_\ast\mu_p$ is \textit{the stationary measure} of the probabilistic IFS
$(K,\cF, \nu_p)$, which means that
\be\lbl{sta-meas}
\nu_p(A)=\sum_{j=1}^k p_j\,\nu_p\bigl(f_j^{-1}(A)\bigr),
\qquad A\in\cB(K).
\ee
The stationary measure is unique \cite[Theorem~2.1.1]{BP-Fractals}.
If $f_i$'s are similitudes, $\nu_p$ is called \emph{a self-similar measure}.

Assume now that $\{f_i\}_{i\in[k]}$ are similitudes with contraction ratios
$(r_1,r_2,\dots,r_k)$.  \textit{The similarity dimension of $K$} is $s>0$ such that
$$
r_1^s+r_2^s+\cdots+r_k^s=1.
$$
The stationary measure corresponding to the probability vector
$p=(r_1^s,r_2^s,\dots,r_k^s)$ is called the \emph{natural measure} of the IFS $\cF$.
If the Open Set Condition holds, then the Moran--Hutchinson theorem implies that
$\dim_H K=s$ and that the natural measure is the normalized
$s$-dimensional Hausdorff
measure restricted to $K$, namely $c\,\cH^s|_K$ (cf.~\cite[Theorem~2.2.2]{BP-Fractals}).

We include the definition of the Open Set Condition for the reader's convenience.
\begin{definition}\label{def.OS}
  A family of maps $\{f_i\}_{i\in [k]}$ on $\R^d$ satisfies the Open Set Condition if there is
  a bounded nonempty  open set $U\subset \R^d$ such that
  \begin{align*}
    f_j(U)\subset U,& \quad j\in [k],\\
    f_i(U)\cap f_j(U)=\emptyset, &\quad i\neq j.
  \end{align*}
  \end{definition}

  The Open Set Condition says that the overlap of $f_i(K) \cap f_j(K)$ for $i\neq j$ is small.
  It is one of several separation conditions used to describe the degree of the overlap.
  All sets in Example~\ref{ex.fractals} satisfy the Open Set Condition. For instance, to verify this
  condition for the Sierpinski gasket, one can take the interior of
  the convex hull of Sierpinski gasket.
  In Section~\ref{sec.isomorph},
  we will define IFS with essentially injective coding (see
  Definition~\ref{df.ess-inject}), which satisfy
 the following explicit formula for the pushforward measure $\nu_p$ on the cylinders:
\begin{equation}\label{nu-cyl}
\nu(K_w)=p_{w_1}p_{w_2}\dots p_{w_n},
\qquad \forall w\in\Sigma_n,\; n\in\N.
\end{equation}

\section{Self-similar IPS}\label{sec.self-IPS}
\setcounter{equation}{0}

\subsection{The model}
Having reviewed graphon IPS 
and the necessary background on fractals, we now turn to the formulation of a dynamical model 
on self-similar networks.

Let $K$ be an attractor of an IFS $\cF=\{f_1,f_2,\dots, f_k\}$
equipped with a the pushforward (probability)
measure $\nu=\pi_\ast\mu$ subject to \eqref{nu-cyl} and 
consider the following IPS:
\begin{align}\label{ss-KM}
  \dot u_w &= f(t, u_w) + \sum_{|v|=n} W_{wv} D(u_w, u_v)\nu(K_v),\quad
             w\in \Sigma_n,\\
  \label{ss-KM-ic}
             u_w(0) &= g_w, \quad  w\in \Sigma_n,
\end{align}
where 
\begin{equation}\label{ave-W}
W_{wv}=\fint_{K_{wv}} W(x,y) d(\nu\times\nu)(x,y), \quad g_w=\fint_{K_w} gd\nu.
\end{equation}
Throughout, $\fint_A f(x)\,dm(x)$ denotes the average value of $f$ over $A$, namely
\[
\fint_A f(x)\,dm(x) := \frac{1}{m(A)} \int_A f(x)\,dm(x).
\]

\begin{equation}\label{partition}
    \mathcal{K}^m\doteq \left\{K_w:\;K_w=f_w(K),\; |w|=m\right\}
\end{equation}
is a self-similar partition of $K=\bigcup_{|w|=m} K_w.$
We remind the reader that $f_w\doteq f_{w_1}\circ f_{w_2} \circ\dots\circ f_{w_m}$.

Functions $W\in L^2(K\times K, \nu\times\nu)$, $g\in L^2(K,\nu)$. Further
we assume
 \be\lbl{bound-W}
 \bar W\doteq \max\left\{ \esssup_{x\in K} \int_K |W(x,y)| d\nu(y), \;
  \esssup_{y\in K} \int_K |W(x,y)| d\nu(x)\right\}<\infty.
\ee
Functions $f(t,u)$ and $D(u,v)$ are jointly continuous and 
\begin{align}\lbl{Lip-f}
  |f(t,u)-f(t,u^\prime)|\le L_f |u-u^\prime|,& \quad \forall t\in\R, \; u, u^\prime\in \R^k,\\
  \lbl{Lip-D}
  |D(u,v)-D(u^\prime,v^\prime)|\le L_D \left(|u-u^\prime| +|v-v^\prime|\right),& \quad \forall
                                                                                 u,v, u^\prime, v^\prime\in \R^k,
\end{align}
where $L_f$ and $L_D$ are positive Lipschitz constants.
In addition, we assume
\be\lbl{bound-D}
\|D\|_\infty\le 1.
\ee

It is instructive to compare \eqref{ss-KM} with the IPS on a weighted graph \eqref{KM}. Note that the role of
the graphon on the unit square is now assumed by \( W \) on \( K \times K \), where \( K \) is an attractor of
the IFS $\cF$.
The weights \( (W_{wv}) \) reflect the nested, possibly fractal, structure of \( K \). This endows the
dynamical network model \eqref{ss-KM} with a natural self-similar organization.

\begin{remark}\label{rem.partitions}
Recall that \( [0,1] \) is the attractor of the IFS
\(\{ f_1(x) = \tfrac{x}{2},\; f_2(x) = \tfrac{1}{2} + \tfrac{1}{2}x \}\).
Consequently, the model \eqref{ss-KM} contains the graphon IPS \eqref{KM}
as a special case. More broadly, self-similar IPS  offer a natural
framework for the study of IPS on
spaces endowed with hierarchical partitions.
\end{remark}

In analogy with \eqref{cKM}, we expect the following continuum limit for \eqref{ss-KM}
\begin{align}\label{ss-heat}
  \p_tu(t,x)& =f(t,u)+\int_K W(x,y) D\left(u(t,x),u(t,y)\right)d\nu(y), \\
  \label{ss-heat-ic}
  u(0,x)& =g(x),\qquad x\in K.
\end{align}

The wellposedness of the IVP \eqref{ss-heat}, \eqref{ss-heat-ic} is established by the following theorem.
\begin{theorem}\label{thm.well}
  Suppose \eqref{bound-W}-\eqref{bound-D} hold and $g\in L^2(K,\nu)$. Then for any $T>0$,
  the IVP \eqref{ss-heat}-\eqref{ss-heat-ic} has a unique solution $u\in C\left([-T, T]; L^2(K,\nu)\right)$.
\end{theorem}
\begin{proof}
  Let $\bX \doteq L^2(K,\nu)$ and consider $\bM:\bX\to \bX$ defined as follows
  \begin{equation}\label{def-M}
    \bM[\bu]\doteq \int_K W(\cdot, y)D\left(\bu(\cdot), \bu(y)\right) d\nu(y).
  \end{equation}

  By \eqref{bound-D} and \eqref{bound-W},
\[
  |\bM[\bu](x)|
  \le \int_K |W(x,y)|\,d\nu(y)
  \le \bar W
\]
for a.e.\ $x\in K$. Thus, $\bM[\bu]\in \bX$.

  Equation~\eqref{ss-heat} is a nonlinear differential equation
  on the Banach space $\bX$:
  \begin{equation}\label{dxdt}
    \frac{d\bu}{dt}=F(t,\bu),
    \end{equation}
    where
    \begin{equation}\label{rhsM}
      F(t,\bu)\doteq f(t,\bu)+\bM[\bu].
      \end{equation}
  
      In view of \eqref{Lip-f}, to show that the right-hand side of \eqref{rhsM} is Lipschitz continuous
      in $\bu$, we only need to verify that $\bM$ is Lipschitz continuous.
      To this end, let $\|\cdot\|$ stand for the norm in $L^2(K,\nu)$ and note
  \begin{align*}
    \| \bM[\bu]-\bM[\bv]\|& \le \left\| \int_K \left| W(\cdot,y) \right|
                            \left| D\left(\bu(\cdot), \bu(y)\right)
                        - D\left(\bv(\cdot), \bv(y)\right)\right| d\nu(y)\right\|\\
    \le& L_D \left\| \int_K \left| W(\cdot,y) \right| \left|\bu(\cdot) -\bv(\cdot)\right| d\nu(y)\right\|\\
                          & +
                            L_D \left\| \int_K \left| W(\cdot,y) \right| \left|\bu(y) -\bv(y)\right| d\nu(y)\right\|\\ 
    \le & L_D\bar W
          \|\bu-\bv\|+ L_D\left\| \left(\int_K  W(\cdot,y)^2 d\nu(y) \right)^{1/2} \|\bu-\bv\|
          \right\|  \\
            &\le  L_D\left( \bar W+\|W\|_{L^2(K\times K,\nu\times\nu)}\right) \|\bu-\bv\|.
  \end{align*}

  The statement of the theorem now follows from a standard result on existence and uniqueness of
  solutions of initial value problems for differential equations in Banach spaces
  (cf.~\cite[Theorem~1.2; Chapter~VII]{DalKrejn}).
\end{proof}

\subsection{Galerkin discretization}\label{sec.Galerkin}

We want to show that \eqref{ss-heat} is the continuum
limit of the self-similar IPS \eqref{ss-KM}. The key observation is that
\eqref{ss-KM} can be viewed as a Galerkin discretization of
\eqref{ss-heat}. We then analyze the associated Galerkin scheme,
prove its convergence, and derive an explicit convergence rate.

We begin by discretizing the underlying space \( K \).
Recall the level-\( m \) partition \( \mathcal{K}^m \) of \( K \)
(cf.~\eqref{partition}). This partition induces a finite-dimensional
subspace of
\( \mathcal{H} \doteq L^2(K,\nu) \), defined by
\[
\mathcal{H}^m
:= \operatorname{span}\{\mathbf{1}_A : A \in \mathcal{K}^m\}.
\]
Recall that \( \mathbf{1}_A \) denotes the indicator function of the set
\( A \).
Although the sets \( K_w \) and \( K_v \) corresponding to distinct
\( w,v \in \Sigma_m \) may intersect, their intersections have
\( \nu \)-measure zero by \eqref{nu-cyl}. Consequently, the resulting
overlaps in the supports of the basis functions do not affect 
the \( L^2 \)-projections onto
\( \mathcal{H}^m \).

We now construct the Galerkin approximation. The solution of the initial
value problem \eqref{ss-heat}–\eqref{ss-heat-ic} is approximated by a
piecewise constant function of the form
\begin{equation}\label{sum}
  u^m(t,x)
  = \sum_{|w|=m} u_w(t)\,\mathbf{1}_{K_w}(x).
\end{equation}
Inserting \eqref{sum} into \eqref{ss-heat} and projecting the resulting
equation onto \( \mathcal{H}^m \) yields precisely the finite-dimensional
system \eqref{ss-KM}.

It is convenient to rewrite \eqref{ss-KM} as an
evolution equation on \( \mathcal{H} \):
\begin{equation}\label{re-KM-ss}
  \partial_t u^m(t,x)
  = f(t,u^m)
  + \int_K W^m(x,y)\,
  D\bigl(u^m(t,x),u^m(t,y)\bigr)\, d\nu(y),
\end{equation}
where the discretized kernel \( W^m \) is given by
\begin{equation}\label{def-Wm}
  W^m
  := \sum_{|w|,|v|=m} W_{wv}\,\mathbf{1}_{K_w \times K_v}, \quad W_{wv}=\fint_{K_w \times K_v}W\,d(\nu\times \nu).
\end{equation}

The initial condition \eqref{ss-heat-ic} is approximated in the same
spirit by
\begin{equation}\label{approx-g}
  g^m
  := \sum_{|w|=m} g_w\,\mathbf{1}_{K_w},
  \qquad
  g_w := \fint_{K_w} g\, d\nu .
\end{equation}

\subsection{W-random self-similar networks and continuum limit}

Along with the
deterministic model \eqref{ss-KM}, we consider an IPS on a random network:
\begin{equation}\label{KM-Wss}
  \dot{\mathsf{u}}_w=f(t, \mathsf{u}_w) +\sum_{|v|=n} \xi_{wv} D(\mathsf{u}_w, \mathsf{u}_v)
  \nu(K_v),\quad
  w\in \Sigma_n,
\end{equation}
where $\xi_{wv}$ are independent Bernoulli random variables such that
\begin{equation}\label{ss-Bernoulli}
\fP(\xi_{wv}=1)=\fint_{K_w\times K_v} W d(\nu\times\nu),\quad \fP(\xi_{wv}=0)=1-\fP(\xi_{wv}=1).
\end{equation}
For the random model, in addition, we assume $0\le W\le 1$.

In analogy to the graphon IPS, we expect that both the deterministic model \eqref{ss-KM} and its
random counterpart \eqref{KM-Wss} converge to the following nonlocal equation
on $K$:
\begin{equation}\label{fKM}
\partial_t u(t,x) = f(t,u) + \int_K W(x,y) D\left(u(t,x), u(t,y)\right)\, d\nu(y), \quad x \in K,
\end{equation}
where $\nu$ is a stationary probability measure on $K$ and $W \in
L^1(K\times K, \nu\times\nu)$.

Likewise, if the IPS \eqref{ss-KM} is supplied with random initial conditions or has random parameters,
we expect that in the large $n$ limit its dynamics is captured by the Vlasov equation
\begin{align}\label{ss-Vlasov}
 & \partial_t\bar\rho(t,u,x)+\partial_u\left\{ V(t,u,x)\bar\rho(t,u,x) \right\}=0,\\
  \label{ss-Vlasov-VF}
  & V(t,u,x)= f(t,u)+  \int_{K}\int_{R^k} W(x,y) D(u,v) \bar\rho(t,v,y) dv\,d\nu(y).
\end{align}
Here, $\bar\rho(t,\cdot, x)$ represents the probability density of the state of particle $x\in K$. Note that compared
to \eqref{Vlasov}, the label set $K$ is now fractal and the integral on the right-hand side of
\eqref{ss-Vlasov-VF} is taken with respect to the measure $\nu$.

Our next goal is to justify \eqref{ss-heat} as the continuum limit of \eqref{ss-KM} and the Vlasov equation
\eqref{ss-Vlasov} as the mean-field limit of \eqref{KM-Wss}.
The precise results are formulated in the two theorems below.

\begin{theorem}\label{thm.ss-heat}
  Let $u(t,x)$ be the solution of the IVP for \eqref{ss-heat} subject to $u(0,\cdot)=g\in L^2(K,\nu)$.
  Likewise, suppose
  \begin{align*}
    u^n(t,x) &= \sum_{|j|=n} u_{j}(t) \1_{K_j}(x),\\
    \mathsf{u}^n(t,x) &= \sum_{|j|=n} \mathsf{u}_{j}(t) \1_{K_j}(x),
\end{align*}
solve the IVPs for \eqref{ss-KM} and \eqref{KM-Wss}, respectively, and satisfy 
$$
u_n(0,\cdot)=\mathsf{u}^n (0,\cdot)= g^n.
$$
Then
\begin{equation}\label{ss-approx-heat}
  \|u-v\|_{C(0,T; L^2(K,\nu))} \le
  C \left( \|g-g^n\|_{L^2(K,\nu)} + \|W-W^n\|_{L^2(K\times K, \nu\times\nu)}\right),\;  v\in\{ u^n, \mathsf{u}^n\}.
\end{equation}
Here, in the case of $v=\mathsf{u}^n$ estimate \eqref{ss-approx-heat} holds almost surely, i.e., for almost every
realization of the W-random graph \eqref{ss-Bernoulli}.
\end{theorem}
\begin{remark}\label{rem.ss-heat-approx}
  Estimate \eqref{ss-approx-heat} can be extended to cover sparse self-similar networks in parallel to how
  this was done \cite{Med19} for graphon IPS.
  Furthermore, the Large Deviation Principle from \cite{DupMed22} translates to the self-similar
  setting as well.
\end{remark}

We next turn to the Vlasov equation on $K$. To this end, we adapt the
definition of the local
empirical measure to the present setting
\begin{equation}\label{ss-loc-emp}
  \mathfrak{m}^x_{m,\ell,t}(A)=k^{-\ell}
  \sum_{|j|=\ell} \1_A \left(u_{wj}(t)\right), \quad
         A\in\mathcal{B}(\R^k), \; x\in K_{w},\; |w|=m.
       \end{equation}
Further,      
  \begin{equation}\label{Vlasov-mes}
    \mathfrak{m}^x_t(A)=\int_A \bar\rho(t,x,u) du, \quad A\in\mathcal{B}(\R^k).
  \end{equation}
Thus, as in the case of graphon IPS we have
\begin{theorem}\label{thm.ss-Vlasov}
  For given $\epsilon, T>0$ and sufficiently large $m, \ell\in\N$, we have
  $$
  \sup_{t\in [0,T]} \int_K d_{BL} (\mathfrak{m}^x_{m,\ell,t},
  \mathfrak{m}^x_t) d\nu(x) <\epsilon,
  $$
provided $\mathfrak{m}^x_{m,\ell,0}$
converge weakly to $\mathfrak{m}^x_0$ for almost every $x\in K$ as $m,\ell\to\infty.$
\end{theorem}

We will address the question of the convergence rate for the discrete model \eqref{ss-KM}
after proving Theorems~\ref{thm.ss-heat} and \ref{thm.ss-Vlasov}.

\section{Isomorphisms}\label{sec.isomorph}
\setcounter{equation}{0}

In preparation for studying convergence of self-similar IPS, we
establish an isomorphism between a given IFS
\(\mathcal{F}=\{f_i\}_{i\in [k]}\) with
attractor \(K\), equipped with a probability measure \(\nu\), and a
canonical IFS \(\mathcal{G}=\{g_i\}_{i\in [k]}\) whose attractor is
the unit interval
\(Q=[0,1]\), equipped with the corresponding self-similar measure \(\lambda\).
This isomorphism is then lifted, via pullbacks, to the function spaces
\(L^1(K,\nu)\) and \(L^1(Q,\lambda)\), as well as to the corresponding
product spaces \(L^1(K\times K,\nu\times\nu)\) and
\(L^1(Q\times Q,\lambda\times\lambda)\).
These identifications will be used in the next section to construct an
isomorphism between self-similar IPS on \(K\) and their graphon
counterparts on \(Q\). In turn, this will allow us to translate
existing convergence results for graphon IPS to their self-similar
counterparts.

\subsection{Preparation}\label{sec.aux}

In this subsection, we collect several measure-theoretic preliminaries
that will be helpful in what follows.

Let $\Sigma = [k]^{\mathbb{N}}$ be the symbolic space equipped with a
Bernoulli probability measure $\mu$.
Consider the IFS $(K,\{f_i\}_{i\in[k]})$.
Denote by $\pi_K:\Sigma \to K$  the natural projection and let $\nu = (\pi_K)_*\mu$.
The triple $(K,\{f_i\}_{i\in[k]}, \nu)$ is referred to as \textit{a probabilistic IFS}.

It will be convenient to work with complete measure spaces. Thus,
we consider the completions of $\mu$ and $\nu$ and keep the same notation
for the completed measures.
Let $\mathcal{A}_\mu$ and $\mathcal{A}_\nu$ denote the corresponding
$\sigma$-algebras of measurable sets.

Assume that $(K,\{f_i\}_{i\in[k]},\nu)$ satisfies the Lusin $(\mathbf{N})$
property:
\begin{equation}\label{N-prop}
  \mu(A)=0 \quad \Longrightarrow \quad \nu(\pi_K(A))=0
  \qquad \forall A \in \mathcal{A}_\mu .
\end{equation}

\begin{lemma}\label{lem.Nprop}
Under the assumptions above, the following holds:
\begin{enumerate}
  \item $\pi_K(A) \in \mathcal{A}_\nu$ for every $A \in \mathcal{A}_\mu$.
  \item $\pi_K^{-1}(A) \in \mathcal{A}_\mu$ for every $A \in \mathcal{A}_\nu$.
  \item $\pi_K$ is measure preserving, i.e.
  \[
    \nu(A) = \mu(\pi_K^{-1}(A)) \qquad \forall A \in \mathcal{A}_\nu .
  \]
\end{enumerate}
\end{lemma}
\begin{remark}\label{rem.null} Note that $\pi_K$ maps null sets to null
  sets by the Lusin $(\mathbf{N})$ property. Likewise, $\pi_K^{-1}$
  sends null sets to null sets by the definition of the pushforward
  measure $\nu= (\pi_K)_*\mu$.
\end{remark}  
\begin{proof}
\begin{enumerate}
\item
Let $A \in \mathcal{A}_\mu$.
Since $\mu$ is the completion of a Borel measure on $\Sigma$,
there exist a Borel set $B \subset \Sigma$ and a $\mu$-null set $N$
such that
\[
A = B \cup N .
\]

The map $\pi_K$ is continuous, and $\Sigma$ and $K$ are Polish spaces; hence $\pi_K(B)$
is an analytic  subset of $K$ (cf.~\cite[Theorem~14.4]{Kechris-SetTheory}).
Analytic sets are universally measurable (cf.~\cite[Theorem~21.10]{Kechris-SetTheory}) 
and, in particular, measurable with respect to any complete Borel measure, 
thus $\pi_K(B) \in \mathcal{A}_\nu$.

By the Lusin $(\mathbf{N})$ property \eqref{N-prop},
$\nu(\pi_K(N)) = 0$, and since $\nu$ is complete,
$\pi_K(N) \in \mathcal{A}_\nu$.
Therefore,
\[
\pi_K(A) = \pi_K(B) \cup \pi_K(N) \in \mathcal{A}_\nu.
\]

\item
Let $A \in \mathcal{A}_\nu$.
Since $\nu$ is complete, there exist a Borel set $B \subset K$
and a $\nu$-null set $N$ such that
\[
A = B \cup N .
\]
Because $\pi_K$ is continuous, $\pi_K^{-1}(B)$ is a Borel subset of
$\Sigma$, hence $\mu$-measurable.
Moreover,
\[
\mu(\pi_K^{-1}(N)) = \nu(N) = 0 .
\]
Thus,
\[
\pi_K^{-1}(A) = \pi_K^{-1}(B) \cup \pi_K^{-1}(N) \in \mathcal{A}_\mu .
\]

\item
The measure-preserving property follows directly from the definition
of the pushforward measure $\nu = (\pi_K)_*\mu$:
for every $A \in \mathcal{A}_\nu$,
\[
\nu(A) = \mu(\pi_K^{-1}(A)) .
\]
\end{enumerate}
\end{proof}

\begin{definition}\label{df.nonsingular} Let $(K,\mathcal{A}_\nu,\nu)$ be a complete measure space.  A measurable map
  $f: K\rightarrow K$ is called \emph{nonsingular} if for every null set $N$, $\nu\left(f^{-1}(N)\right)=0$.
\end{definition}

\begin{lemma}\label{lem.nosingular}
  Let $(K, \{f_i\}_{i\in [k]}, \nu_p)$ be a probabilistic IFS. Then $f_i:K\to K$ is nonsingular for every
  $i\in [k]$.
\end{lemma}

\begin{proof}
The  self-similar measure $\nu_p$ satisfies
\[
\nu_p=\sum_{j=1}^k p_j\,\nu_p\circ f_j^{-1},\qquad p_j>0.
\]
If $\nu_p(N)=0$, then
\begin{equation}\label{nu-stationary}
  0=\nu_p(N)=\sum_{j=1}^k p_j\,\nu_p(f_j^{-1}(N)).
  \end{equation}
Since every term on the right-hand side of \eqref{nu-stationary} is nonnegative and $p_j>0$, we obtain
\[
\nu_p(f_i^{-1}(N))=0
\]
for every $i\in [k]$.
\end{proof}

\begin{definition}\label{df.ess-inject}
  A probabilistic IFS $(K, \{f_i\}_{i\in [k]}, \nu_p)$ is said to have an essentially injective 
    coding if the noninjectivity set of $\pi_K$ 
  \begin{equation}\label{non-inject}
\cN_{K}=\left\{w\in\Sigma:\; \operatorname{card}\left\{\pi_K^{-1}(\pi_K(w))\right\}>1\right\}
\end{equation}
has measure $0$:
\begin{equation}\label{tiny}
\mu_p(\cN_K)=0.
\end{equation}

\end{definition}

In the remainder of this paper, we work with IFSs that have essentially injective coding. To this end,
we indicate the following implication of this property.
\begin{lemma}\label{lem.Lusin}
  Let $(K, \{f_i\}_{i\in [k]}, \nu_p)$ have essentially injective coding. Then
  it satisfies the Lusin $(\mathbf{N})$ property.
\end{lemma}
\begin{proof}
Let $A \in \mathcal A_{\mu_p}$ with $\mu_p(A)=0$. Choose a Borel null set $B\supset A$. Since $\pi_K$ is
Borel measurable, $\pi_K(B)$ is analytic and hence $\nu_p$-measurable.
Moreover,
\[
\pi_K^{-1}\bigl(\pi_K(B)\bigr)
\subset B\cup\cN_K.
\]
Therefore,
\[
\nu_p\bigl(\pi_K(B)\bigr)
=
\mu_p\left(\pi_K^{-1}\bigl(\pi_K(B)\bigr)\right)
=0.
\]
Since
\[
\pi_K(A)\subset\pi_K(B),
\]
completeness of $\nu_p$ implies that $\pi_K(A)$ is measurable and has
measure zero. 
\end{proof}

Thus, by Lemma~\ref{lem.Nprop}, the natural projection $\pi_K$ and its
inverse are measurable with respect to completed $\sigma$-algebras
$\cA_\mu$ and $\cA_\nu$,  respectively.

We conclude  this subsection with three examples of self-similar sets illustrating condition \eqref{tiny}.

\begin{figure}
	\centering
 \textbf{a}\includegraphics[width =.3\textwidth]{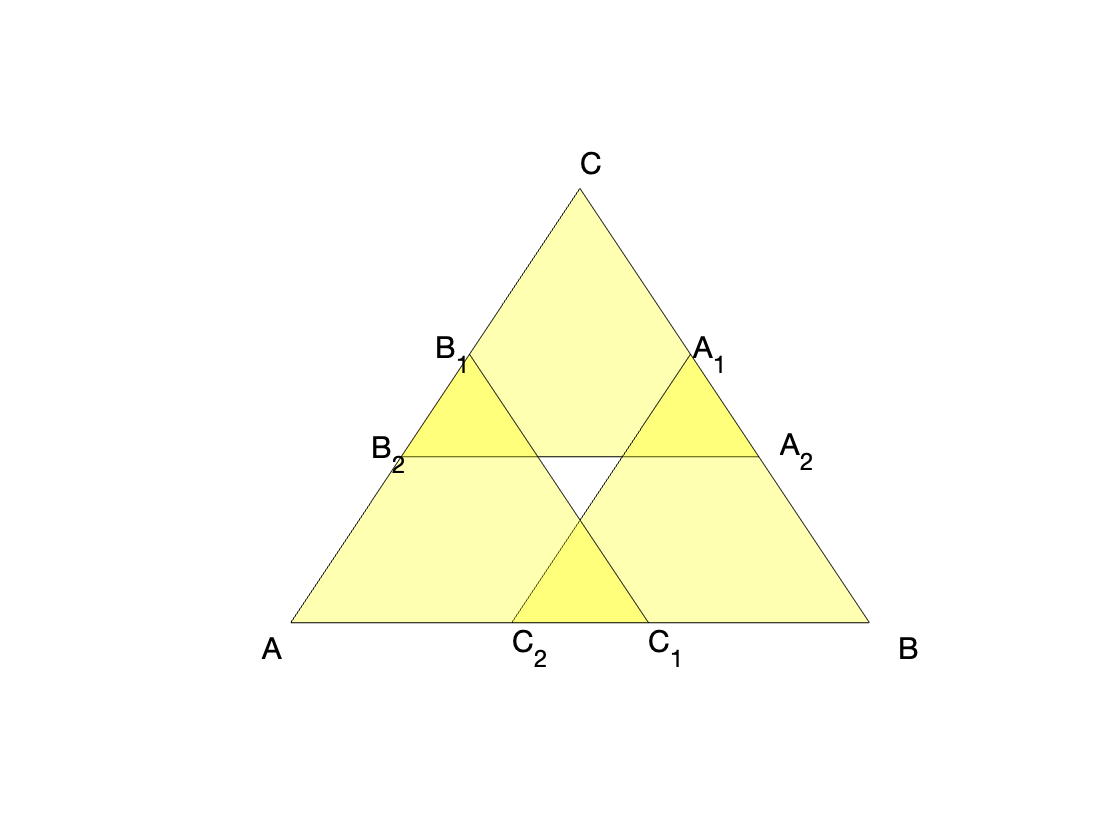}\;
 \textbf{b} \includegraphics[width = .3\textwidth]{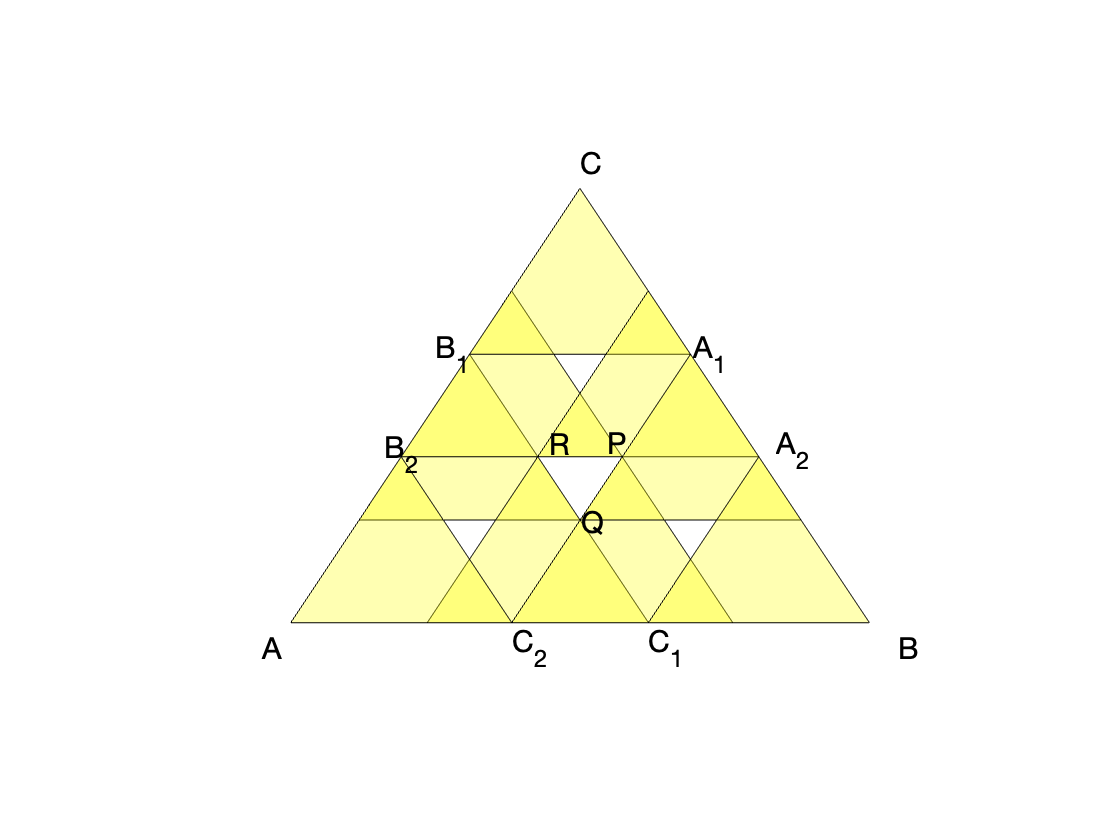}\;
 \textbf{c}\includegraphics[width =.3\textwidth]{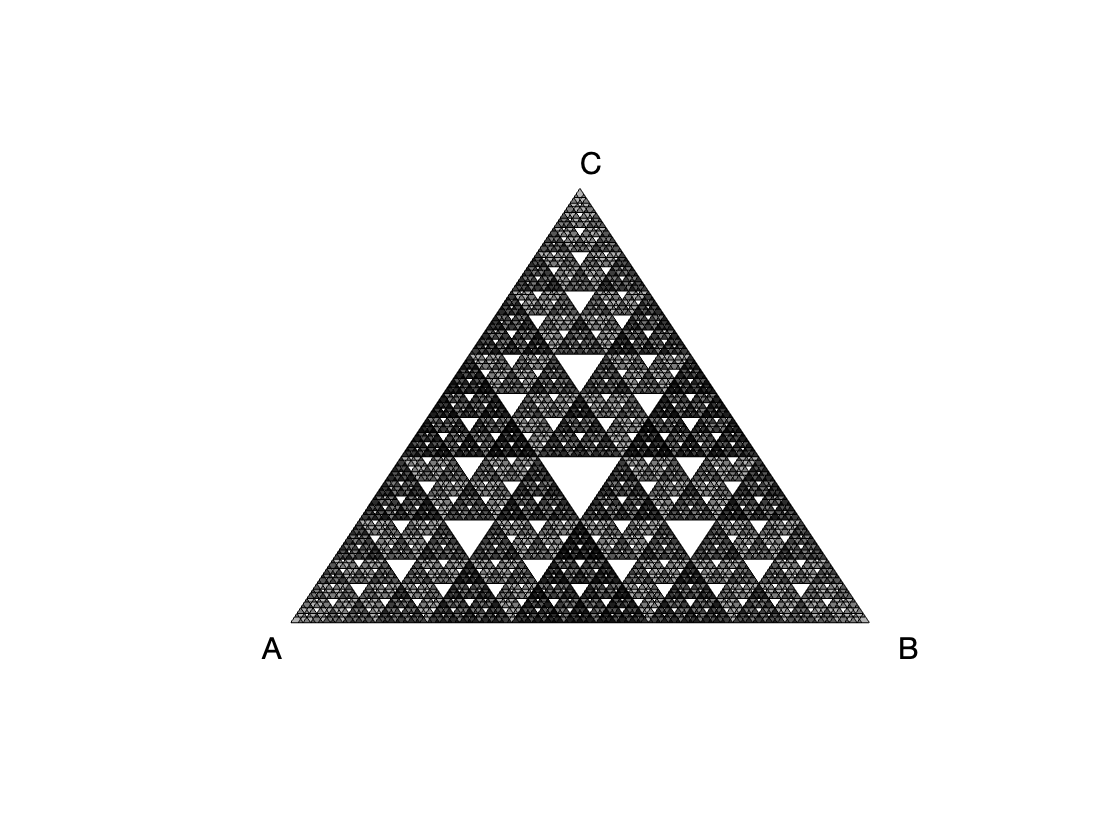}
 \caption{\textbf{a, b}) The first two steps in the construction of the golden gasket
(see Example~\ref{ex.overlap} for its definition). \textbf{c}) The golden gasket.
  }\label{f.goldSG}
\end{figure}

\begin{example}\label{ex.overlap}
  \begin{enumerate}
  \item The middle-third Cantor set, $C$, is the attractor of the IFS
    $\{f_1(x)=\frac{x}{3}, \; f_2(x)=\frac{x}{3}+\frac{2}{3}\}$. It satisfies the Strong Separation
    Condition: $f_1(C)\cap f_2(C)=\emptyset$. The natural projection $\pi_C$ in this case
    is a bijection. Thus, $\cN_C=\emptyset$ and $\mu_p(\cN_C)=0$.
  \item Sierpinski gasket, $G$, is the attractor of a system of the three similitudes
    $\{f_i\}_{i\in [3]}$ (see \eqref{IFS-SG}):
  $$
    G=\bigcup_{i=1}^3 f_i(G).
  $$  
  Unlike the Cantor set, the Sierpinski gasket does not satisfy the Strong Separation Condition, because there
  is an overlap between the images $f_1(G), f_2(G),$ and $f_3(G)$:
  $$
f_i(G)\cap f_j(G)=f_i(v_j)=f_j(v_i),\quad i\neq j, \; i,j\in [3].
  $$
  The noninjectivity set $\cN_G$ is nonempty but it is countable and, thus,
  $\mu_p(\cN_G)=0$.
\item The golden gasket $\Lambda$ is the attractor of the following set of similitudes:
  $$
  f_i(x)=\lambda x+ t_i,\; i\in [3]; \quad t_1=(0,0), t_2=(1-\lambda,0),
  t_3=\left(\frac{1-\lambda}{2}, \frac{\sqrt{3}(1-\lambda)}{2}\right),
  $$
  and $\lambda=\frac{\sqrt{5}-1}{2}$ is the reciprocal of the golden ratio
  (see Figure~\ref{f.goldSG}) \cite{BMS2004, BSS-self-similar}.

  By construction, $f_{122}(\Lambda)=f_{211}(\Lambda)=\bigtriangleup C_1QC_2$
  (see Figure~\ref{f.goldSG}). Similarly, $f_{133}(\Lambda)=f_{311}(\Lambda)$
  and $f_{233}(\Lambda)=f_{322}(\Lambda)$. Thus, all points in the
  cylinders $\Lambda_{122}, \Lambda_{133}, \Lambda_{233}$ have at least two
  different symbolic representations and, therefore, $\mu_p(\cN_\Lambda)>0$.
    \end{enumerate}
\end{example}

\subsection{Isomorphic IFSs}\label{sec.iso-IFS}
We begin by specifying the type of isomorphism between two
probabilistic IFSs\footnote{Consulting Theorem~\ref{thm.SGisomorph} may help the 
  reader contextualize the material in the remainder of this subsection.}.
In the remainder of this section, we assume that all probabilistic IFSs have been completed as measure
spaces.

\begin{definition}\label{df.isomorph}
Two probabilistic IFSs
\[
(K, \cF=\{f_i\}_{i\in [k]},\nu_p)
\quad\text{and}\quad
(Q, \cG=\{g_i\}_{i\in [k]},\lambda_p)
\]
are said to be \emph{isomorphic} if there exists an isomorphism
of the complete measure spaces modulo null sets
\[
\phi:(K,\cA_{\nu_p},\nu_p)
\longrightarrow
(Q,\cA_{\lambda_p},\lambda_p)
\]
such that
\begin{equation}\label{commute-KQ}
\phi\circ f_i=g_i\circ\phi,\qquad i\in[k],
\end{equation}
modulo null sets.
\end{definition}

\begin{remark}
  In \eqref{commute-KQ}, we view $\phi$ as an equivalence class under equality almost everywhere
  rather than a particular representative defined pointwise. This is possible due to nonsingularity
  of the contractions $f_i, i\in [k]$.
  
Indeed, choose any measurable extension $\widehat\phi:K\to Q$. If two extensions differ
only on a null set $N$, their compositions with $f_i$ can differ only on
$f_i^{-1}(N)$. Thus composition with $f_i$ is well defined modulo null sets
provided $f_i$ is nonsingular, which is guaranteed by Lemma~\ref{lem.nosingular}.
  \end{remark}

\begin{theorem}\label{thm.isomorph}
  Let $(K,\mathcal F=\{f_i\}_{i\in[k]},\nu_p)$ and $(Q,\mathcal G \{g_i\}_{i\in[k]},\lambda_p)$ be two probabilistic IFSs with
  essentially injective coding. Then  $(K,\mathcal F,\nu_p)$ and $(Q,\mathcal G,\lambda_p)$
  are isomorphic.
\end{theorem}

\begin{proof}
   By Lemma~\ref{lem.Lusin}, $\pi_K$ and $\pi_Q$ and their inverses are measurable with respect to the
   corresponding completed $\sigma$-algebras.


Since both coding maps
\[
\pi_K:\Sigma\to K,
\qquad
\pi_Q:\Sigma\to Q
\]
are essentially injective and
\[
\nu_p=(\pi_K)_*\mu_p,
\qquad
\lambda_p=(\pi_Q)_*\mu_p,
\]
they define measure-space isomorphisms modulo null sets. Hence
\[
\phi=\pi_Q\circ\pi_K^{-1}
\]
defines a measure-space isomorphism from $K$ to $Q$ modulo null sets.
In particular,
\[
\phi\circ\pi_K=\pi_Q
\qquad \mu_p\text{-a.e.}
\]

Recall the coding identities
\[
\pi_K\circ\sigma_i=f_i\circ\pi_K,
\qquad
\pi_Q\circ\sigma_i=g_i\circ\pi_Q,
\qquad i\in[k].
\]
Using these identities, we obtain
\[
\begin{aligned}
(\phi\circ f_i)\circ\pi_K
&=\phi\circ\pi_K\circ\sigma_i\\
&=\pi_Q\circ\sigma_i\\
&=g_i\circ\pi_Q\\
&=(g_i\circ\phi)\circ\pi_K
\qquad \mu_p\text{-a.e.}
\end{aligned}
\]
Here the almost-everywhere identities remain valid under composition
because the maps $f_i, g_i$ and $\sigma_i, \; i\in [k]$ are nonsingular
(cf.~Lemma~\ref{lem.nosingular}). Since
$\nu_p=(\pi_K)_*\mu_p$, it follows that
\[
\phi\circ f_i=g_i\circ\phi
\qquad \nu_p\text{-a.e.},
\qquad i\in[k].
\]
Thus $\phi$ is an isomorphism of the two probabilistic IFSs.
\end{proof}

\subsection{Graphon IFS}

We are interested in establishing an isomorphism between a given
probabilistic IFS $(K, \cF, \nu_p)$ and  $(Q,\cG, \lambda_p)$  with
$Q\doteq [0,1]$. It will be used to establish a connection between
self-similar and graphon IPS in the following section.

To this end, consider the following IFS:
\begin{equation}\label{def-gi}
  g_i(x)=\frac{x}{k}+x_{i-1}, \quad i\in [k],
\end{equation}
where $x_i=\frac{i}{k}$, $i=0,\ldots,k$.

Then $Q=[0,1]$ is the attractor of $\cG=\{g_i\}_{i\in[k]}$:
\[
Q=\bigcup_{i=1}^k g_i(Q).
\]

Define the natural projection ${\pi_Q}:\Sigma\to Q$:
\begin{equation}\label{Q-natural}
\Sigma\ni w=(w_1w_2\dots) \mapsto \bigcap_{j=1}^\infty g_{w_1w_2\dots w_j} (Q).
\end{equation}

Let $\lambda_p=({\pi_Q})_\ast\mu_p$ be the pushforward of the Bernoulli
measure. On cylinders $Q_{w_1w_2\dots w_n}=g_{w_1w_2\dots w_n}(Q)$, we have
\begin{equation}\label{lambda-cyl}
  \lambda_p(Q_{w_1w_2\dots w_n})=
  \mu_p\big( {\pi_Q}^{-1} (Q_{w_1w_2\dots w_n}) \big) =p_{w_1}p_{w_2}\dots p_{w_n}.
\end{equation}
In particular, for  $p=(k^{-1}, k^{-1},\dots, k^{-1})$,
$\lambda_p$ coincides with the Lebesgue measure on $[0,1]$.

Furthermore, $\lambda_p$ is the unique stationary measure
$$
\lambda_p(A)=\sum_{i=1}^k p_i\lambda_p\big(g_i^{-1}(A)\big)\quad \forall A\in \cB(Q).
$$

We verify that  $(Q,\cG, \lambda_p)$ has essentially injective coding.
\begin{lemma}\label{lem.critical}
  \begin{equation}\label{verify-condition}
  \mu_p(\cN_Q)=0.
 \end{equation}
\end{lemma}
\begin{proof}
The noninjectivity set for ${\pi_Q}$ is generated by the overlap set
$$
\cO(Q)=\bigcup_{1\le i<j\le k} \left(g_i(Q)\cap
  g_j(Q)\right)=\left\{\frac{1}{k}, \frac{2}{k},\dots, \frac{k-1}{k} \right\}.
$$
From this we define the critical set
$$
\cC(Q)={\pi_Q}^{-1}\big(\cO(Q)\big) =\{ i\dot{k}, (i+1)\dot{1}:\;
i\in [k-1]\},
$$
where $\dot i $ stands for $(ii\dots)$.

The noninjectivity set of ${\pi_Q}$ is then given by
$$
\cN_Q=\bigcup_{w\in\Sigma^\ast}\bigcup_{\tau\in \cC(Q)} (w\tau).
$$
Since $\cN_Q$ is countable and $\mu_p$ is nonatomic, we have $\mu_p(\cN_Q)=0$.
\end{proof}

Next, we illustrate the construction of the isomorphism between two IFSs in the proof
of Theorem~\ref{thm.isomorph} with the following example.

\begin{theorem}\label{thm.SGisomorph}
Let $(G, \cF=\{f_1, f_2, f_3\},\nu)$ be SG and $(Q,\cG=\{g_1, g_2,
g_3\},\lambda)$ be a graphon IFS defined by \eqref{def-gi}. We assume
that $p=(1/3, 1/3, 1/3)$ so that $\nu$
is the normalized $\left(\frac{\log 3}{\log 2}\right)$-dimensional Hausdorff measure
(cf.~Section~\ref{sec.selfsim-measures}) and
$\lambda$ is the Lebesgue measure on the unit
interval.
Then $(G, \cF,\nu)$ and $(Q,\cG,\lambda)$ are isomorphic.
\end{theorem}

\begin{proof}
Recall that $G$ is  the attractor of the three contracting similitudes
$$
f_i(x)=\frac{1}{2}(x-v_i)+v_i, i\in [3]
$$
(see Example~\ref{ex.fractals}). The overlap set for $G$ is
$$
\cO(G)=\{ f_1(v_2)=f_2(v_1), f_1(v_3)=f_3(v_1), f_2(v_3)=f_3(v_2)\}.
$$
The critical set is
$$
\cC(G)=\pi^{-1}_G(\cO(G))=\{(1\dot{2}), (2\dot{1}), ( 3\dot{1}), ( 1\dot{3}), ( 2\dot{3}), ( 3\dot{2})\}.
$$
The noninjectivity set of $G$ is
$$
\cN_G=\bigcup_{w\in\Sigma^\ast}
\{(w1\dot{2}), (w2\dot{1}), (w3\dot{1}), (w1\dot{3}), (w2\dot{3}), ( w3\dot{2})\}.
$$

Next we turn to $(Q=[0,1], \{g_i\}_{i\in[3]}, \lambda)$, where
$$
g_i(x)=\frac{1}{3}x+x_{i-1}, \quad i\in [3], 
$$
with $x_j=\frac{j}{3}, j=0,1,2,3.$ $\lambda$ is the Lebesgue measure on $[0,1]$.
The noninjectivity set of $Q$ is
$$
\cN_Q=\bigcup_{w\in\Sigma^\ast} \{(w1\dot{3}), (w2\dot{1}), (w2\dot{3}), ( w3\dot{1})\}.
$$

Clearly, $\mu_p(\cN_G)=\mu_p(\cN_Q)=0$.
By Theorem~\ref{thm.isomorph}, $(G,\cF,\nu)$ and
$(Q,\cG,\lambda)$ are isomorphic.
\end{proof}

It is instructive to represent the attractor of a probabilistic IFS
$(K,\mathcal{F}, \nu_p)$ by a rooted tree. The node set of the tree is
given by $\Sigma^*$, the set of finite words over the alphabet $[k]$ (cf.~\ref{Sigma-star}).
The adjacency is defined by the natural relation: $w\in\Sigma_n$ is a parent for all
$(wi)\in\Sigma_{n+1}$ with $i\in [k]$. The first two levels of the tree corresponding to the Sierpinski gasket
are illustrated in Fig.~\ref{f.tree}.

The isomorphic IFSs  $(G,\mathcal{F}, \nu_p)$ and $(Q,\mathcal{G}, \lambda_p)$ share the same tree.
Furthermore, at each node $w\in\Sigma_n$, $\nu_p(K_w)=\lambda_p(Q_w)=\mu_p([w])$
(cf.~\eqref{Bernoulli-m}).

\begin{figure}
  \begin{center}
\begin{tikzpicture}[
  root/.style  = {circle, draw, minimum size=9mm, inner sep=1pt, font=\small\bfseries},
  inner/.style = {circle, draw, minimum size=8mm, inner sep=1pt, font=\small},
   leaf/.style = {circle, draw, minimum size=8mm, inner sep=1pt, font=\small},
  edge from parent/.style = {draw},
  level distance = 18mm,
  level 1/.style = {sibling distance=50mm},
  level 2/.style = {sibling distance=15mm}
]

\node[root]{$\emptyset$}
  child { node[inner]{$(1)$}
    child { node[leaf]{$(11)$} }
    child { node[leaf]{$(12)$} }
    child { node[leaf]{$(13)$} }
  }
  child { node[inner]{$(2)$}
    child { node[leaf]{$(21)$} }
    child { node[leaf]{$(22)$} }
    child { node[leaf]{$(23)$} }
  }
  child { node[inner]{$(3)$}
    child { node[leaf]{$(31)$} }
    child { node[leaf]{$(32)$} }
    child { node[leaf]{$(33)$} }
  };
\end{tikzpicture}
\end{center}
\caption{The root and first two levels of the tree representing the Sierpinski gasket.}\label{f.tree}
\end{figure}

\subsection{Isomorphic function spaces}\label{sec.iso-functions}

Having constructed an isomorphic IFS 
\((Q,\mathcal{G},\lambda_p)\), we now establish an isomorphism between
\(L^1(K,\nu_p)\) and \(L^1(Q,\lambda_p)\).
Since \(p\) is fixed, we henceforth omit it from the notation of the measures
\(\mu\), \(\nu\), and \(\lambda\).

In the remainder of this subsection, we focus on the relationship between
\(L^1(K,\nu)\) and \(L^1(Q,\lambda)\).
Throughout, it is understood that \((K,\mathcal{F},\nu)\) and
\((Q,\mathcal{G},\lambda)\) are isomorphic probabilistic IFSs that have essentially injective coding.

Recall that \(\Sigma=[k]^{\mathbb{N}}\) is the symbolic space endowed with the Bernoulli
measure \(\mu\).

Denote by
\[
\pi_K:\Sigma\to K,
\qquad
\pi_Q:\Sigma\to Q
\]
the natural projections associated with the IFSs \(\mathcal F\) and
\(\mathcal G\), respectively, and let
\[
\nu=(\pi_K)_*\mu,
\qquad
\lambda=(\pi_Q)_*\mu.
\]

We define the pullback operators
\[
U_K:L^1(K,\nu)\to L^1(\Sigma,\mu),
\qquad
U_Q:L^1(Q,\lambda)\to L^1(\Sigma,\mu),
\]
by
\[
U_K f = f\circ \pi_K,
\qquad
U_Q g = g\circ \pi_Q .
\]
Since \(\pi_K\) and \(\pi_Q\) are measure-preserving, both operators are linear
isometries onto their respective ranges. Here and throughout, the term
\emph{isometry} is understood in the $L^1$-sense, namely as a linear map
preserving the \(L^1\)-norm. If one works in an \(L^2\) setting instead, the isometries
preserve the inner product as well.
\begin{figure}
  \begin{center}
    \begin{tikzpicture}[>=stealth, node distance=3cm, every node/.style={font=\large}]

\node (Sigma) at (0,3) {$\Sigma$};
\node (K) at (-3,0) {$K$};
\node (Q) at (3,0) {$Q$};

\draw[->, thick] (Sigma) -- (K) node[midway, left] {$\pi_K$};
\draw[->, thick] (Sigma) -- (Q) node[midway, right] {$\pi_Q$};

\draw[->, dashed, bend left=45] (K) to node[midway, left] {$U_Kf=f\circ\pi_K$} (Sigma);
\draw[->, dashed, bend right=45] (Q) to node[midway, right] {$U_Qg=g\circ\pi_Q$} (Sigma);
\draw[->, dashed, bend right=30] (K) to node[midway, below] {$T = U_Q^{-1}U_K$} (Q);
\end{tikzpicture}
   
\end{center}
\caption{The isomorphism between measure spaces $(K,\nu)$ and $(Q,\lambda)$ induces a
  corresponding isomorphism between the function spaces
  $L^1(K,\nu)$ and $L^1(Q,\lambda)$.}
\label{f.diagram}
\end{figure}

Since \((K,\mathcal{F},\nu)\) and  \((Q,\mathcal{G},\lambda)\) have essentially injective coding,
$\pi_K$ and $\pi_Q$ are bijective modulo null sets. Consequently,
$U_K$ and $U_Q$ are surjective and hence define linear isometric
isomorphisms
\[
U_K:L^1(K,\nu)\to L^1(\Sigma,\mu),
\qquad
U_Q:L^1(Q,\lambda)\to L^1(\Sigma,\mu).
\]

In particular, the inverse of $U_Q$ is well-defined modulo null sets by
\begin{equation}\label{def-inverse-U}
U_Q^{-1}F=F\circ\pi_Q^{-1},
\qquad F\in L^1(\Sigma,\mu).
\end{equation}

This allows us to define the canonical operator
\begin{equation}\label{def-T}
T:L^1(K,\nu)\longrightarrow L^1(Q,\lambda),
\qquad
Tf := U_Q^{-1}(U_K f).
\end{equation}
Equivalently, \(Tf\) is the unique (up to \(\lambda\)-null sets) function satisfying
\begin{equation}\label{Tf-conj}
f\circ\pi_K = (Tf)\circ\pi_Q
\quad\mu\text{-a.e.}.
\end{equation}

By construction, \(T\) is a surjective linear isometry preserving positivity and
integrals:
\[
\int_K f\,d\nu = \int_Q Tf\,d\lambda,
\qquad
f\in L^1(K,\nu).
\]

We emphasize that this identification is canonical up to null sets and depends
only on the symbolic representations induced by the IFSs. In particular, when the
maps $f_i$ are similitudes and $\nu$ is the natural self-similar
measure with $p=(k^{-1},k^{-1},\dots, k^{-1})$, the
operator $T$ provides a concrete realization of functions on $K$ as functions on
the unit interval $Q$ with respect to Lebesgue measure.

The correspondence extends naturally to product spaces, yielding an isometric
isomorphism
\[
  L^1(K\times K,\nu\times\nu)
\cong
L^1(Q\times Q,\lambda\times\lambda).
\]
In particular, interaction kernels on $K\times K$ can be represented as
kernels on $Q\times Q$. When $Q=[0,1]$ and $\lambda$ is Lebesgue
measure, this provides the corresponding graphon representation.

\begin{figure}
 	\centering
  \textbf{a}\includegraphics[width =.45\textwidth]{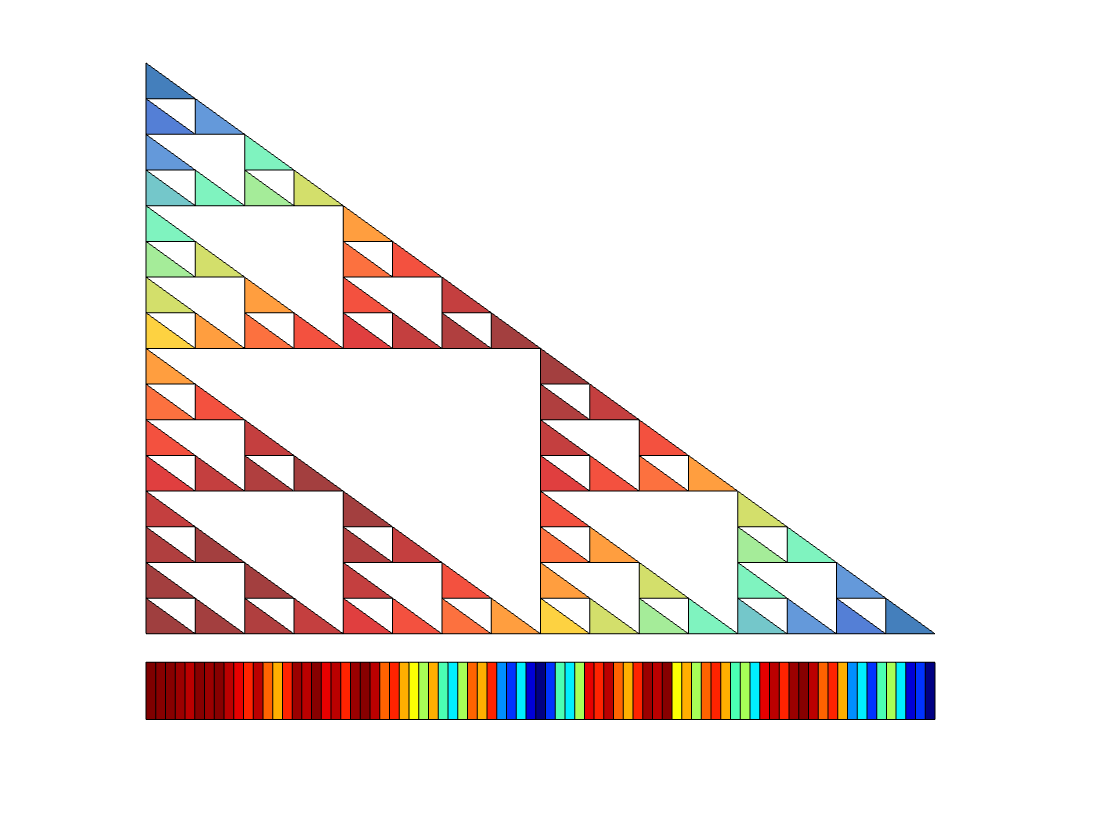}
  \textbf{b} \includegraphics[width =.45\textwidth]{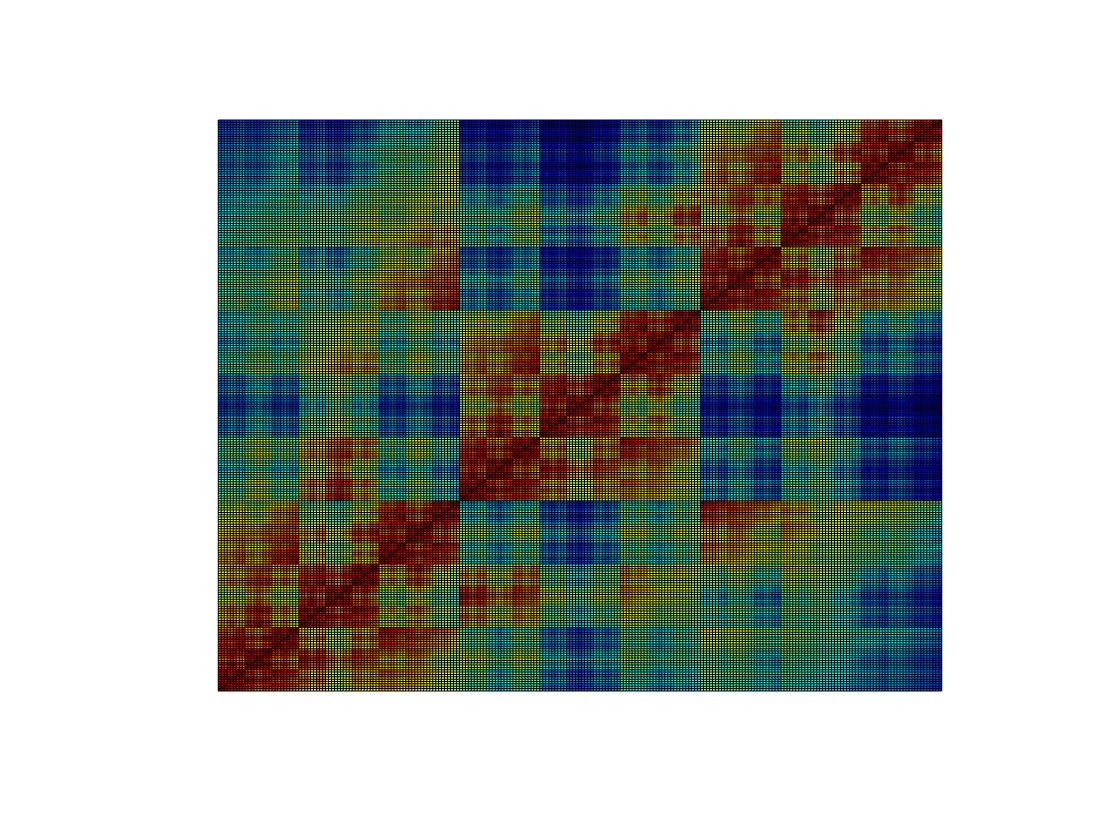}\\
  \caption{ \textbf{a}. The heat map of the function $f(x)=e^{-|x_1-x_2|}$ on the Sierpinski gasket
    $G\subset\R^2$ (above) and its representation as a function on $[0,1]$ (below).
    \textbf{b}. Graphon representation of $W(x,y)=e^{-\|x-y\|}$, $(x,y)\in G\times G\subset\R^2\times\R^2$.
   }\label{f.graphon}
 \end{figure}

\subsection{Martingale representation}

In this subsection, we use martingales to compute \(Tf\in L^1(Q,\lambda)\)
for a given \(f\in L^1(K,\nu)\) directly from cylinder averages. Besides
providing an alternative description of the isomorphism, this construction is
useful for numerical approximation.

For \(w\in\Sigma_n\), set
\[
p_w=\mu([w])=\nu(K_w)=\lambda(Q_w).
\]
The cylinder families \(\{K_w:w\in\Sigma_n\}\) and
\(\{Q_w:w\in\Sigma_n\}\) form measurable partitions modulo null sets. Let
\[
\mathcal K_n=\sigma(\{K_w:w\in\Sigma_n\}),
\qquad
\mathcal Q_n=\sigma(\{Q_w:w\in\Sigma_n\}).
\]
For \(f\in L^1(K,\nu)\), define
\[
f_n
:=
\mathbb E(f\mid\mathcal K_n)
=
\sum_{w\in\Sigma_n}\nu_w(f)\,\mathbf1_{K_w},
\qquad
\nu_w(f)
:=
\frac{1}{\nu(K_w)}\int_{K_w}f\,d\nu.
\]
Then \((f_n)\) is a martingale.
Furthermore, since $\pi_K$ is an isomorphism modulo null sets and the cylinder $\sigma$-algebras
generate the Borel $\sigma$-algebra on $\Sigma$, the filtration \((\mathcal K_n)\) generates $\cA_\nu$
modulo null sets. Thus, by the Martingale Convergence Theorem \cite{Stroock-Prob},
\[
f_n\to f
\quad\text{in }L^1(K,\nu)\text{ and }\nu\text{-a.e.}.
\]

Define
\begin{equation}\label{def-ftilde-n}
\widetilde f_n
:=
\sum_{w\in\Sigma_n}\nu_w(f)\,\mathbf1_{Q_w}.
\end{equation}

\begin{theorem}\label{thm.ftilde}
The sequence \((\widetilde f_n)\) converges in \(L^1(Q,\lambda)\) and
\(\lambda\)-a.e. to \(Tf\).
\end{theorem}

\begin{proof}
First, we show that \((\widetilde f_n)\) is a martingale with respect to
\((\mathcal Q_n)\). Indeed, for $w\in\Sigma_n$,
the conditional expectation is constant on $Q_w$, and
\eqref{def-ftilde-n} yields
\begin{align*}
  \mathbb E(\widetilde f_{n+1}\mid\mathcal Q_n)
  &= \frac{1}{\lambda(Q_w)} \int_{Q_w} \widetilde f_{n+1}\, d\lambda\\
&=
\frac{1}{\lambda(Q_w)}
\sum_{s\in[k]}\nu_{ws}(f)\lambda(Q_{ws})\\
&=
\frac{1}{\nu(K_w)}
\sum_{s\in[k]}\int_{K_{ws}}f\,d\nu\\
&=
\nu_w(f),
\end{align*}
where the cylinder sets \(K_{ws}\), \(s\in[k]\), partition \(K_w\)
modulo \(\nu\)-null sets. Thus,
\[
\mathbb E(\widetilde f_{n+1}\mid\mathcal Q_n)=\widetilde f_n.
\]

Next, the isomorphism \(T\) sends cylinder indicators to cylinder indicators
modulo null sets:
\begin{equation}\label{indicator}
T\mathbf1_{K_w}=\mathbf1_{Q_w}
\quad\lambda\text{-a.e.},
\qquad w\in\Sigma^*.
\end{equation}
Indeed, \eqref{Tf-conj} and the coding identities imply
\[
\mathbf1_{K_w}\circ\pi_K
=
\mathbf1_{[w]}
=
\mathbf1_{Q_w}\circ\pi_Q
\quad\mu\text{-a.e.}.
\]
By linearity,
\[
\widetilde f_n=Tf_n
\quad\lambda\text{-a.e.}
\]
for every \(n\). Since \(T\) is an \(L^1\)-isometry and
\(f_n\to f\) in \(L^1(K,\nu)\),
\[
\|\widetilde f_n-Tf\|_{L^1(Q,\lambda)}
=
\|T(f_n-f)\|_{L^1(Q,\lambda)}
=
\|f_n-f\|_{L^1(K,\nu)}
\longrightarrow0.
\]
Hence \(\widetilde f_n\to Tf\) in \(L^1(Q,\lambda)\). Since
\((\widetilde f_n)\) is a martingale, its almost-everywhere martingale limit
must agree with the \(L^1\)-limit, and therefore
\[
\widetilde f_n\to Tf
\quad\lambda\text{-a.e.}
\]
as well.
\end{proof}

Under this identification, \(f\) and \(Tf\) have the same cylinder averages:
\[
\nu_w(f)=\lambda_w(Tf),
\qquad w\in\Sigma^*.
\]
The same construction applied to
\(W\in L^1(K\times K,\nu\times\nu)\) gives
\(\widetilde W\in L^1(Q\times Q,\lambda\times\lambda)\) with
\[
\widetilde W_n(x,y)
=
\sum_{w,v\in\Sigma_n}
\left(
\fint_{K_w\times K_v}W\,d(\nu\times\nu)
\right)
\mathbf1_{Q_w\times Q_v}(x,y),
\]
and
\[
\widetilde W_n\longrightarrow\widetilde W
\quad\text{in }L^1(Q\times Q,\lambda\times\lambda).
\]
This realizes explicitly the isomorphism between
\(L^1(K\times K,\nu\times\nu)\) and
\(L^1(Q\times Q,\lambda\times\lambda)\).

\section{Convergence of self-similar IPS}\label{sec.converge-self-sim}
\setcounter{equation}{0}

The goal of this section is to prove Theorems~\ref{thm.ss-heat} and
\ref{thm.ss-Vlasov}. To this end, we use the results of the
previous section to construct a graphon IPS that is isomorphic to the self-similar IPS \eqref{ss-KM}.
We then use this isomorphism, together with existing results for graphon IPS 
\cite{Med14a, Med19, KVMed18}, to justify the continuum limit \eqref{fKM} and the Vlasov equation
\eqref{ss-Vlasov}.

Since relating self-similar IPSs to their graphon counterparts is our main
objective, throughout the remainder of this paper we assume, without further
comment, that
\[
p=\left(k^{-1},k^{-1},\dots,k^{-1}\right).
\]
When the maps \(f_i\) are similitudes, this choice yields the natural
self-similar measure on \(K\).
On \(Q\), it always results in the restriction of the Lebesgue
measure, so that the space of graphons is naturally identified with a subset of
\(L^1(Q\times Q,\lambda\times\lambda)\).

  \subsection{Isomorphic IPS} \label{sec.iso-IPS}

  We now return  to the self-similar IPS \eqref{fKM} on fractal $K$.
  Recall that $K$ is an attractor of a  probabilistic IFS $(K, \{f_i\}_{i\in [k]}, \nu)$.
  Following the recipe of \S~\ref{sec.iso-IFS}, we construct an isomorphic IFS
  $(Q, \{g_i\}_{i\in [k]}, \lambda)$. To this end, let
  \begin{equation}\label{Q-to-K}
    \phi: K \to Q
  \end{equation}
  denote the measure--preserving bijection such that
 $$
\phi\circ f_i = g_i\circ \phi, \quad i\in [k]
$$
modulo null sets.

 Next, we construct a graphon IPS on the unit interval $Q$ that is isomorphic to the
  self-similar IPS \eqref{fKM} on the fractal $K$. To this end, we define
$\tilde W: Q\times Q\to \R$ and $\tilde g: Q\to\R$ as follows
  \begin{align}\label{def-W-tilde}
    \tilde W(x,y)=& W\left(\phi^{-1}(x), \phi^{-1}(y)\right), \;
                                    (x,y)\in  Q\times Q,\\
                       \label{def-g-tilde}
                       \tilde g(x)& = g\left(\phi^{-1}(x)\right), \;x\in Q.
  \end{align}
  With these definitions in place, we consider the following graphon IPS
\begin{align}\label{tilde-KM}
  \dot{\tilde u}_w &= f(t, \tilde u_w) + \sum_{|v|=n} \tilde W_{wv} D(\tilde u_w, \tilde u_v)\lambda(Q_v),\quad
             w\in \Sigma_n,\\
  \label{tilde-KM-ic}
             \tilde u_w(0) &= \tilde g_w, \quad  w\in \Sigma_n,
\end{align}
where 
$$
\tilde W_{wv}\doteq\fint_{Q_w\times Q_v} \tilde W(x,y) d(\lambda\times\lambda)(x,y),\qquad
\tilde g_w\doteq\fint_{Q_w} \tilde g(x) d\lambda(x).
$$

By construction,
 \begin{align*}
   \tilde W_{wv}&= \frac{1}{\lambda(Q_w) \lambda(Q_v)}
                              \int_{Q_w\times Q_v} \tilde W(x,y) d(\lambda\times\lambda)(x,y)\\
                       &=    \frac{1}{\lambda(Q_w) \lambda(Q_v)}
                              \int_{Q_w\times Q_v} W\left(\phi^{-1}(x),\phi^{-1}(y)\right) d(\lambda\times\lambda)(x,y)\\
         &=        \frac{1}{\nu(K_w) \nu(K_v)} \int_{K_w\times K_v} W(x,y) d (\nu\times\nu)(x,y) \\
    &=  W_{wv}.
  \end{align*}
  Likewise,
  \begin{equation*}
    \tilde g_{w}= g_w.
  \end{equation*}

Thus, under the identification induced by \(\phi\), the self-similar
IPS \eqref{ss-KM} coincides with
the graphon IPS \eqref{tilde-KM} associated with \(\widetilde W\).


\subsection{The proofs of Theorems~\ref{thm.ss-heat} and \ref{thm.ss-Vlasov}}\label{sec.proofs}

With the graphon IPS \eqref{tilde-KM} at hand, the proofs of Theorems~\ref{thm.ss-heat} and
\ref{thm.ss-Vlasov} follow from Theorem~\ref{thm.heat} and
Theorem~\ref{thm.Vlasov}, respectively.

Specifically, the correspondence between the kernels and initial data (cf.~\eqref{def-W-tilde} and
\eqref{def-g-tilde}) of the self-similar IPS \eqref{ss-heat}, \eqref{ss-heat-ic} and graphon
IPS \eqref{cKM} with $W:=\tilde W$ and initial data $\tilde g$ translates into the following relation
between the solutions of the two  IVPs:
\begin{equation}\label{relate-u-tilde-u}
 u(t,x)= \tilde u\left(t,\phi (x)\right), \qquad  \nu-\mbox{a.e.}\;x \in K.
\end{equation}

From this we have,
\begin{align*}
    \| u^n(t,\cdot)-u(t,\cdot) \|_{L^2(K,\nu)} &=
                                                 \| \tilde u^n(t,\phi(\cdot))-\tilde u(t,\phi(\cdot)) \|_{L^2(K,\nu)}\\
                                               &= \| \tilde u^n(t,\cdot)-\tilde u(t,\cdot) \|_{L^2(Q,\lambda)}
                                                 \rightarrow 0,
  \end{align*}
where we used Theorem~\ref{thm.heat} in the last step. This proves Theorem~\ref{thm.ss-heat}.

We proceed similarly to establish the relation between the empirical 
measures $\mathfrak{m}^x_{m,\ell,t}$ and $\mathfrak{m}^x_{t}$ generated by the evolution
of the self-similar IPS \eqref{ss-KM}, \eqref{ss-KM-ic} and its
mean-field
limit \eqref{ss-Vlasov} on one hand
and those generated by the graphon IPS \eqref{cKM} with $W:=\tilde W$ and initial data $\tilde g$
and the corresponding continuum limit \eqref{Vlasov}.
The latter are denoted by $\tilde{\mathfrak{m}}^x_{m,\ell,t}$ and $\tilde{\mathfrak{m}}^x_{t}$ respectively.
Recall that
$$
\tilde{\mathfrak{m}}^x_{t}(B)=\int_B\tilde\rho(t,x,u)du,
$$
where $\tilde\rho$ satisfies the Vlasov equation \eqref{Vlasov} with $W:=\tilde W$.

Using \eqref{Q-to-K}, for every Borel set \(B\in\mathcal B(\mathbb R^k)\) we have
$$
\mathfrak{m}^x_{m,\ell,t}(B)= \tilde{\mathfrak{m}}^{\phi(x)}_{m,\ell,t}(B) \quad\mbox{and}\quad
\mathfrak{m}^x_{t}(B)=\tilde{\mathfrak{m}}^{\phi(x)}_{t}(B)\qquad
\forall m,\ell\in\Z, \; t\in\R^+ 
$$
for any Borel set $B\in\mathcal{B}(\R^k)$ and $\nu-$a.e.  $x\in K$. From this, it follows
\begin{align*}
 \int_K d_{BL} (\mathfrak{m}^x_{m,\ell,t}, \mathfrak{m}^x_t) d\nu(x) & =
                                    \int_{K} d_{BL} \left(  \tilde{\mathfrak{m}}^{\phi(x)}_{m,\ell,t},
                                    \tilde{\mathfrak{m}}^{\phi(x)}_t \right) d\nu\left(x\right)\\
  &= \int_{Q} d_{BL} \left(\tilde{\mathfrak{m}}^y_{m,\ell,t},
  \tilde{\mathfrak{m}}^y_t\right) d\lambda(y),
\end{align*}
where,  in the last line, we used $\lambda=\phi_\ast\nu$.
By Theorem~\ref{thm.Vlasov},
\(\sup_{t\in [0,T]} \int_{Q} d_{BL} (\tilde{\mathfrak{m}}^y_{m,\ell,t}, \tilde{\mathfrak{m}}^y_t) d \lambda(y)\)
can be made less than any given $\epsilon>0$ provided $m$ and $\ell$ are large enough.
This proves Theorem~\ref{thm.ss-Vlasov} .
 
  \section{The rate of convergence} \label{sec.rate}
  \setcounter{equation}{0}

 In this section, we develop a method for estimating the rate of convergence
of the discrete models \eqref{ss-KM} and \eqref{ss-KM-ic} on self-similar
networks. By Theorem~\ref{thm.ss-heat}, the problem reduces to estimating
the accuracy with which the graphon
$W \in L^2(K \times K, \nu \times \nu)$ and the initial data
$g \in L^2(K, \nu)$ are approximated by their $L^2$-projections onto
finite-dimensional subspaces of step functions on $K \times K$ and $K$,
respectively:
\begin{equation*}
\|W - W^n\|_{L^2(K \times K, \nu \times \nu)}, \quad
\|g - g^n\|_{L^2(K, \nu)}.
\end{equation*}
The isomorphism between self-similar IPS and its graphon counterpart that
was used to establish convergence of the discrete problems does not allow
one to translate the information about convergence rates from graphon
systems to self-similar networks. Therefore, in this section, we study the rate of convergence for
the latter class of models directly.

For $\phi\in L^p(Q)$, a function on a Euclidean domain, the error of approximation by the $L^2$-projection
onto the subspace of step functions is determined by the $L^p$-Lipschitz regularity of
$\phi\in\operatorname{Lip}\left(\alpha,  L^p(Q)\right), p\ge 1, 0<\alpha\le 1$ (cf.~Lemma~\ref{lem.rate}).
In the remainder of this section, we extend the definition of the generalized Lipschitz spaces to self-similar
domains and use them to estimate the convergence rate for self-similar IPS.

Throughout this section, we assume that $K$ is the attractor of an IFS
$\cF=\{f_i\}_{i\in[k]}$ consisting of contracting similitudes with the
same contraction ratio $\lambda\in(0,1)$:
\[
  |f_i(x)-f_i(y)|=\lambda |x-y|,
  \qquad x,y\in\mathbb R^d,\quad i\in[k].
\]
In addition, we assume that for every $i,j\in[k]$ there exists
$\tau_{ij}\in\mathbb R^d$ such that
\begin{equation}\label{translation-assumption}
  K_{wj}=K_{wi}+\lambda^{|w|}\tau_{ij},
  \qquad w\in\Sigma^*.
\end{equation}
Further, let $\nu$ be the natural self-similar measure on $K$, i.e.,
$p_1=\cdots=p_k=k^{-1}$ in \eqref{nu-cyl}.

\begin{remark}\label{rem.translate}
Assumption~\eqref{translation-assumption} holds, in particular, if
$K$ is the attractor of an IFS $\cF=\{f_i\}_{i\in[k]}$ of the form
\begin{equation}\label{eq.common-linear-part}
  f_i(x)=\lambda x+b_i,\qquad i\in[k],
\end{equation}
where $0<\lambda<1$ and $b_i\in\mathbb R^d$. In this case,
$\tau_{ij}=b_j-b_i$.
\end{remark}

\begin{figure}
\begin{center}
\begin{tikzpicture}[scale=3]
  \tkzDefPoint(-1,0){A}
  \tkzDefPoint(1,0){B}
  \tkzDefPoint(0,sqrt(3)){C}

  \tkzDefPoint(-1/2,sqrt(3)/2){E}
  \tkzDefPoint(1/2,sqrt{3}/2){F}
  \tkzDefPoint(0,0){G}

  \tkzDefPoint(-1/2,1/4){O}
  \tkzDefPoint(0,3*sqrt(3)/4){S}
  \tkzDefPoint(1/2, 1/4){T}
  
  \tkzDrawPolygon[thick, blue](A,B,C)
  \tkzDrawPolygon[thick, blue](E,F,G)

  \tkzDrawSegment[->,thick, black](O,S)
  \tkzDrawSegment[->,thick, black](O,T)

 \tkzLabelPoint[above](S){$\tau_{12}$}
 \tkzLabelPoint[below](T){$\tau_{13}$}
\end{tikzpicture}
\end{center}
\caption{Vectors $\tau_{12}$ and $\tau_{13}$, which  are used in the definition of
  the modulus of continuity (see Definition~\ref{df.modulus}) when applied to functions on SG.}
\label{f.arrows}
\end{figure}

Generalized Lipschitz spaces were the key to estimating the rate of convergence
for graphon IPSs (see Theorem~\ref{thm.rate}). Naturally, we want to extend
these spaces to fractal domains. The first step is to identify an appropriate
analog of the modulus of continuity in the fractal setting.

For a Euclidean domain $Q$, the modulus of continuity measures the difference
between $f \in L^p(Q)$ and its translation $f(\cdot + h)$ for small $|h|$
(cf.~\eqref{modulus}). When attempting to extend this notion to the fractal
setting, we immediately encounter the following difficulty: given $x \in K$
and $h \in \mathbb{R}^d$, the point $x + h$ need not belong to $K$, even for
arbitrarily small $|h|$. Consequently, the first step is to identify an
appropriate analog of translation on $K$.

For $i,j\in[k]$, $i\ne j$, let $\tau_{ij}$ be the translation vector
from \eqref{translation-assumption}, so that
\begin{equation}\label{tau}
  f_j(K)=f_i(K)+\tau_{ij}.
\end{equation}
See Figure~\ref{f.arrows} for an illustration of the vectors $\tau_{ij}$
for the Sierpinski gasket.

With $(\tau_{ij})$ in hand, we are now prepared to define the
modulus of continuity at level $m \in \mathbb{N}$.

\begin{definition}\label{df.modulus}
For $\phi \in L^p(K,\nu)$, define the $L^p$-modulus of continuity at level $m\in\N$ as follows
\[
\omega_p(\phi, m)
=
\sup_{\ell \ge m}
\max_{i \neq j}
\left\{
\|\phi(\cdot + \tau) - \phi(\cdot)\|_{L^p(\tilde{K}_{\tau}, \nu)}
:\;
\tau = \lambda^{\ell} \tau_{ij},\ i,j \in [k]
\right\},
\]
where
\[
\tilde{K}_{\tau} = \{ x \in K : x + \tau \in K \}.
\]
\end{definition}

The generalized Lipschitz spaces are now defined in direct analogy to the Euclidean case.

\begin{definition}\label{def.genLip}
For $\alpha \in (0,1]$, define the generalized Lipschitz space
\[
\operatorname{Lip}\bigl(\alpha, L^p(K,\nu)\bigr)
=
\left\{
\phi \in L^p(K,\nu) :
\exists\, C > 0 \text{ such that }
\omega_p(\phi, m) \le C \lambda^{\alpha m}
\right\},
\]
equipped with the seminorm
\[
\|\phi\|_{\operatorname{Lip}(\alpha, L^p(K,\nu))}
\doteq
\sup_{m \in \mathbb{N}} \lambda^{-\alpha m} \, \omega_p(\phi, m).
\]
\end{definition}
\begin{remark}\label{rem.Nik-Bes}
This definition provides a natural generalization of Nikolskii--Besov
spaces to the fractal setting (cf.~\cite{BKP2019}).
\end{remark}

The following lemma yields the error of approximation of an
\(L^p\)-function on \(K\) by its cell-average approximation on
\(m\)-cylinders.

\begin{lemma}\label{lem.ss-rate}
  Let $\{f_i\}_{i\in[k]}$ be a system of contracting similitudes with attractor $K$ equipped with
  \textit{natural} self-similar measure $\nu$ satisfying \eqref{nu-cyl} and consider
  $\phi\in \operatorname{Lip}\left(\alpha, L^p(K,\nu)\right)$.
  Then
  \be\lbl{Lp-rate}
  \|\phi^m-\phi\|_{L^p(K,\nu)} \le
  \frac{ k^{1/p-1} \left( k-1 \right)^{1/p} }{1-\lambda^\alpha}
  \|\phi\|_{\Lip\left(L^p(K,\nu)\right)} \lambda^{\alpha m}.
  \ee
\end{lemma}
\begin{remark}\label{rem.product}
 To apply Lemma~\ref{lem.ss-rate} to
  $W\in \operatorname{Lip}\left(\alpha, L^2(K\times K,\nu\times\nu)\right)$,
  we note that $K\times K$ is the attractor of contracting similitudes 
  $\{f_{ij}=(f_i, f_j)\}_{i,j\in [k]}$ with the same contraction constant $\lambda$.
  Therefore, \eqref{Lp-rate} applies to $W\in L^p(K\times K,\nu\times \nu)$ after replacing $k$ by
  $k^2$.
  \end{remark}

The combination of Lemma~\ref{lem.ss-rate} and Theorem~\ref{thm.ss-heat} yields
the following theorem.

\begin{theorem}\label{thm.ss-rate}
  Let $u(t,x)$ be the solution of the IVP for \eqref{ss-KM} subject to
  $u(0,\cdot)=g\in \operatorname{Lip}\left( \alpha_1,
    L^2(K,\nu)\right)$ and let $u^n$ and $\mathsf{u}^n$ be
the solutions of the discrete problems, as described in
  Theorem~\ref{thm.ss-heat}.
If $W\in \operatorname{Lip}\left( \alpha_2, L^2(K\times K,\nu\times\nu)\right)$
then
\begin{equation}\label{ss-rate}
  \|u-v\|_{C(0,T; L^2(K,\nu))}\le C \lambda^{\alpha n},
  \quad v\in\{ u^n, \mathsf{u}^n\},\; \alpha=\min\{\alpha_1, \alpha_2\}.
\end{equation}
Here, in the case of $v=\mathsf{u}^n$ estimate \eqref{ss-rate} holds almost surely.
\end{theorem}

It remains to prove  Lemma~\ref{lem.ss-rate}.
\begin{proof}[Proof of Lemma~\ref{lem.ss-rate}.]
Fix $m\ge 1$ and partition $K$ into $k^m$ subsets
$$
K_{w}=f_{w}(K), \quad w \in \Sigma_m.
$$

We represent $\phi^{m+1}$ as 
  \be\lbl{phi-m+1}
  \phi^{m+1}=\sum_{|w|=m} \sum_{j\in [k]} \phi_{wj} \1_{K_{wj}},
  \ee
  where
  $$
  \phi_{wj} = \fint_{K_{wj}} \phi d\nu.
  $$
  
  Likewise,
  \begin{align}
    \nonumber
    \phi^{m}& =\sum_{|w|=m} \sum_{j\in [k]} \phi_{w} \1_{K_{wj}}\\
    \lbl{phi-m}
            &= k^{-1} \sum_{|w|=m} \sum_{j\in [k]} \sum_{l\in [k]}\phi_{wl} \1_{K_{wj}},
  \end{align}
  where we used $\nu(K_{wj})=k^{-1} \nu(K_{w})$, which follows  (cf.~\ref{nu-cyl}) with
  $p_i=k^{-1}, i\in [k]$.

              By subtracting \eqref{phi-m+1} from \eqref{phi-m}, we have
              \begin{align}\nonumber
              k \left(   \phi^m(x)-\phi^{m+1}(x)\right)& =\sum_{|w|=m}\sum_{j\in [k]}\sum_{l\neq j} \left(
                                                         \fint_{K_{wj}}-\fint_{K_{wl}}\right) \phi(z) d\nu(z) \,\1_{K_{wj}}(x)\\
                \label{represent}
                   & =  \sum_{|w|=m}\sum_{j\in [k]} \sum_{l\neq j}
               \fint_{K_{wj}} \left[ \phi(y+\tau^m_{jl})- \phi(y) \right] d\nu(y) \,\1_{K_{wj}}(x),
 \end{align}
where             
$$
\tau^m_{ij}=\lambda^m\tau_{ij},\quad i,j\in [k], \; i\neq j.
$$
After raising both sides of \eqref{represent} to the $p$th power and integrating over $K$, we have
\begin{align*}
k^p \int |\phi^m-\phi^{m+1}|^p d\nu & =   \sum_{w,j,l}
  \left|\fint_{K_{wj}} \left[ \phi(y+\tau^m_{jl})- \phi(y) \right] d\nu(y) \right|^p  \nu(K_{wj})\\
& \le  \sum_{j,l\in [k]} \sum_{w}
                          \int_{K_{wj}} \left| \phi(y+\tau^m_{jl})- \phi(y) \right|^p d\nu(y) \\
 &\le   k(k-1)\omega_p^p(\phi, m).
\end{align*}
From this we conclude
$$
\|\phi^m-\phi^{m+1}\|_{L^p(K,\nu)} \le \frac{ \left(k(k-1)\right)^{1/p}}{k} \omega_p(\phi, m).
$$

For any integer $M>m$, we have
\begin{equation}\label{bound-dyadic+}
\begin{split}
\|\phi^m-\phi^M\|_{L^p(K,\nu)}
&\le
\sum_{j=m}^{M-1}
\|\phi^j-\phi^{j+1}\|_{L^p(K,\nu)}\\
&\le
\frac{\left(k(k-1)\right)^{1/p}}{k}
\|\phi\|_{\Lip(L^p(K,\nu))}
\frac{\lambda^{\alpha m}}{1-\lambda^\alpha}.
\end{split}
\end{equation}

Since $\phi^M\to\phi$ in $L^p(K,\nu)$ as $M\to\infty$,
passing to the limit in \eqref{bound-dyadic+} yields \eqref{Lp-rate}.

\end{proof}

\section{Numerical example}\label{sec.example}
\setcounter{equation}{0}

In this section, we illustrate the convergence analysis in the previous section with a numerical
example. Consider the following initial value problem for a nonlocal heat equation on
the Sierpinski gasket $G$ (cf.~Example~\ref{ex.fractals}\,(2)):
\begin{align}
  \label{nKM}
  \partial_t u(t,x)
  &= \int_K W(x,y)\bigl(u(t,y)-u(t,x)\bigr)\,d\nu(y),\\
  \label{nKM-ic}
  u(0,x)
  &=
  \begin{cases}
    1,  & x\in f_2(G),\\
    -1, & x\in f_1(G)\cup f_3(G),
  \end{cases}
\end{align}
where
\[
  W(x,y)=\exp\{-2|x-y|^2\},
\]
and $K$ is the Sierpinski gasket.

Next, we discretize \eqref{nKM}--\eqref{nKM-ic} as described in
\S~\ref{sec.Galerkin}. Specifically, we approximate $u(t,x)$ by the piecewise constant
Galerkin approximation $u^m(t,x)$ in \eqref{sum}, which yields the finite-dimensional
system \eqref{re-KM-ss}. We integrate this system using the 
fourth-order Runge--Kutta method with time step $\Delta t=10^{-3}$ up to $T=0.1$.
The temporal discretization error can be analyzed by standard methods for numerical
integration of ODEs (see \cite[Section~6]{KVMed22}). The small time step is chosen so
that the temporal discretization error is negligible compared with the spatial
Galerkin error studied below.

To visualize the solution more conveniently, we use the measure-preserving
representation of $K$ on the unit interval developed in the proof of
Theorem~\ref{thm.SGisomorph}.
Namely, for
$x=\pi_K(w)$, where $w=(w_1,w_2,\dots)\in\Sigma$, we set
\[
  \phi(x)=\sum_{i=1}^\infty\frac{w_i-1}{3^i}.
\]
Up to the null set on which the symbolic representation is not unique, $h$ defines a
measure-preserving map from $(K,\mu)$ to $[0,1]$ equipped with Lebesgue measure.
Figure~\ref{f.numerics}a shows $u^m(t,\phi^{-1}(\bar x))$ as a function of
$\bar x\in[0,1]$ at $t=0.1$ and $m=6$. In Figure~\ref{f.numerics}b, we compare the
numerical approximations obtained at three consecutive levels of spatial discretization.

\begin{figure}
  \centering
  \textbf{a}\;
  \includegraphics[width=.35\textwidth]{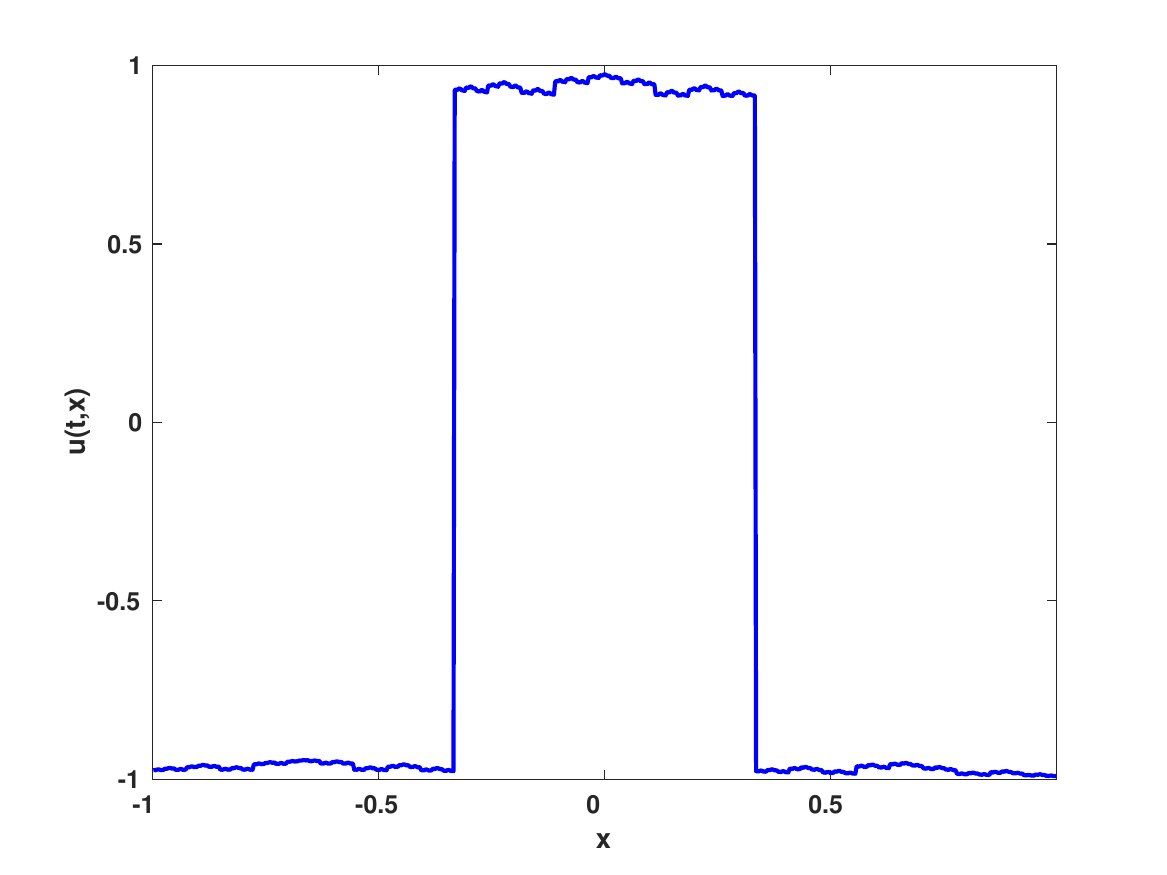}\quad
  \textbf{b}\;
  \includegraphics[width=.35\textwidth]{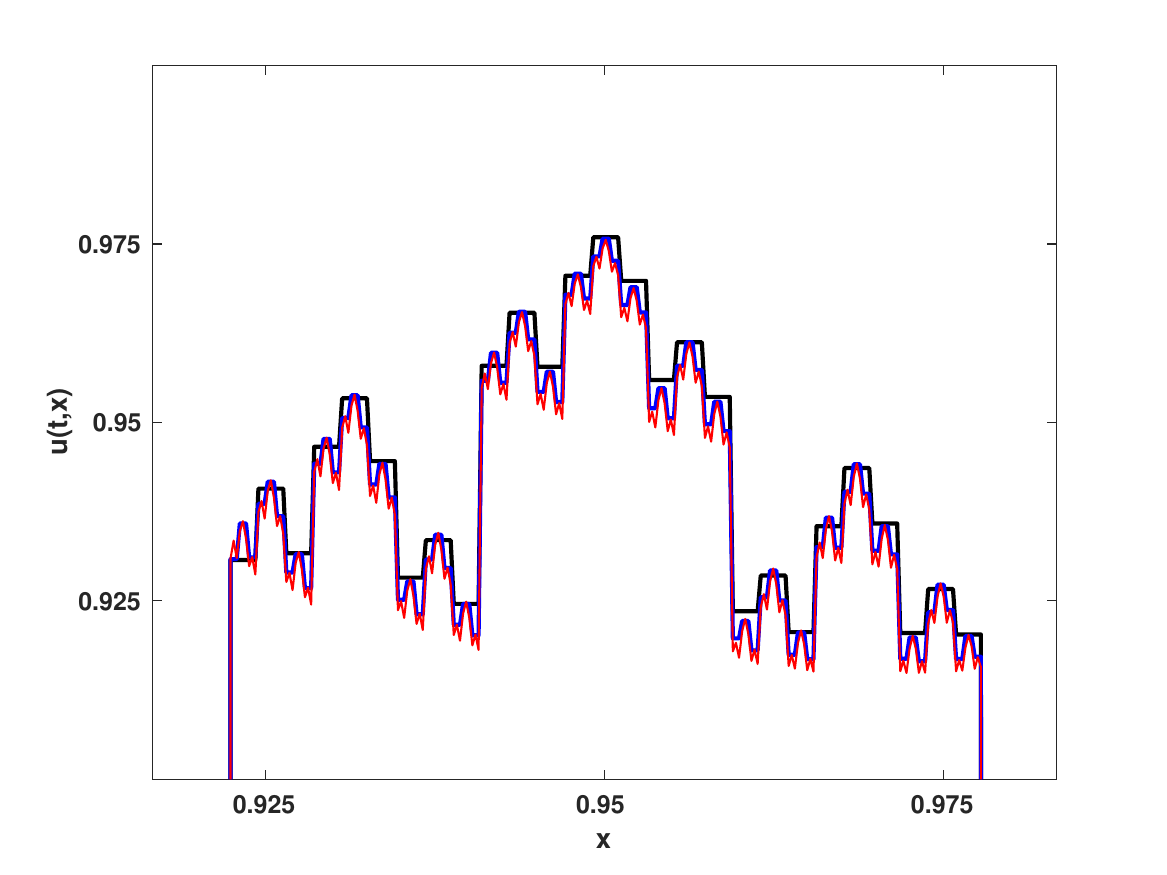}
  \caption{\textbf{a}) Numerical solution of the IVP
    \eqref{nKM}--\eqref{nKM-ic} computed at $t=0.1$ with $m=6$.
    \textbf{b}) Numerical approximations $u^m(0.1,x)$ computed with
    $m=4$ (black), $m=5$ (blue), and $m=6$ (red).}
  \label{f.numerics}
\end{figure}

We next use the numerical solutions to test the convergence rate predicted by the
analysis in \S~\ref{sec.rate}. Since the exact solution of \eqref{nKM} is not
available, we compare Galerkin approximations at successive levels. Define
\[
  \Delta^m
  \doteq
  \bigl\|u^{m+1}(t,\cdot)-u^m(t,\cdot)\bigr\|_{L^2(K,\nu)}.
\]
From \eqref{Lp-rate}, we expect
\[
  \Delta^m\sim C\lambda^{\alpha m},
\]
where $\lambda$ is the contraction constant of the IFS. For the Sierpinski gasket, $\lambda=2^{-1}$.
From the two consecutive values $\Delta^m$ and $\Delta^{m+1}$, we obtain the following estimate of the convergence
exponent:
\begin{equation}\label{est-alpha}
  \alpha_m
  \doteq
  \frac{\log\Delta^{m+1}-\log\Delta^m}{\log\lambda}.
\end{equation}

We used \eqref{est-alpha} to estimate $\alpha$ from numerical solutions at successive
levels of spatial discretization. The results are summarized in
Table~\ref{tab.alpha}.
\begin{table}[H]\label{tab.alpha}
\centering
\caption{Numerical estimates of $\alpha$}

\vspace{0.4 cm}

\begin{tabular}{|c|ccc|}
\hline
$m$        & 3     & 4      & 5     \\
\hline
$\alpha_m$ & 1.002 & 1.015 & 1.007 \\
\hline
\end{tabular}
\end{table}

For the present example, the initial condition \eqref{nKM-ic} belongs to the
piecewise constant Galerkin space already at the first level and therefore introduces
no spatial projection error at subsequent levels. Since the kernel $W$ is smooth, the
convergence estimate in \S~\ref{sec.rate} predicts $\alpha=1$. The numerical values
in Table~\ref{tab.alpha} are consistent with this prediction.

\section{Discussion}\label{sec.discuss}
\setcounter{equation}{0}

The results of this work establish a unified framework for studying IPSs on self-similar
networks and for understanding their macroscopic behavior through continuum and mean-field limits posed
on fractal domains. Our analysis shows that
the large-scale dynamics of IPS defined on self-similar networks admit limiting descriptions that are exactly the
same as in the graphon setting. In particular, we demonstrate an explicit isomorphism between self-similar IPS and
their graphon counterparts. This correspondence can be effectively used to justify the Vlasov equation as the
mean-field limit for self-similar IPS.

Another important contribution of this work is the derivation of the rate of 
convergence estimates for discrete self-similar
models. In contrast to the Euclidean case, where one can use
generalized Lipschitz spaces, which are
well-known in
analysis, the construction of the function spaces in the fractal setting requires careful adaptation of generalized
Lipschitz spaces to functions on fractal domains. The scale of
generalized Lipschitz spaces introduced in this paper, together with the corresponding estimates for $L^2$-projections
onto piecewise constant subspaces,  provides a foundation for the rate of convergence analysis of the Galerkin
method for models on fractal domains. We show that these spaces afford the
minimal regularity required to obtain sharp convergence rates for the
discontinuous
Galerkin approximations of
nonlocal equations on fractal sets. 

Modeling of many physical, biological, and technological systems involves geometric structures that are
not well approximated by Euclidean domains. Hierarchical networks, porous media, and biological tissue
frequently exhibit self-similarity at multiple scales. The approach developed in this paper
suggests a new  class of PDE models on fractals, models that may help to better understand transport and
diffusion in heterogeneous media on the one hand, and  synchronization
and collective dynamics in
self-similar
networks on the other. Moreover, the isomorphism between self-similar and graphon IPS constructed
in this work is used to construct transported PDEs on self-similar fractal domains \cite{MedTre2026}.

Several promising directions stem naturally from the present work. One is the extension of the theory to IPS
with stochastic forcing similar to how it was done for graphon IPS in
\cite{MedSim22}.
Another is the exploration
of nonlocal diffusion on fractal sets. Further,
this work opens a way to studying transition to
synchronization in the Kuramoto model on self-similar networks along the lines of the analysis
of the Kuramoto model on graphs \cite{CMM22a, CMM23}. It is also of interest to explore whether the
isomorphism between graphon and self-similar IPS extends to models on other spaces with partitions
\cite{Kigami-partitions}. Finally,
given the growing interest in data and machine learning, the techniques developed in this paper may
find applications in the analysis of complex networks inferred from data and networks used in training large
language models \cite{YanKulick2025}.

\vskip 0.2cm
\noindent
{\bf Acknowledgements.} 
This work was partially supported by NSF grant DMS 2406941.

\appendix
\section{Integration of functions on self--similar domains}\label{sec.integrate}
  \setcounter{equation}{0}

  The implementation of the Galerkin scheme on self--similar domains
  as described in \S~\ref{sec.Galerkin} requires evaluation of integrals of the form
  $$
  \fint_{K_w\times K_v} W d(\nu\times\nu) \quad \mbox{and}\quad \fint_{K_w} gd\nu \qquad w,v\in\Sigma_m.
  $$

  For the reasons explained in Remark~\ref{rem.product}, it is sufficient to address this problem for
  $\fint_{K_w}\phi d\nu$ with $\phi\in L^1(K,\nu)$. Furthermore, since
  $$ \frac{\nu (\cdot)}{\nu (K_w)} $$
  is a probability measure on $K_w$, without loss of generality we may consider
  \begin{equation}\label{evaluate}
  \int_K \phi d\nu, \qquad \phi\in L^1(K, \nu).
  \end{equation}
  Therefore, in the remainder of this section we will focus on the
  problem of evaluation of \eqref{evaluate}.

  Throughout this section, $K$ is an attractor of a system of contracting similitudes $\{f_i\}_{i\in [k]}$
  equipped with the natural self-similar probability measure $\nu$, i.e.,
  $(K, \{f_i\}_{i\in [k]},\nu)$ is a probabilistic IFS. Further, we assume \eqref{tiny}, which in turn
  implies \eqref{nu-cyl}.

  \subsection{Monte-Carlo method}\label{sec.Monte-Carlo}
  The Bernoulli measure $\mu$ is an ergodic probability measure invariant under $\sigma$, and 
  by the Ergodic Theorem (cf.~\cite[Theorem~6.1]{Falc-Tech}), for $\mu$-a.e. $x\in\Sigma$,
  we have
  \begin{align}\nonumber
\lim_{m\to\infty} \frac{1}{m} \sum_{j=1}^{m-1} \phi\left({\pi_K} (\sigma^j(x))\right) & 
                               =\int_\Sigma \phi\circ{\pi_K} (x) d\mu (x)\\ \label{ergodic}       
                                                                                         &=\int_K \phi (y) d\nu(y),
  \end{align}
  where we used $\nu={\pi_K}_\ast\mu$ in the second line.

  To approximate $\int_K\phi d\nu$ we choose a random $x\in\Sigma$ as follows:
  $$
  x=(x_1, x_2, \dots, x_{2M}, \dots ),\quad (x_1, x_2, \dots, x_{2M})\sim
  \operatorname{Uniform}(\Sigma_{2M}), M\gg 1.
  $$

  Then
  $$
   \frac{1}{M} \sum_{j=0}^{M-1} \phi({\pi_K}\circ\sigma^j(x))
   =\frac{1}{M} \sum_{j=0}^{M-1}\phi({\pi_K}(\xi_j)),
   $$
   where
  $$
  \xi_j:=\sigma^j(x)=(x_{1+j}, x_{2+j}, \dots)\in\Sigma.
  $$
Finally, we approximate
\[
  \phi({\pi_K}(\xi_j))
  \approx \frac{1}{k}\sum_{i=1}^k
  \phi\bigl({\pi_K}(x_{1+j},x_{2+j},\dots,x_{M+j},\dot{i})\bigr),
  \qquad j=0,1,\dots,M-1,
\]
where, as before, $\dot{i}=iii\dots$.

  Thus,
  $$
  \frac{1}{M} \sum_{j=0}^{M-1} \phi({\pi_K}\circ\sigma^j(x))\approx\frac{1}{Mk} \sum_{j=0}^{M-1}
  \sum_{i=1}^k \phi({\pi_K}(x_{1+j}, x_{2+j},\dots,x_{M+j},\dot{i}))=:S_M.
  $$
  For large $M$, $S_M$ approximates $\int_K\phi d\nu$.

\subsection{Quasi-Monte-Carlo method}\label{sec.quasi-Monte-Carlo}

A similar algorithm can be developed using uniformly distributed sequences
(cf.~\cite{InfVol09}). To this end, choose $x_0\in K$ and, for
$w\in\Sigma_m$, define
\[
  x_w=f_w(x_0).
\]
By Theorem~3.1 in \cite{InfVol09}, the resulting point sets
$\{x_w:w\in\Sigma_m\}$ are uniformly distributed with respect to $\nu$.
Consequently, for every $\phi\in C(K)$,
\begin{equation}\label{quasi}
  \lim_{m\to\infty}
  \left|
    \frac{1}{k^m}\sum_{|w|=m}\phi(x_w)
    -\int_K\phi\,d\nu
  \right|
  =0.
\end{equation}

It is natural to ask for quantitative estimates of the convergence rate in
\eqref{quasi}. In related settings, such estimates can be obtained using
Koksma--Hlawka-type inequalities (cf.~\cite{InfVol09}). To our knowledge, a
corresponding estimate for the point sets considered here has not been established.
On the other hand, if $\phi$ is known to be H\"older continuous, then the
approximation error can be estimated effectively, as we show below.

For simplicity, we restrict our discussion to the case where $K$ is the standard
Sierpinski gasket introduced in Example~\ref{ex.fractals}.
Let $T$
be an equilateral triangle with sides of length $1$, and denote
\[
  T_w=f_w(T),\qquad w\in\Sigma_m.
\]
The vertices of $T_w$ are given by
\[
  q_i^w={\pi_K}(w\dot{i}),\qquad i\in[3].
\]
Consider the following quadrature rule:
\begin{equation}\label{approx-int}
  S_m\doteq\frac{1}{3^{m+1}}
  \sum_{|w|=m}\sum_{i=1}^3\phi(q_i^w).
\end{equation}

\begin{lemma}\label{lem.approx}
Suppose that $\phi$ is H\"older continuous on $K$ with exponent
$\alpha\in(0,1]$. Then
\begin{equation}\label{geom-rate}
  \left|S_m-\int_K\phi\,d\nu\right|
  =O(h_m^\alpha),
  \qquad h_m=2^{-m}.
\end{equation}
\end{lemma}

\begin{proof}
Let $C_\phi>0$ be a H\"older constant for $\phi$. Since
\[
  \operatorname{diam}(K_w)=2^{-m}=h_m,
\]
for every $x\in K_w$ and $i\in[3]$ we have
\[
  |\phi(x)-\phi(q_i^w)|
  \le C_\phi h_m^\alpha.
\]
Averaging over the three vertices yields
\[
  \left|
    \frac{1}{3}\sum_{i=1}^3\phi(q_i^w)-\phi(x)
  \right|
  \le C_\phi h_m^\alpha,
  \qquad x\in K_w.
\]
Integrating over $K_w$ and using $\nu(K_w)=3^{-m}$, we obtain
\[
  \left|
    \frac{1}{3^{m+1}}\sum_{i=1}^3\phi(q_i^w)
    -\int_{K_w}\phi\,d\nu
  \right|
  \le
  \frac{C_\phi}{3^m}h_m^\alpha.
\]
Finally, summing over all $w\in\Sigma_m$ and recalling that the sets
$K_w$, $|w|=m$, form a partition of $K$ up to sets of $\nu$-measure
zero, we obtain
\begin{align*}
  \left|S_m-\int_K\phi\,d\nu\right|
  &\le
  \sum_{|w|=m}
  \left|
    \frac{1}{3^{m+1}}\sum_{i=1}^3\phi(q_i^w)
    -\int_{K_w}\phi\,d\nu
  \right| \\
  &\le C_\phi h_m^\alpha.
\end{align*}
This proves \eqref{geom-rate}.
\end{proof}

\bibliographystyle{amsplain}

\begin{thebibliography}{10}

\bibitem{ABD2023}
Manuela Aguiar, Christian Bick, and Ana Dias, \emph{Network dynamics with
  higher-order interactions: coupled cell hypernetworks for identical cells and
  synchrony}, Nonlinearity \textbf{36} (2023), no.~9, 4641--4673 (English).

\bibitem{AleMed2026}
Artem Alexandrov and Georgi~S. Medvedev, \emph{Graphon spin systems as exactly
  solvable models}, Phys. Rev. E \textbf{114} (2026), 044102.

\bibitem{AurCar22}
Alexander Aurell, Ren{\'e} Carmona, G{\"o}k{\c{c}}e Dayan{\i}kl{\i}, and
  Mathieu Lauri{\`e}re, \emph{Finite state graphon games with applications to
  epidemics}, Dyn. Games Appl. \textbf{12} (2022), no.~1, 49--81 (English).

\bibitem{DDS24}
Blanca Ayuso~de Dios, Simone Dovetta, and Laura~V. Spinolo, \emph{On the
  continuum limit of epidemiological models on graphs: convergence and
  approximation results}, Math. Models Methods Appl. Sci. \textbf{34} (2024),
  no.~8, 1483--1532 (English).

\bibitem{BSS-self-similar}
Bal{\'a}zs B{\'a}r{\'a}ny, K{\'a}roly Simon, and Boris Solomyak,
  \emph{Self-similar and self-affine sets and measures}, Math. Surv. Monogr.,
  vol. 276, Providence, RI: American Mathematical Society (AMS), 2023
  (English).

\bibitem{BCN24}
Gianmarco Bet, Fabio Coppini, and Francesca~Romana Nardi, \emph{Weakly
  interacting oscillators on dense random graphs}, J. Appl. Probab. \textbf{61}
  (2024), no.~1, 255--278 (English).

\bibitem{Bick2025}
Christian Bick and Davide Sclosa, \emph{Dynamical systems on graph limits and
  their symmetries}, Journal of Dynamics and Differential Equations \textbf{37}
  (2025), no.~2, 1171--1206.

\bibitem{BP-Fractals}
Christopher~J. Bishop and Yuval Peres, \emph{Fractals in probability and
  analysis}, Camb. Stud. Adv. Math., vol. 162, Cambridge: Cambridge University
  Press, 2017 (English).

\bibitem{Bogachev2}
V.~I. Bogachev, \emph{Measure theory}, vol.~II, Springer, Berlin, 2007.

\bibitem{BKP2019}
Vladimir~I. Bogachev, Egor~D. Kosov, and Svetlana~N. Popova, \emph{A new
  approach to {Nikolskii}-{Besov} classes}, Mosc. Math. J. \textbf{19} (2019),
  no.~4, 619--654 (English).

\bibitem{BCCZ19}
Christian Borgs, Jennifer~T. Chayes, Henry Cohn, and Yufei Zhao, \emph{An
  {$L^p$} theory of sparse graph convergence {I}: {L}imits, sparse random graph
  models, and power law distributions}, Trans. Amer. Math. Soc. \textbf{372}
  (2019), no.~5, 3019--3062. \MR{3988601}

\bibitem{BotPor2025}
Lucas B{\"o}ttcher and Mason~A. Porter, \emph{Dynamical processes on metric
  networks}, SIAM J. Appl. Dyn. Syst. \textbf{24} (2025), no.~4, 2848--2885
  (English).

\bibitem{BMS2004}
Dave Broomhead, James Montaldi, and Nikita Sidorov, \emph{Golden gaskets:
  variations on the {Sierpi{\'n}ski} sieve}, Nonlinearity \textbf{17} (2004),
  no.~4, 1455--1480 (English).

\bibitem{CaeChW24}
A.~M. Caetano, S.~N. Chandler-Wilde, A.~Gibbs, D.~P. Hewett, and A.~Moiola,
  \emph{A {Hausdorff}-measure boundary element method for acoustic scattering
  by fractal screens}, Numer. Math. \textbf{156} (2024), no.~2, 463--532
  (English).

\bibitem{Hew2025}
Ant{\'o}nio~M. Caetano, Simon~N. Chandler-Wilde, Xavier Claeys, Andrew Gibbs,
  David~P. Hewett, and Andrea Moiola, \emph{Integral equation methods for
  acoustic scattering by fractals}, Proceedings of the Royal Society A:
  Mathematical, Physical and Engineering Sciences \textbf{481} (2025),
  no.~2306, 20230650.

\bibitem{CaiHua21}
Peter~E. Caines and Minyi Huang, \emph{Graphon mean field games and their
  equations}, SIAM J. Control Optim. \textbf{59} (2021), no.~6, 4373--4399.
  \MR{4340663}

\bibitem{Cha17}
Sourav Chatterjee, \emph{Large deviations for random graphs}, Lecture Notes in
  Mathematics, vol. 2197, Springer, Cham, 2017, Lecture notes from the 45th
  Probability Summer School held in Saint-Flour, June 2015, \'{E}cole
  d'\'{E}t\'{e} de Probabilit\'{e}s de Saint-Flour. [Saint-Flour Probability
  Summer School]. \MR{3700183}

\bibitem{ChiMed19a}
Hayato Chiba and Georgi~S. Medvedev, \emph{The mean field analysis of the
  {K}uramoto model on graphs {I}. {T}he mean field equation and transition
  point formulas}, Discrete Contin. Dyn. Syst. \textbf{39} (2019), no.~1,
  131--155. \MR{3918168}

\bibitem{ChiMed19b}
\bysame, \emph{The mean field analysis of the {K}uramoto model on graphs {II}.
  {A}symptotic stability of the incoherent state, center manifold reduction,
  and bifurcations}, Discrete Contin. Dyn. Syst. \textbf{39} (2019), no.~7,
  3897--3921. \MR{3960490}

\bibitem{ChiMed22}
\bysame, \emph{Stability and bifurcation of mixing in the {K}uramoto model with
  inertia}, SIAM J. Math. Anal. \textbf{54} (2022), no.~2, 1797--1819.
  \MR{4396969}

\bibitem{CMM18}
Hayato Chiba, Georgi~S. Medvedev, and Matthew~S. Mizuhara, \emph{Bifurcations
  in the {K}uramoto model on graphs}, Chaos \textbf{28} (2018), no.~7, 073109,
  10. \MR{3833337}

\bibitem{CMM2018}
\bysame, \emph{Bifurcations in the {Kuramoto} model on graphs}, Chaos
  \textbf{28} (2018), no.~7, 073109, 10 (English).

\bibitem{CMM22a}
\bysame, \emph{Instability of mixing in the {K}uramoto model: From bifurcations
  to patterns}, Pure and Applied Functional Analysis \textbf{7} (2022),
  1159--1172.

\bibitem{CMM23}
\bysame, \emph{Bifurcations and patterns in the {Kuramoto} model with inertia},
  J. Nonlinear Sci. \textbf{33} (2023), no.~5, 21 (English), Id/No 78.

\bibitem{ChoOme2018}
Teresa Chouzouris, Iryna Omelchenko, Anna Zakharova, Jaroslav Hlinka, Premysl
  Jiruska, and Eckehard Sch{\"o}ll, \emph{Chimera states in brain networks:
  Empirical neural vs. modular fractal connectivity}, Chaos \textbf{28} (2018),
  no.~4, 045112.

\bibitem{Cop-oscill22}
Fabio Coppini, \emph{Long time dynamics for interacting oscillators on graphs},
  Ann. Appl. Probab. \textbf{32} (2022), no.~1, 360--391 (English).

\bibitem{Cop22}
\bysame, \emph{A note on {Fokker}-{Planck} equations and graphons}, J. Stat.
  Phys. \textbf{187} (2022), no.~2, 12 (English), Id/No 15.

\bibitem{CDG20}
Fabio Coppini, Helge Dietert, and Giambattista Giacomin, \emph{A law of large
  numbers and large deviations for interacting diffusions on {E}rdos--{R}enyi
  graphs}, Stoch. Dyn. \textbf{20} (2020), no.~2, 2050010, 19. \MR{4080158}

\bibitem{CosSha21}
Cl{\'e}ment Cosco and Assaf Shapira, \emph{Topologically induced metastability
  in a periodic {XY} chain}, J. Math. Phys. \textbf{62} (2021), no.~4, 15
  (English), Id/No 043301.

\bibitem{CucSma2007}
Felipe Cucker and Steve Smale, \emph{Emergent behavior in flocks}, IEEE Trans.
  Autom. Control \textbf{52} (2007), no.~5, 852--862 (English).

\bibitem{DalKrejn}
Yu.~L. Daletskij and M.~G. Krejn, \emph{Stability of solutions of differential
  equations in {Banach} space. {Translated} from the {Russian} by {S}.
  {Smith}}, Transl. Math. Monogr., vol.~43, American Mathematical Society
  (AMS), Providence, RI, 1974 (English).

\bibitem{DeVore-book}
Ronald~A. DeVore and George~G. Lorentz, \emph{Constructive approximation},
  Grundlehren der Mathematischen Wissenschaften [Fundamental Principles of
  Mathematical Sciences], vol. 303, Springer-Verlag, Berlin, 1993. \MR{1261635}

\bibitem{Dob79}
R.~L. Dobrushin, \emph{Vlasov equations}, Funct. Anal. Appl. \textbf{13}
  (1979), 115--123 (English).

\bibitem{DorBul12}
Florian D{\"o}rfler and Francesco Bullo, \emph{Synchronization and transient
  stability in power networks and nonuniform {Kuramoto} oscillators}, SIAM J.
  Control Optim. \textbf{50} (2012), no.~3, 1616--1642 (English).

\bibitem{Dud02}
R.~M. Dudley, \emph{Real analysis and probability}, Cambridge Studies in
  Advanced Mathematics, vol.~74, Cambridge University Press, Cambridge, 2002,
  Revised reprint of the 1989 original. \MR{1932358}

\bibitem{DupMed22}
Paul Dupuis and Georgi~S. Medvedev, \emph{The large deviation principle for
  interacting dynamical systems on random graphs}, Comm. Math. Phys.
  \textbf{390} (2022), no.~2, 545--575. \MR{4384715}

\bibitem{Falc1999}
K.~J. Falconer, \emph{Semilinear {PDEs} on self-similar fractals}, Commun.
  Math. Phys. \textbf{206} (1999), no.~1, 235--245 (English).

\bibitem{Falc-Tech}
Kenneth Falconer, \emph{Techniques in fractal geometry}, John Wiley \& Sons,
  Ltd., Chichester, 1997. \MR{1449135}

\bibitem{Falc-FracGeom}
\bysame, \emph{Fractal geometry}, third ed., John Wiley \& Sons, Ltd.,
  Chichester, 2014, Mathematical foundations and applications. \MR{3236784}

\bibitem{FalcHu2001}
Kenneth~J. Falconer and Jiaxin Hu, \emph{Nonlinear diffusion equations on
  unbounded fractal domains}, J. Math. Anal. Appl. \textbf{256} (2001), no.~2,
  606--624 (English).

\bibitem{GaoCai20}
Shuang Gao and Peter~E. Caines, \emph{Graphon control of large-scale networks
  of linear systems}, IEEE Trans. Automat. Control \textbf{65} (2020), no.~10,
  4090--4105. \MR{4159107}

\bibitem{GarciaLarina2014}
Monica~D. Garcia and Irina~V. Larina, \emph{Vascular development and
  hemodynamic force in the mouse yolk sac}, Frontiers in Physiology \textbf{5}
  (2014), 308.

\bibitem{GJK22}
Marios~Antonios Gkogkas, Benjamin J{\"u}ttner, Christian Kuehn, and
  Erik~Andreas Martens, \emph{Graphop mean-field limits and synchronization for
  the stochastic {Kuramoto} model}, Chaos \textbf{32} (2022), no.~11, 13
  (English), Id/No 113120.

\bibitem{GkoKue22}
Marios~Antonios Gkogkas and Christian Kuehn, \emph{Graphop mean-field limits
  for {Kuramoto}-type models}, SIAM J. Appl. Dyn. Syst. \textbf{21} (2022),
  no.~1, 248--283 (English).

\bibitem{Gol16}
Fran{\c{c}}ois Golse, \emph{On the dynamics of large particle systems in the
  mean field limit}, Macroscopic and large scale phenomena: coarse graining,
  mean field limits and ergodicity, Lect. Notes Appl. Math. Mech., vol.~3,
  Springer, [Cham], 2016, pp.~1--144. \MR{3468297}

\bibitem{HFE20}
Yosra Hafiene, Jalal~M. Fadili, Christophe Chesneau, and Abderrahim Elmoataz,
  \emph{Continuum limit of the nonlocal {{\(p\)}}-{Laplacian} evolution problem
  on random inhomogeneous graphs}, ESAIM, Math. Model. Numer. Anal. \textbf{54}
  (2020), no.~2, 565--589 (English).

\bibitem{Harris}
K.D. Harris and T.D. Mrsic-Flogel, \emph{Cortical connectivity and sensory
  coding}, Nature \textbf{503(7474)} (2013).

\bibitem{HegKra2005}
Rainer Hegselmann and Ulrich Krause, \emph{Opinion dynamics driven by various
  ways of averaging}, Comput. Econ. \textbf{25} (2005), no.~4, 381--405
  (English).

\bibitem{HinMei20}
Michael Hinz, Dorina Koch, and Melissa Meinert, \emph{Sobolev spaces and
  calculus of variations on fractals}, Analysis, probability and mathematical
  physics on fractals. Based on the presentations at the 6th conference,
  Cornell University, Ithaca, NY, USA, June 2017, Hackensack, NJ: World
  Scientific, 2020, pp.~419--450 (English).

\bibitem{Hut81}
John~E. Hutchinson, \emph{Fractals and self-similarity}, Indiana Univ. Math. J.
  \textbf{30} (1981), no.~5, 713--747. \MR{625600}

\bibitem{InfVol09}
Maria Infusino and Aljo\v{s}a Vol\v{c}i\v{c}, \emph{Uniform distribution on
  fractals}, Unif. Distrib. Theory \textbf{4} (2009), no.~2, 47--58.
  \MR{2559412}

\bibitem{Jab14}
Pierre-Emmanuel Jabin, \emph{A review of the mean field limits for {V}lasov
  equations}, Kinet. Relat. Models \textbf{7} (2014), no.~4, 661--711.
  \MR{3317577}

\bibitem{JabDat23}
Pierre-Emmanuel Jabin and Datong Zhou, \emph{The mean-field {Limit} of sparse
  networks of integrate and fire neurons}, Preprint, {arXiv}:2309.04046
  [math.{PR}] (2023), 2023.

\bibitem{JOP2017}
Anders Johansson, Anders {\"O}berg, and Mark Pollicott, \emph{Ergodic theory of
  {Kusuoka} measures}, J. Fractal Geom. \textbf{4} (2017), no.~2, 185--214
  (English).

\bibitem{KVMed17}
Dmitry Kaliuzhnyi-Verbovetskyi and Georgi~S. Medvedev, \emph{The semilinear
  heat equation on sparse random graphs}, SIAM J. Math. Anal. \textbf{49}
  (2017), no.~2, 1333--1355. \MR{3634990}

\bibitem{KVMed18}
\bysame, \emph{The {M}ean {F}ield {E}quation for the {K}uramoto {M}odel on
  {G}raph {S}equences with {N}on-{L}ipschitz {L}imit}, SIAM J. Math. Anal.
  \textbf{50} (2018), no.~3, 2441--2465. \MR{3799057}

\bibitem{KVMed22}
\bysame, \emph{Sparse {M}onte {C}arlo method for nonlocal diffusion problems},
  SIAM J. Numer. Anal. \textbf{60} (2022), no.~6, 3001--3028. \MR{4500528}

\bibitem{Kechris-SetTheory}
Alexander~S. Kechris, \emph{Classical descriptive set theory}, Grad. Texts
  Math., vol. 156, Berlin: Springer-Verlag, 1995 (English).

\bibitem{Kig01}
Jun Kigami, \emph{Analysis on fractals}, Cambridge Tracts in Mathematics, vol.
  143, Cambridge University Press, Cambridge, 2001. \MR{1840042}

\bibitem{Kigami-partitions}
\bysame, \emph{Geometry and analysis of metric spaces via weighted partitions},
  Lect. Notes Math., vol. 2265, Cham: Springer, 2020 (English).

\bibitem{KVZ1972}
M.~A. Krasnosel'skii, G.~M. Vainikko, P.~P. Zabreiko, Ya.~B. Rutitskii, and
  V.~Ya. Stetsenko, \emph{Approximate solution of operator equations},
  Wolters-Noordhoff, Groningen, 1972.

\bibitem{KueThr2019}
Christian Kuehn and Sebastian Throm, \emph{Power network dynamics on graphons},
  SIAM J. Appl. Math. \textbf{79} (2019), no.~4, 1271--1292 (English).

\bibitem{KueXu22}
Christian Kuehn and Chuang Xu, \emph{Vlasov equations on digraph measures}, J.
  Differ. Equations \textbf{339} (2022), 261--349 (English).

\bibitem{KueXu25}
\bysame, \emph{Vlasov equations on directed hypergraph measures}, SN Partial
  Differ. Equ. Appl. \textbf{6} (2025), no.~1, 49 (English), Id/No 9.

\bibitem{Kur84}
Y.~Kuramoto, \emph{Cooperative dynamics of oscillator community}, Progress of
  Theor. Physics Supplement (1984), 223--240.

\bibitem{VRR99}
Vito Latora, Andrea Rapisarda, and Stefano Ruffo, \emph{Chaos and statistical
  mechanics in the {Hamiltonian} mean field model.}, Physica D \textbf{131}
  (1999), no.~1-4, 38--54 (English).

\bibitem{LovGraphLim12}
L.~Lov{\'a}sz, \emph{Large networks and graph limits}, AMS, Providence, RI,
  2012.

\bibitem{LovaszAMS}
L\'{a}szl\'{o} Lov\'{a}sz, \emph{Large networks and graph limits}, American
  Mathematical Society Colloquium Publications, vol.~60, American Mathematical
  Society, Providence, RI, 2012. \MR{3012035}

\bibitem{Luc18}
Eric Lu\c{c}on, \emph{Quenched asymptotics for interacting diffusions on
  inhomogeneous random graphs}, Stochastic Process. Appl. \textbf{130} (2020),
  no.~11, 6783--6842. \MR{4158803}

\bibitem{Med14a}
Georgi~S. Medvedev, \emph{The nonlinear heat equation on dense graphs and graph
  limits}, SIAM J. Math. Anal. \textbf{46} (2014), no.~4, 2743--2766.
  \MR{3238494}

\bibitem{Med14b}
\bysame, \emph{The nonlinear heat equation on {{\(W\)}}-random graphs}, Arch.
  Ration. Mech. Anal. \textbf{212} (2014), no.~3, 781--803 (English).

\bibitem{Med14c}
\bysame, \emph{Small-world networks of {K}uramoto oscillators}, Phys. D
  \textbf{266} (2014), 13--22. \MR{3129708}

\bibitem{Med19}
\bysame, \emph{The continuum limit of the {K}uramoto model on sparse random
  graphs}, Communications in Mathematical Sciences \textbf{17} (2019), no.~4,
  883--898.

\bibitem{MedMiz22}
Georgi~S. Medvedev and Matthew~S. Mizuhara, \emph{Chimeras unfolded}, J. Stat.
  Phys. \textbf{186} (2022), no.~3, Paper No. 46, 19. \MR{4381194}

\bibitem{MM2026a}
\bysame, \emph{Kuramoto model on {Sierpinski} gasket. {I}: {Harmonic} maps},
  Stud. Appl. Math. \textbf{156} (2026), no.~5, 27 (English), Id/No e70233.

\bibitem{MM2026b}
\bysame, \emph{The {Kuramoto} model on the {Sierpinski} {Gasket}. {II}:
  {Twisted} states}, J. Nonlinear Sci. \textbf{36} (2026), no.~4, 33 (English),
  Id/No 70.

\bibitem{MMP22}
Georgi~S. Medvedev, Matthew~S. Mizuhara, and Andrew Phillips, \emph{A global
  bifurcation organizing rhythmic activity in a coupled network}, Chaos
  \textbf{32} (2022), no.~8, Paper No. 083116, 13. \MR{4465990}

\bibitem{MedPel24}
Georgi~S. Medvedev and Dmitry~E. Pelinovsky, \emph{Turing bifurcation in the
  {Swift}-{Hohenberg} equation on deterministic and random graphs}, J.
  Nonlinear Sci. \textbf{34} (2024), no.~5, 36 (English), Id/No 88.

\bibitem{MedSim22}
Georgi~S. Medvedev and Gideon Simpson, \emph{A numerical method for a nonlocal
  diffusion equation with additive noise}, Stoch. Partial Differ. Equ., Anal.
  Comput. \textbf{11} (2023), no.~4, 1433--1469 (English).

\bibitem{MedTan15a}
Georgi~S. Medvedev and Xuezhi Tang, \emph{Stability of twisted states in the
  {Kuramoto} model on {Cayley} and random graphs}, J. Nonlinear Sci.
  \textbf{25} (2015), no.~6, 1169--1208 (English).

\bibitem{MedTan18}
\bysame, \emph{The {K}uramoto model on power law graphs: Synchronization and
  contrast states}, Journal of Nonlinear Science (2018).

\bibitem{MedTre2026}
Georgi~S. Medvedev and Emmanuel Tr{\'e}lat, \emph{From {PDEs} on standard domains
  to self-similar particle systems on fractals}, Preprint, {arXiv}:2604.24388
  [math.{AP}] (2026).

\bibitem{Mosco13}
Umberto Mosco, \emph{Analysis and numerics of some fractal boundary value
  problems}, Boll. Unione Mat. Ital. (9) \textbf{6} (2013), no.~1, 53--73
  (English).

\bibitem{Neu78}
H.~Neunzert, \emph{Mathematical investigations on particle - in - cell
  methods}, vol.~9, 1978, pp.~229--254.

\bibitem{Neu84}
\bysame, \emph{An introduction to the nonlinear {B}oltzmann-{V}lasov equation},
  Kinetic theories and the {B}oltzmann equation ({M}ontecatini, 1981), Lecture
  Notes in Math., vol. 1048, Springer, Berlin, 1984, pp.~60--110. \MR{740721
  (87i:82061)}

\bibitem{NijDeV2022}
Eddie Nijholt and Lee DeVille, \emph{Dynamical systems defined on simplicial
  complexes: symmetries, conjugacies, and invariant subspaces}, Chaos
  \textbf{32} (2022), no.~9, 20 (English), Id/No 093131.

\bibitem{Nikol-approximation}
S.~M. Nikolskii, \emph{Approximation of functions of several variables and
  imbedding theorems}, Die Grundlehren der mathematischen Wissenschaften, Band
  205, Springer-Verlag, New York-Heidelberg, 1975, Translated from the Russian
  by John M. Danskin, Jr. \MR{0374877}

\bibitem{RavBar2003}
Erzsebet Ravasz and Albert-Laszlo Barabasi, \emph{Hierarchical organization in
  complex networks}, Physical Review E \textbf{67} (2003), no.~2, 026112.

\bibitem{Shepard}
Gordon~M. Shepherd, \emph{The {S}ynaptic {O}rganization of the {B}rain}, Oxford
  University Press, 2004.

\bibitem{Str06}
Robert~S. Strichartz, \emph{Differential equations on fractals}, Princeton
  University Press, Princeton, NJ, 2006, A tutorial. \MR{2246975}

\bibitem{Stroock-Prob}
Daniel~W. Stroock, \emph{Mathematics of {P}robability}, Grad. Stud. Math., vol.
  149, Providence, RI: American Mathematical Society (AMS), 2013 (English).

\bibitem{Tso}
Gilbert~C.D. Ts{\'o}, D.Y. and T.N. Wiesel, \emph{Relationships between
  horizontal interactions and functional architecture in cat striate cortex as
  revealed by cross-correlation analysis},  \textbf{6(4)} (1986), 1160--70.

\bibitem{YanKulick2025}
Mingsong {Yan}, Charles {Kulick}, and Sui {Tang}, \emph{{On the Convergence and
  Size Transferability of Continuous-depth Graph Neural Networks}}, arXiv
  e-prints (2025), arXiv:2510.03923.

\bibitem{YooJeoBar2002}
Soon-Hyung Yook, Hawoong Jeong, and Albert-Laszlo Barabasi, \emph{Modeling the
  {Internet}'s large-scale topology}, Proceedings of the National Academy of
  Sciences of the United States of America \textbf{99} (2002), no.~21,
  13382--13386.
\end{thebibliography}

\def\cprime{$'$} \def\cprime{$'$} \def\cprime{$'$}
\providecommand{\bysame}{\leavevmode\hbox to3em{\hrulefill}\thinspace}
\providecommand{\MR}{\relax\ifhmode\unskip\space\fi MR }
\providecommand{\MRhref}[2]{%
  \href{http://www.ams.org/mathscinet-getitem?mr=#1}{#2}
}
\providecommand{\href}[2]{#2}

\end{document}